\documentclass[reqno,12pt]{amsart}
\usepackage{graphics,wrapfig, graphicx, mathrsfs,xcolor}

\usepackage{amsthm,amsmath,amssymb,amscd,mathrsfs,xcolor,subfig,tikz,fixdif}
\usepackage{indentfirst}
\usepackage{cite}
\usepackage{bm, enumerate}
\usepackage[dvips]{epsfig}
\usepackage{latexsym}
\usepackage{float}
\usepackage{mathtools}
\usepackage{slashed}

\newtheorem{theorem}{Theorem}[section]
\newtheorem{remark}{Remark}[section]
\newtheorem{definition}{Definition}[section]
\newtheorem{lemma}[theorem]{Lemma}
\newtheorem{prop}[theorem]{Proposition}

\counterwithin{table}{section}
\numberwithin{table}{section}

\numberwithin{equation}{section}
\def\fai{\varphi}
\def\tr{{\rm tr}}
	\def\de{\Delta}
	\def\a{\alpha}
	\def\p{\rho}
	\def\ga{\gamma}
	\def\lam{\lambda}
	\def\ta{\theta}
	\def\les{\lesssim}
	\def\pa{\partial}
	\def\be{\beta}
	\def\da{\delta}
	\def\ep{\epsilon}
	\def\varep{\varepsilon}
	
	\def\fe{\phi}
	\def\dl{\underline{L}}
	\def\sg{\slashed{g}}
	\def\zpi{\prescript{Z}{}{\pi}}
	\def\tpi{\prescript{T}{}{\pi}}
	\def\lpi{\prescript{L}{}{\pi}}
	\def\ypi{\prescript{Y}{}{\pi}}
	\def\ppi{\prescript{P}{}{\pi}}
	\def\supda{L^{\infty}(\Sigma_{t}^{\tilde{\da}})}\def\supu{L^{\infty}(\Sigma_{t}^{u})}
	\def\ltwou{L^{2}(\Sigma_{t}^{u})}
\begin{document}        
	\title[Shock formation for Rotating Shallow Water System]{Shock formation for the $2$D rotating shallow water equations with non-zero vorticity}
\date{}

\author{Zhendong Chen}
\address{School of Mathematical Sciences, Institute of Natural Sciences, Shanghai Jiao Tong University, Shanghai, 200240, China
}	
\email{zwei1218@sjtu.edu.cn}
 
	\author{Chunjing Xie}
\address{School of Mathematical Sciences, Institute of Natural Sciences, Ministry of Education Key Laboratory of Scientific and Engineering Computing, CMA-Shanghai, Shanghai Jiao Tong University, Shanghai, 200240, China
}
\email{cjxie@sjtu.edu.cn}

\maketitle
\begin{abstract}
	In the paper, the shock formation for the two-dimensional rotating shallow water system is established. We construct a large class of initial data which leads to the finite-time blow-up for the solutions. Moreover, the solutions are allowed to have non-zero large vorticity (in derivative sense), even up to the shock. Our results provide the first complete geometric description of the shock formation mechanism to the two-dimensional rotating shallow water system with vorticity. The formation of shock is characterized by the collapse of the characteristic hypersurfaces, where the first-order derivatives of the velocity, the height, and the specific vorticity blow up while the potential vorticity remains Lipschitz continuous. The methods developed in this paper should also be useful in studying the shock formation for the Euler equations with various source terms and a class of quasilinear Klein-Gordon equations in multi-dimensions.
\end{abstract}

\section{Introduction and Main Results}
The two-dimensional rotating shallow water system,
\begin{align}
\begin{split}\label{shallowwater0}
\frac{\pa h}{\pa t}+\frac{\pa (hv^1)}{\pa x_1}+\frac{\pa (hv^2)}{\pa x_2}&=0,\\
\frac{\pa v^1}{\pa t}+v^1\frac{\pa v^1}{\pa x_1}+v^2\frac{\pa v^2}{\pa x_2}-v^2&=-\frac{\pa h}{\pa x_1},\\
\frac{\pa v^2}{\pa t}+v^1\frac{\pa v^1}{\pa x_1}+v^2\frac{\pa v^2}{\pa x_2}+v^1&=-\frac{\pa h}{\pa x_2},
\end{split}
\end{align}
describes the behavior of fluid in the regime of large-scale geophysical fluid motion under the action of Coriolis force,  where $h$ denotes the height of the fluid, $v^1$ and $v^2$ are the velocity in $x_1$ and $x_2$ directions, respectively. One may refer to \cite{pedlosky2013} for more physical background on the rotating shallow water system. 

The major goal of this paper is to show the formation of shocks for rotating shallow water system \eqref{shallowwater0} with a large class of initial data. Before we state our main results, we give a quick review for the literature on the shock formation and development for hyperbolic equations and systems.

\subsection{Review of existing literature}
\subsubsection{Shock formation to the Compressible Euler system}
Clearly, system \eqref{shallowwater0} can be regarded as the compressible Euler system with some additional source terms, which is the prototype of the hyperbolic systems of conservation laws. The notable feature of the Cauchy problem to nonlinear hyperbolic systems is that smooth initial data can lead to singularity formation in finite time. Here we first recall some important progress on shock formation for the compressible Euler system and nonlinear hyperbolic equations. It was Riemann who first studied the nonlinear effects to the $1$D isentropic Euler equations. In Riemann's fundamental work \cite{Riemann1860},
he introduced the Riemann invariants and proved that wave compression leads to the shock waves. The shock means the first-order derivatives of the solution blow up in finite time while the solution itself remains bounded. In \cite{Lax1964}, Lax generalized Riemann's result into the $2\times2$  genuinely nonlinear hyperbolic system in one space dimension. Lax used the Riemann invariants to diagonalize the system and  was able to give sufficient and necessary conditions for the system to admit a shock or not.  
Later, John \cite{doi:10.1002/cpa.3160270307} achieved a remarkable result for the shock formation to general $n\times n$ hyperbolic systems of conservation laws in one space dimension. 
 For more progress on the shock formation for 1D hyperbolic systems, one may refer to \cite{LIU197992,Hormander1997,zhouyi} and references therein.
 
 In the multi-dimensional case, the first general result for the singularity formation to the compressible Euler system in three spatial dimensions was obtained by Sideris in \cite{Sideris1985} for polytropic gases. In particular, he constructed an open set of initial compressed data which leads to the finite lifespan of the solutions  by using dissipative energy estimates. A significant question for the multi-dimensional compressible Euler system after Sideris' result is which quantity blows up in finite time. Alinhac studied the two-dimensional compressible isentropic Euler equations with radial symmetry in \cite{Alinhac1993}. He showed that a large class of small radially symmetric data leading to the finite time blow-up of the solutions and gave a precise estimate for the blow-up time. Later, in a series of works \cite{alinhac1999blowup,alinhac1999blowup1,alinhac2001blowup,alinhac2001blowup2}, Alinhac proved the shock formation to the $2D$ quasilinear wave equations without any symmetry assumptions. 

 A major breakthrough in understanding the shock mechanism for the compressible Euler system in multi-dimensions was first made by Christodoulou and his collaborator in \cite{christodoulou2007,christodoulou2014compressible}. In\cite{christodoulou2014compressible}, Christodoulou and Miao studied the shock formation for the classical, non-relativistic, isentropic compressible Euler equations in three spatial dimensions with initial irrotational data. Starting from the short pulse data, the authors gave a detailed analysis of the solutions near the singularity (shock). The authors introduced a geometric framework equipped with the acoustic coordinates system. The solution is regular in the acoustic coordinates. The shock formation corresponds to the transformation between two coordinates degenerating and collapse of the characteristic null hypersurfaces, which is quantified by the inverse foliation density $\mu$ (see Section 2). The framework in \cite{christodoulou2014compressible} has been an important tool for understanding shock formation for hyperbolic equations in multi-dimensions. Later, Miao and Yu applied the framework in \cite{christodoulou2007,christodoulou2014compressible} to study the shock formation for a class of quasi-linear wave equations with cubic terms in three-dimensional space in \cite{Ontheformationofshocks}. They constructed the explicitly short pulse data which leads to the formation of shock in finite time. For recent progress on shock formation and global existence for multi-dimensional quasi-linear wave equations, one may refer to \cite{yin1,Speck20172,Speck2016,Speck2014} and references therein.
 
 For  multi-dimensional Euler equations with non-zero vorticity, a remarkable result on the shock formation was obtained by Luk and Speck in \cite{luk2018shock}. They developed a framework to estimate the vorticity coupled with fluid variables and proved that for perturbed plane-symmetric data, the solutions form a shock in finite time.
 Later, this result was also generalized to the problem for the 3D full compressible Euler system in \cite{Luk2021TheSO}. 
 
 Recently, in 
 a series of significant works \cite{BSV2Disentropiceuler,BSV3Disentropiceuler,BSV3Dfulleuler,formationofunstableshock}, Buckmaster, Shkoller, and Vicol proved the formation of point shock to the compressible Euler system in multi-dimensional case. In\cite{BSV2Disentropiceuler}, they considered the two-dimensional isentropic Euler system under azimuthal symmetry (but different from $1$D problem), where the smooth initial data has finite energy and nontrivial vorticity. By using the modulated self-similar variables, they showed the point shock forms in finite time with explicitly computable blow-up time and location and obtained that the solutions near the shock formation point are of cusp type. Then, the similar result was generalized to three-dimensional Euler systems without any symmetry condition in \cite{BSV3Disentropiceuler, BSV3Dfulleuler}. More recently, 
 in \cite{merle1,merle2}, Merle, Rapha{\"e}l, Rodnianski, and Szeftel made significant progress to establish another blow-up mechanism of solutions to the compressible Euler and Navier-Stokes equations so-called implosion, which means the velocity and the density blow up in finite time. 
 \subsubsection{Shock development for compressible Euler system} In addition to the shock formation results, the shock development problem considers extending the solutions as weak solutions to the compressible Euler system and constructing a shock surface, across which the solutions jump, after the first blow-up time. In \cite{lebaud}, the first result of constructing a shock wave for $1$D $p$-system was shown by Lebaud when the solution is a simple wave, and was later generalized in \cite{chenshuxing}. Under the spherically symmetric assumption, the construction and development of a shock wave to the 3D compressible Euler system were obtained first in Yin \cite{Yin_2004} and then in \cite{chrisdeve} by Christodoulou and Lisibach through a geometric approach. In \cite{Christodoulou2017TheSD}, without any symmetry conditions, Christodoulou generalized the results in \cite{chrisdeve} to the multi-dimensional compressible Euler system for the irrotational and isentropic solutions. 
 The first result in shock development to the multi-dimensional compressible Euler system with non-zero vorticity and entropy was achieved in \cite{shockdevelopmentBSV}, where the authors considered the 2D compressible Euler system in azimuthal symmetry by generalizing the method in  their former papers \cite{BSV2Disentropiceuler,BSV3Disentropiceuler,BSV3Dfulleuler}. Recently, there were some important works \cite{speckshockdeve,vicoldeve2} involving understanding the singular boundary of the maximal globally hyperbolic development (i.e., the largest spacetime where the classical solution is uniquely  determined by the initial data) to compressible Euler system with vorticity and entropy, which is necessary for solving the shock development problem in multi-dimensional case.  
\subsubsection{Shock formation and global existence to the shallow water system}
There are many studies considering the global existence or singularity formation to the Cauchy problem of system \eqref{shallowwater0}. A key feature of the rotation is that it may help prolong the lifespan of the solutions, see \cite{LIU2004262, shallowcheng} and references therein. Under the assumption of vanishing relative vorticity, Cheng and Xie \cite{ChengXieRSW} proved the global existence of classical solutions to the two-dimensional rotating  shallow water system when the initial data is small. One of the key ingredients in \cite{ChengXieRSW} is that the rotating shallow water system can be written as a quasilinear Klein-Gordon system. Later, in \cite{shallowcheng2}, it was proved  that the solutions of the one-dimensional rotating shallow water system with general initial data could develop singularities in finite time, even though the one-dimensional rotating shallow water system can have global smooth solutions where the initial data have zero relative vorticity and are sufficiently small.  Inspired by the works of Sideris in \cite{Sideris1985} and Rammaha \cite{Rammaha}, the finite time singularity formation to the two-dimensional shallow water system was proved for a general class of initial data \cite{HUANG202345}. Furthermore, under the radial symmetric assumption, it was shown that the singularity formation for some compact supported initial data by using weighted energy estimates \cite{HUANG202345}. It is desirable to give the explicit singularity mechanism for the rotating shallow water system.
 \subsection{Main Results}
 Adapted from the framework in \cite{christodoulou2014compressible} and the framework in \cite{luk2018shock}, we study the formation of a shock to the two-dimensional rotating shallow water system with non-zero vorticity. Furthermore, we construct a class of short pulse data and give the geometric description of the shock mechanism.
 
 One can write \eqref{shallowwater0} as
\begin{align}
Bh&=-h\text{div}v,\label{equationh}\\
Bv^i&=\ep_{ij}v^j-\pa_ih,\label{equationvi1}
\end{align}
where $\ep_{ij}$ denotes the anti-symmetric symbol with $\ep_{12}=1$ and $B=\pa_t+v^i\pa_i$ denotes the material derivative. Let 
\begin{align}
\p=\ln h,\quad \omega=\pa_1v^2-\pa_2v^1,\quad \zeta:=\omega e^{-\p}.
\end{align}
Here $\omega$ and $\zeta$ are said to be the fluid vorticity and  the specific vorticity, respectively. One considers $\omega$ as a vector in $\mathbb{R}^3$ with $\omega=(0,0,\pa_1v^2-\pa_2v^1)$. Taking curl in \eqref{equationvi1} yields
\begin{align*}
B\omega^i&=\ep_{iad}\pa_aBv^d+\ep_{iad}[B,\pa_a]v^d
=\ep_{iad}\ep_{3dj}\pa_av^j+\omega\cdot\nabla v^i-\omega^i(\nabla\cdot v)=-(\da_{3i}+\omega^i)\text{div}v.
\end{align*} Then, $(\p,v^i,\zeta)$ satisfies the following system
\begin{align}
B\p&=-\text{div}v,\label{equationp}\\
Bv^i&=\ep_{ij}v^j-e^{\p}\pa_i\p,\label{equationvi}\\
B\zeta&=-e^{-\p}\text{div}v.\label{equationvorticity}
\end{align}
We also call $\xi:=\zeta+e^{-\p}$ the potential vorticity of the fluid for later use. It follows from \eqref{equationp} and \eqref{equationvorticity} that the potential vorticity satisfies a homogeneous transport equation: \begin{align}
	B\xi=0.
\end{align}
This implies that if initial potential vorticity is zero, then it vanishes for all time. We define the sound speed as $\eta:=\sqrt{h}$.

Instead of studying the equations on the whole space $\mathbb{R}^2$, we consider the system \eqref{equationp}-\eqref{equationvorticity} on $(t,x)\in\mathbb{R}^+\times\Sigma_0$, where $(x_1,x_2)\in\Sigma_0:=\mathbb{R}\times \mathbb{S}^1$ with $\mathbb{S}^1$ as the unit circle. Let $\da>0$ be a suitably small parameter. We equip the initial data for the rotating shallow water system as follows:
 \begin{itemize}
 	\item[(1)] for $x_1\geq1$,  set 
 	\begin{align}
 		(\p,v^i)=(0,0,0);\label{initialdata0}
 	\end{align}
 	\item[(2)] for $1-\da \leq x_1\leq 1$, we set
 	\begin{align}
 		\p=\da\fe_0\left(\frac{1-x_1}{\da},x_2\right),\ v^1=\da\fe_1\left(\frac{1-x_1}{\da},x_2\right),\  v^2=\da^2\fe_2\left(\frac{1-x_1}{\da},x_2\right)\label{initialdata1}
 	\end{align}
 for some $C^{\infty}([0,1]\times \mathbb{S}^1)$ functions $\fe_0(s,\theta),\fe_1(s,\theta)$, and $\fe_2(s,\theta)$. Furthermore, it is required that (see also Lemma \ref{initialdata})
 	\begin{align}
 		|L_{\text{f}}(\p,v^1)|+|L_{\text{f}}\zeta|\les\da,\label{initialdata2}
 	\end{align}
 	where $L_{\text{f}}=\pa_t+(v^1+\eta)\pa_1+v^2\pa_2.$
 \end{itemize}
 The main results of this paper can be stated as follows:
 \begin{theorem}\label{main}
 	For the initial data constructed as in \eqref{initialdata0}-\eqref{initialdata2} (see also Lemma \ref{initialdata}), if the following condition holds:
 	\begin{align}
 		\max \pa_s\fe_1(s,\ta)\geq 1,
 	\end{align}
 	then the solution to the shallow water system \eqref{shallowwater0} will form a shock at $t=T_{\ast}\leq 1$, which can be computed explicitly. Furthermore, $\pa_{\a}\p$ and $\pa_{\a}v^1$ blow up at $T_{\ast}$ for some $\a\in\{0,1\}$ in the sense that 
  \begin{align}
      |\pa_{\a}\p|,|\pa_{\a}v^1|\geq\frac{C}{T_{\ast}-t},\quad \text{as}\ t\to T_{\ast}
  \end{align}
  for some constant $C>0$.  Moreover, $\pa_{\a}\zeta$ blows up at $T_{\ast}$ for some $\a\in\{0,1\}$ due to the discontinuity of height of fluid at blow-up time.
 \end{theorem}
 
 \begin{center}
 	\begin{tikzpicture}
 		\draw (-1,0)--(5,0)--(6,1)--(0,1)--(-1,0);
 		\path[thick] (4.2,3.3) edge [out=40,in=240] (4.9,4.3);
 		\path[blue] (2.8,0) edge[out=35,in=260] (4.2,3.3);
 		\path[blue] (3.8,1) edge[out=60,in=265] (4.9,4.3);
 		\path[red] (3.5,0) edge[out=55,in=270] (4.2,3.3);
 		\path[red] (4.4,1) edge[out=70,in=270] (4.9,4.3);
 		\path[blue] (1.7,0) edge[out=35,in=260] (3.5,3.3);
 		\path[blue] (2.9,1) edge[out=45,in=265] (4.2,4.3);
 		\path[blue] (3.5,3.3) edge[out=45,in=235] (4.2,4.3);
 		\path[blue] (1,0) edge[out=35,in=260] (3.2,3.3);
 		\path[blue,dashed] (2.2,1) edge[out=45,in=265] (3.9,4.3);
 		\path[blue] (3.2,3.3) edge[out=45,in=235] (3.9,4.3);
 		\path[blue] (0.2,0) edge[out=35,in=260] (2.9,3.3);
 		\path[blue,dashed] (1.6,1) edge[out=45,in=265] (3.6,4.3);
 		\path[blue] (2.9,3.3) edge[out=45,in=235] (3.6,4.3);
 		\draw[blue] (4.2,0)--(5.4,3.3);
 		\draw[blue] (5.0,1)--(6.1,4.3);
 		\draw[blue] (5.4,3.3)--(6.1,4.3);
 		\node (mu0) at (4.5,3.6) {};
 		\node (mu02) at (3.8,5) {a shock forms};
 		\draw [->,thick] (mu02) to [out=-70, in=100] (mu0);
 		\node (x1) at (2,-0.3) {$x_1\in[1-\da,1]$};
 		\node (si0) at (5,0.5) {$\Sigma_0$};
 		\node (x2) at (-1,0.7) {$x_2\in\mathbb{T}$};
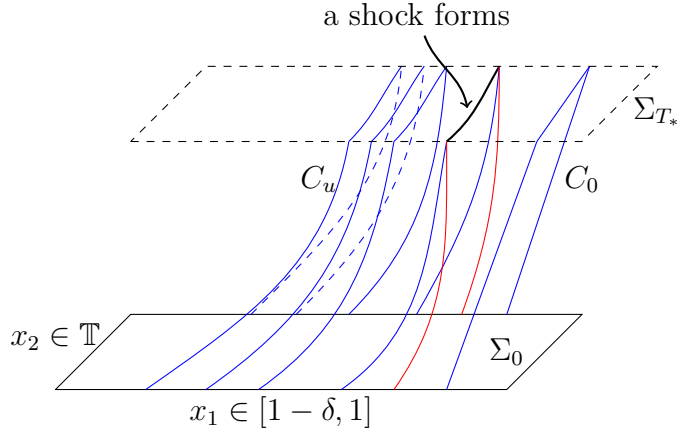
 		\node () at (2,-1) {\textbf{Figure 1: The shock mechanism to the shallow water equations}};
 		\node () at (6,2.8) {$C_0$};
 		\node () at (2.5,2.8) {$C_u$};
 		\draw[dashed] (0,3.3)--(6,3.3)--(7,4.3)--(1,4.3)--(0,3.3);
 		\node () at (7,3.7) {$\Sigma_{T_{\ast}}$};
 	\end{tikzpicture}
 \end{center}
 There are a few remarks in order.
 \begin{remark}
 	In Figure 1, $\Sigma_{t}$ is the level set at time $t$. $C_u$ is the characteristic null hypersurface which is the level set of ekinoal function $u$ (see Section 2) and disjoint initially. As time goes on, $C_u$ will become dense and eventually collapse at some time $t=T_{\ast}$. Note that the initial data outside $x_1=1$ is trivial, which implies $C_0$ is the standard null cone.
 \end{remark}
 \begin{remark}\label{main1}
 	The result can be generalized to the case for
 	\begin{align*}
 		\p=\da^{1+\iota}\fe_0\left(\frac{1-x_1}{\da},x_2\right),\ v^1=\da^{1+\iota}\fe_1\left(\frac{1-x_1}{\da},x_2\right),\ v^2=\da^{2+\iota}\fe_2\left(\frac{1-x_1}{\da},x_2\right),
 	\end{align*}
 	with $\iota\in(-1,1)$. With such initial data, the solution forms a shock at $t=T_{\ast}\leq\frac{1}{\da^{\iota}}+O(\da)$.
 \end{remark}
 \begin{remark}\label{main2}
 	Despite the blow-up of $\pa_{\a}\zeta$  due to the nonlinear interaction, the first derivatives of potential vorticity $\xi$ remain small up to the shock formation. Different from the work \cite{luk2018shock}, where all derivatives of vorticity remain small, we consider the case with large potential vorticity, which allows the higher-order ($\geq2$) derivatives  of $\xi$ to be large.
 \end{remark}
 \begin{remark}\label{main3} ($C^{\frac{1}{3}}$ blow-up)
 	Consider the 1D Burgers equation 
  \begin{align*}
      U_t+UU_x=0
  \end{align*}
  with the self-similar ansatz $U(x,t)=(-t)^{\frac{1}{2}}\bar{U}\left(\frac{x}{(-t)^{\frac{3}{2}}}\right).$ Then, $\bar{U}(y)$ solves the following equation 
 	\begin{align*}
 		-\frac{1}{2}\bar{U}+\left(\frac{3}{2}y+\bar{U}\right)\pa_y\bar{U}=0.
 	\end{align*}
 	$\bar{U}(y)$ is implicitly given as
 	\begin{align}
 		y=-\bar{U}-\bar{U}^3.\label{burgesself}
 	\end{align}
 	Let $\langle y\rangle=(1+y^2)^{\frac{1}{2}}$. Then, $\bar{W}=\langle x_2\rangle\bar{U}(\langle x_2\rangle^3x_1)$ is the solution to the following $2D$ self-similar Burgers equation\footnote{This profile is also used in \cite{BSV3Disentropiceuler} to show the shock formation to 3D Euler system.}:
 	\begin{align*}
 		-\frac{1}{2}\bar{W}+\left(\frac{3}{2}x_1+\bar{W}\right)\pa_{x_1}\bar{W}+\frac{1}{2}x_2\pa_{x_2}\bar{W}=0.
 	\end{align*} 
 	Define 
  \begin{align}	\bar{W}_0(x_1,x_2)=\bar{W}\left(\frac{1-\frac{\da}{2}-x_1}{\da},x_2\right). \label{selfsimilar}
 	\end{align} 	
  Let $\gamma(z)$ be a cut-off function such that 
 	\begin{align*}
 		\gamma(z)&=0 \ \text{for} \ z\geq 1 \ \text{or}\ z\leq 1-\da,\quad \gamma(1-\frac{\da}{2})=1.
 	\end{align*}
 	If one defines
 	\begin{align}
 		\p(0,x_1,x_2)=-\da\ga(x_1)\bar{W}_0, \label{selfsimilar2}
 	\end{align}
 	then, $\p,h,v^1,\zeta\in L^{\infty}([0,T_{\ast}),C^{\frac{1}{3}}).$ The proof of this remark will be given in Lemma \ref{C13}. 
 \end{remark}
  \begin{remark}
 	The initial data constructed in \eqref{initialdata1}-\eqref{initialdata2} is the short pulse data which was first introduced by Christodoulou in \cite{christodoulou2009} to study the formation of black holes in general relativity. This is a class of large data in the sense that along the ``bad'' direction $\pa_{x_1}$, the higher order derivatives will become suitably large.
 \end{remark}
 \begin{remark}
     Note that $v^i$ satisfy some quasilinear Klein-Gordon equations (Lemma \ref{covariant}). Then, the results in this paper can be applied to show the shock formation for a class of quasilinear Klein-Gordon equations with short pulse data.
 \end{remark}
 \begin{remark}
 	Our result can be generalized to the case in \eqref{equationvi} with $\pa_ih$ replaced by $h\pa_i\frac{h^{\ga}}{\ga}$ for any $\ga\neq -1$. For the case $\ga=-1$, the method in this paper together with \cite{yin1} can be used to show the global existence of classical solutions to the rotating shallow water system with similar short pulse data.
 \end{remark}
 \begin{remark}
 It was shown in \cite{CZD1} that shock formation occurs in the irrotational Euler system with damping and the author employed an integrating factor to absorb the damping term into the material derivative. This makes the analysis for the Euler equations with damping quite similar to that without damping. However, for the shallow water system with vorticity, the velocity $v^i$ and the potential vorticity $\zeta$ are coupled in the system. This coupling prevents a direct application of the approach used in the damping case. To this end, we find the system has some good structures (see the following subsections) and use elaborated analysis to obtain the energy estimates for $v^i$ and $\zeta$. Moreover, compared with the result in \cite{CZD1} where the damping effect changes the time of shock formation in an $O(1)$ sense, the results in Theorem \ref{main} show that rotating force only changes the blow-up time in an $O(\da)$ sense. This difference arises because the Coriolis forces, which represent dispersion, are weaker than the dissipation associated with the damping effect.
 \end{remark}
\subsection{Some key ingredients and features in this paper}
In this subsection, we outline the key ingredients and features of this paper and also compare the results with the existing literature.

 In order to outline the key idea of the proof for the main results of this paper, we first give a quick review of the two frameworks to prove the shock formation.
 \subsubsection{Review the main ideas in the existing literature on shock formation}
 In \cite{christodoulou2014compressible,christodoulou2007}, the authors considered the quasilinear wave equation of the form
\begin{align}
    \Box_{g(\fai)}\fai:=\frac{1}{\sqrt{|\det g|}}\pa_{\a}(g^{\a\be}\sqrt{|\det g|}\pa_{\be}\fai)=0\quad \text{in $\mathbb{R}^{3+1}$},
\end{align}
 where $g$ is the acoustical metric (Lorentzian metric), and achieved their results  through the following two steps.
\begin{center}
		\tikzset{global scale/.style={
			scale=#1,
			every node/.append style={scale=#1}
		}
	}
	\begin{tikzpicture}[global scale = 0.9]
		\draw (-3.3,0)--node[below] {$\Sigma_0$}(3.3,0);
		\draw (-3.5,1.5)--node[below] {$\Sigma_t$}(3.5,1.5);
		\draw[blue] (1,0)--(3,3);
		\draw[blue] (2.3,0)--node[above left] {$C_u$}(3.5,3);
		\node at (-2.5,3) {$\mu$ is small};
		\node at (-1.8,0.2) {$\mu\sim 1$};
		\filldraw (2.9,1.5) circle (.04);
		\node at (2.9,1.2) {$S_{t,u}$};
		\node[below] at (-2.3,0) {non-trivial data};
		\node[below] at (2.3,0) {non-trivial data};
		\draw[blue] (3.3,0)--node[above right] {$C_0$}(3.7,3);
		\draw[blue] (-1,0)--(-3,3);
		\draw[blue] (-2.3,0)--node[above right] {$C_u$}(-3.5,3);
		\draw[blue] (-3.3,0)--node[above left] {$C_0$}(-3.7,3);
		\filldraw (3.2,2.25) circle (.05);
		\draw[->] (3.2,2.25)--node[above left] {$L$}(3.5,3);
		\draw[->] (3.2,2.25)--node[above left] {$T$}(2.2,2.25);
		\draw[->] (3.2,2.25)--node[below right] {$X$}(2.4,2);
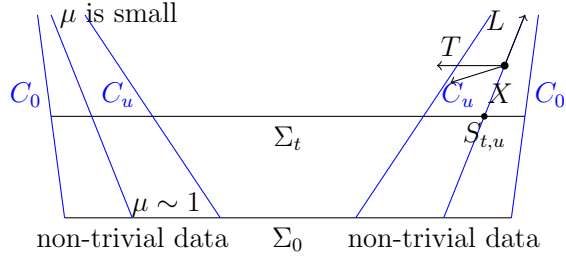
		\node at (-0,-1) {\textbf{Figure2: The acoustical geometry}};
	\end{tikzpicture}
\end{center}
\begin{itemize}
 \item[(1)] The first step is the geometric formulation. Given a solution $\fai$ and the a priori estimates for $\fai$ and its derivatives, they constructed the acoustical coordinates system and related objects as follows:
       \begin{itemize}
       \item[(a)] the acoustic ekional function $u$ defined as $g^{\a\be}\pa_{\a}u\pa_{\be}u=0$ whose level sets are the characteristic null-hypersurfaces $C_u$. The ekional function $u$ forms a component of the acoustical coordinates $(t,u,\ta_1,\ta_2)$ where $(\ta_1,\ta_2)$ are the coordinates on ``spheres" $S_{t,u}$;
       \item[(b)] the inverse foliation density function $\mu$ whose reciprocal measures the density of the hypersurfaces $C_u$. At the blow-up point, $\mu\to 0$ and the characteristic hypersurfaces collapse. 
       \item[(c)] as long as $\mu>0$, the authors could construct the frame $\{L,T, X\}$ which is equivalent to $\{\frac{\pa}{\pa x^{\a}} \}_{\a=0,1,2,3}$, where $L$ represents the null vector field which equals $\frac{\pa}{\pa t}$ in the acoustical coordinates while $T$ and $X_i,i=1,2$ (they are all spacelike) represent the vector fields of ``radial" derivative and the angular derivatives in the acoustical coordinates respectively.
       \end{itemize}
 \item[(2)] The second step is the energy estimates. The authors found a more geometric way to derive energy estimates up to the singularity without derivative loss through the following steps:
     deriving fundamental energy estimates to the wave equation $\Box_g\fai=F$ by using the multiplier method where $F$ represents the general inhomogeneous term. Then, commuting a list of geometric vector fields with $\Box_g\fai=F$ and applying fundamental energy estimates to obtain the top order energy estimates. The most important part of energy estimates is to obtain the following estimates of top-order energy
       \begin{align}
           \int(X^nL\fai)^2+(X^nT\fai)^2, \label{degenerate}
       \end{align}
       which degenerates near shock. The key observation is that along the integral curves of $L$, $\mu$ is monotone near shock ($L\mu<0$), then the behavior of $\mu$ near the singularity can be obtained. Hence, the rate of degeneracy of \eqref{degenerate} could be derived  which requires elaborate analysis of geometric objects so that for some $c>0$, the modified energy
       \begin{align}
           \int\mu^c((X^nL\fai)^2+(X^nT\fai)^2)\label{introenergy}
       \end{align}
       could be controlled. Therefore, one can derive the $\mu$-weighted energy estimates, but this is not sufficient to close the bootstrap assumptions. The authors introduced a decent scheme that by lowering the derivatives of $\fai$,  the weights of $\mu$ will be eliminated eventually. Then, with some low-order but sufficiently large energy estimates, the desired $L^{\infty}$-bounds in bootstrap assumptions is recovered through geometric Sobolev embedding.
 \end{itemize}
 
\subsubsection{Review of main ideas in point shock formation to compressible Euler system} As shown in the breakdown results  \cite{BSV2Disentropiceuler,BSV3Disentropiceuler,BSV3Dfulleuler,formationofunstableshock}, the authors gave another constructive and technical proof of forming a point shock in spacetime to multi-dimensional compressible Euler system without relying on any geometric analysis. In \cite{christodoulou2007,christodoulou2014compressible}, the shock is viewed as a 2-dimensional sub-manifold in the spacetime due to the intersection of characteristic hypersurfaces. The major tools used in \cite{BSV2Disentropiceuler,BSV3Disentropiceuler,BSV3Dfulleuler,formationofunstableshock} consist of the method of the modulation variables and the self-similar coordinates, which were first introduced in \cite{GigaKohnselfsimilar} to study the asymptotic behavior of the solution to the nonlinear heat equation near the point of singularity. Similar to \cite{christodoulou2007,christodoulou2014compressible}, the global existence of the solution in the self-similar coordinates was established and when coming back to the Cartesian coordinates, the transformation degenerates at one point, at which a shock forms. Since the shock forms at a single point, the authors postulated several modulation variables to control the information of the blow-up point, including the blow-up time, location, and direction. They found that these modulation variables satisfy a system of ODEs, which could be solved easily. The second difference is that in closing the bootstrap assumptions, they used the standard Friedrich's energy estimates for the symmetric hyperbolic systems so that the analysis near the singularity, which is the most important part in \cite{christodoulou2007,christodoulou2014compressible}, is not needed. This is mainly due to the fact that the point shock in the Cartesian coordinates corresponds to the time infinity in the self-similar coordinates.

\subsubsection{Key ideas of the proof for the main results in this paper}
Following the frameworks in \cite{christodoulou2014compressible,luk2018shock}, we first reformulate the shallow water system \eqref{shallowwater0} into the following covariant wave equations coupled with the transport equation for vorticity (see Lemma \ref{covariant}):
\begin{align}
    \Box_g\p=F_{\p},\quad \Box_gv^i=F_{v^i},\quad B(\zeta+e^{-\p}-1)=0,\label{covariant1}
\end{align}
where $F_{\p}$ and $F_{v^i}$ represent some inhomogeneous terms. $F_{\p}$ and $F_{v^i}$ include some additional difficult terms, as will be illustrated below. Then, based on the geometric framework in \cite{christodoulou2014compressible}, we derive the fundamental energy estimates for the wave variables and the vorticity (see section \ref{fundamentalenergy}) and in this part, to handle the most difficult terms, we fully use the accurate estimates of $\mu$. Then, we commute the system \eqref{covariant1} with a string of commutation vectorfields to obtain the top order energy estimates as presented in \eqref{introenergy}. This is one of the most crucial parts of the work since we derive the energy estimates without loss of derivatives. Finally, we apply the decent scheme from \cite{christodoulou2014compressible} to eliminate the power of $\mu$ and then recover  the bootstrap assumptions. To demonstrate the formation of a shock in finite time, one notes that in the acoustical coordinates, $\mu$ satisfies the following transport equation\footnote{For Euler system, the nonzero negative coefficient in front of $T\fai$ arise from the exponent $\ga\neq-1$ in the equation of state $p=\frac{1}{\ga}\p^{\ga}$ and in our case, $\ga=2$. Conversely, for $\ga=-1$, the coefficient becomes $0$ and one can show the global existence of the solutions, i.e., $\mu$ is positive for all $t>0$.}
 \begin{align}
L\mu=-T\fai+\text{small terms}.
 \end{align}
 Therefore, following the framework in \cite{Ontheformationofshocks}, we construct the initial data and show that by imposing some conditions on $T\fai|_{t=0}$, $\mu$ will go to $0$ in finite time and a shock forms.
\subsubsection{ Explicit short pulse data}
 \hspace*{2pt}In \cite{luk2018shock}, the authors considered initial data which is launched from simple plane symmetric data. Specifically, in the case of 2D compressible Euler equations, the simple plane symmetric data means
\begin{align}
	\p=\p(x_1),\quad v^1=v^1(x_1),\quad v^2\equiv 0,
\end{align} 
and the Riemann invariant $\mathcal{R}_-:=v^1-\int\p$ completely vanishes. For this data, the vorticity is identically zero. Then, the authors introduced a short pulse  perturbation to such data with non-vanishing vorticity, and proved that  there exists an open set in the Sobolev topology satisfying all the assumptions required for shock formation in their framework. In our paper, inspired by the work of Miao and Yu \cite{Ontheformationofshocks}, we construct an explicit formula for a large class of initial pulse data which can be generated by any smooth, compactly supported ($C^{\infty}_{c}$) functions $f_1$ and $f_2$, as detailed in Lemma \ref{initialdata}. The short pulse data considered here exhibits a broader class in the sense that the choice of $f_1$ and $f_2$ allows for greater flexibility in constructing initial conditions. By  choosing appropriate  $f_1$ and $f_2$, one is able to construct the perturbed simple plane symmetric data explicitly. One important distinction between our work and that of \cite{luk2018shock} lies in the behavior of the vorticity. In \cite{luk2018shock}, the specific vorticity  remains small for all order derivatives, while in our work, the potential vorticity is of $1+O(\da)$ and the higher order derivatives of the potential vorticity will become extremely large. This introduces difficulties in energy estimates and in controlling the higher-order derivatives of the vorticity.

\subsubsection{ The coupled energy estimates}
Because of the presence of vorticity,  besides the framework in \cite{christodoulou2014compressible}, we also used the framework in \cite{luk2018shock}, where the homogeneity of the vorticity equation plays an important role. Roughly speaking, the velocity $v^i$ and the potential vorticity $\xi$ satisfy the following wave-transport system:
\begin{align}
		&\Box_gv^i=\pa_i\zeta+D_i+\zeta Bv^i+\ep_{ij}Bv^j=\pa_i\xi+D_i+\zeta(B+1)v^i,\label{intro2}\\
		& B\zeta=0 \label{intro3}.
\end{align} 
In \eqref{intro2}, $D_i$ ($i=1,2$) are the null forms relative to $g$ which are easy to deal with, and were first introduced in \cite{luk2018shock} (see also definition \ref{nullform}). There are two main difficult terms in \eqref{intro2} to do energy estimates: $\zeta Bv^i$ and $\pa\xi$. The term $\zeta Bv^i$ is caused by the Coriolis force, which may blow up. This is difficult since in top order energy estimates, the blow-up rate of this term is not easy to control. Furthermore, the coefficients of top order energies contributed from this term may be large due to the occurrence of vorticity. This causes difficulties when applying the Gronwall inequality. Fortunately, in view of the construction of the initial data and bootstrap assumptions, the low order energies of the (potential) vorticity is a $O(\da)$ term compared with the energies of $v^i$. Hence, the contribution from this term can be finally handled in the energy estimates.

For $\pa\xi$, one may think $\xi$ as the first derivatives of $v^i$. However, this is not sufficient to close the energy estimates since system \eqref{intro2} loses one derivative. Hence, in energy estimates, it is necessary to consider $\xi$ and $v^i$ to be of the same order. Thanks to the homogeneous equation of $\xi$, one is able to gain one more derivative for $\xi$. Then, we obtains total $N-$th order energy estimates for $v^i$ and $\xi$. From this perspective, the relative vorticity is more regular than $\pa v^i$ as expected.

\subsubsection{ Large vorticity regime}
 Different from the work \cite{luk2018shock}, where the derivatives of  specific vorticity are small along characteristic hypersurfaces up to top order, we allow the potential vorticity to be large in derivative sense (with order $\geq2$). Due to the Coriolis force, the vorticity is enhanced and eventually the first order derivatives of vorticity blow up when a shock forms. To estimate the vorticity, we fully use the equation \eqref{equationvorticity} and the wave equations for the velocity. The basic important fact is that the speed of wave propagation (relative to the sound speed in Euler's case) is strictly faster than the velocity of fluid flow (along particle path). Therefore, the material derivative is a time-like vector field. Hence, it can be decomposed into the combination of the derivative along characteristic null hypersurfaces $L$ and the derivative along time-hyperplane $\Sigma_t$, i.e., $T$. Then, one can write \eqref{intro3} roughly as
\begin{align}
	L\xi=-\mu^{-1}T\xi.\label{intro1}
\end{align}
This implies the spatial derivative of potential vorticity can be changed into the temporal derivative in the acoustical coordinates. Thus, even though the first order spatial derivatives of $v^i$ may blow up as a shock forms, $\pa\xi$ remains small. However, due to the largeness of higher order derivatives of $v^i$, the high order derivatives (order $\geq2$) of $\xi$ are also large. Therefore, it is necessary to compute the blow-up rate for $\xi$ involving high order temporal derivative. This is one of the key steps for the top order energy estimates for the vorticity. Since we construct the explicit initial short pulse data, it is found that the high order derivatives of $\xi$ exhibit one more $\da$ compared with the derivatives of $v^i$ with same order, i.e., $\pa^N\xi=O(\da)\pa^Nv^i$, even though themselves may be large. As we mentioned before, this fact helps one to deal with the possible large coefficient terms caused by the presence of the vorticity in top order energy estimates. On the other hand, \eqref{intro1} is able to show that the potential vorticity remains Lipschitz continuous  even up to the shock, while the first derivatives of the specific vorticity blow up as a shock forms.
\subsubsection{ H$\ddot{\text{o}}$lder continuity for wave variables.}
 In the breakdown results presented in \cite{BSV2Disentropiceuler,BSV3Disentropiceuler,BSV3Dfulleuler,formationofunstableshock}, the solutions near the shock point exhibit cusp-type behavior with $C^{\frac{1}{3}}$-H$\ddot{\text{o}}$lder regularity. In this paper, we adopt the background solution given in \cite{BSV3Disentropiceuler}, which corresponds to a solution of the 2D self-similar Burgers equation, and perturb this solution with a short pulse. For such initial data, we prove, through a different approach, that $(h,v^1,\zeta)\in C^{\frac{1}{3}}$ near the shock. Furthermore, it can be shown that by perturbing various global solutions to 2D self-similar Burgers equation, $(h,v^1,\zeta)\in C^{\frac{1}{2n+1}}$ near the shock for $n\geq 1,$ $n\in\mathbb{Z}$. This demonstrates a broader range of H$\ddot{\text{o}}$lder regularity classes to solutions depending on the structure of perturbation and background solutions, which highlights the sensitivity of the shock structure to initial conditions. Note that in \cite{Luk2021TheSO}, the $C^{\frac{1}{3}}$-H$\ddot{\text{o}}$lder continuous for the wave variables was established by imposing some generic non-degenerated conditions on the initial data.

\subsection{Organization of the paper}
 
 Section \ref{section2} is devoted to the geometric formulation of the problem by adapting the framework from \cite{christodoulou2014compressible}. In Sections \ref{section21}-\ref{scetion23},  the acoustical coordinates system and the explicit initial data are constructed. Then we derive the covariant wave equations for $\fai\in\{\p,v^i\}$. In Sections \ref{section24}-\ref{section25}, we list the structure equations in this paper and compute the deformation tensors.
 
  In Section \ref{section3}, we start with the bootstrap assumptions and then derive the basic estimates for the geometric objects under the bootstrap assumptions. In Section \ref{section32}, the main estimates of the inverse foliation density function $\mu$ are established. They play a crucial role in our analysis. 
  
  In Section \ref{section4}, we first derive the fundamental energy estimates for the wave variables and the vorticity. These estimates are also valid in higher order cases. To this end, in Sections \ref{toporderterms}-\ref{section44}, we obtain the main terms (the top order acoustical terms) that need to be handled in top order estimates and both $L^2$ and $L^{\infty}$ estimates of the lower-order terms.

  In Section \ref{section5},  the estimates of top order acoustical terms involving $\mu$ and $\chi$ are obtained. This is based on the main estimates for $\mu$ in Section \ref{section32}. Then, in Section \ref{section53}, we estimate the various integrals after commutation.
  
 Sections \ref{section6} and \ref{section7} contain the main energy estimates in this paper. In Section \ref{section6},  we obtain the top order energy estimates. They rely on the estimates established in Sections \ref{section3}-\ref{section5}. However, the energies contain the $\mu$-weights and will go to $0$ as $\mu\to 0$. Then, in Section \ref{section7}, we eliminate the weights of $\mu$ in the energy estimates by lowering the order of derivatives in $\fai$. Based on these estimates, we recover the bootstrap assumptions and finish the proof for the main theorem in Section \ref{section8}.

\subsection{Notations}
Throughout the whole paper, the following notations will be used unless stated otherwise:
\begin{itemize}
	\item Latin indices $\{i,j,k,l,\cdots\}$ take the values $1, 2 $ and Greek indices $\{\a,\be,\ga,\cdots\}$
	take the values $ 0, 1, 2$. Repeated indices are meant to be summed.
	\item The convention $f\les h$ means that there exists a universal positive constant $C$ such that $f\leq Ch$.
	\item Let $\fai\in\{\p,v^i\}$ and $\psi\in\{\p,v^i,\zeta+e^{-\p}\}$. Denote $O(\fai)_{b}^{\leq a}$ to be the terms involving $\fai\in\{\p,v^i\}$ of order $\leq a$ with bound $\da^b$. Similarly for $O(\psi)_b^{\leq a}$.  Here, the order means the number of total derivatives acting on $\psi$ and we set $\psi$ to be order of $0$. We also use  the notation $O_b^{\leq a}$ if there is no necessity to distinguish $\fai$ or $\psi$. The notations l.o.ts (lower order terms) mean the terms are of lower order. For example, one can rewrite $\pa^2\psi+\pa\psi$ as $\pa^2\psi+$l.o.ts.   
	\item For the metric $g_{\a\be}$, $g^{\a\be}$ means its inverse such that $g_{\a\be}g^{\be\ga}=\da^{\ga}_{\be}$ with $\da^{\ga}_{\be}$ being the Kronecker symbol.
\item For an object $q$, $\slashed{q}$ means its restriction(projection) on $S_{t,u}$. In particular, $\slashed{\text{div}}$ represents the divergence operator on $S_{t,u}$ such that $\slashed{\text{div}}Y:=\sg^{-1}\slashed{D}Y$ for any $S_{t,u}$ vector field $Y$. $q^{\sharp}$ denotes the $g-$dual of $q$. 
 
	\item The box operator $\Box_g:=g^{\a\be}D^2_{\a\be}$ denotes the covariant wave operator corresponding to the spacetime metric $g$ and $\slashed{\Delta}:=\sg^{-1}\slashed{D}^2$ denotes the covariant Laplacian corresponding to $\sg$ on $S_{t,u}$, where $D$, $\slashed{D}$ are the Levi-Civita connections corresponds to $g$, $\sg$ respectively.
	
	\item For $q$ being a $(0,2)$ tensor and $Y,Z$ being the vector fields, set the contraction as $q_{YZ}:=q_{\a\be}Y^{\a}Z^{\be},$ 
	and similar for the other types of tensors. 
 \item The following important objects involving geometric objects and energies are listed in the following table:
 \begin{table}[htbp]
\setlength{\tabcolsep}{12pt}
\setlength{\belowcaptionskip}{0.2cm}
  \centering
  \caption{Some important objects}\label{table0}
  \begin{tabular}{c|c|c}
  \hline
  Definitions & Notations& \\
  \hline
  Geometric frames & $\{L,T,X\}$ and $\{L,\dl,X\}$ & Lemma\ref{frame}\\
  \hline
  The inverse foliation density & $\mu$ & section 2.1\\
  \hline
  Second fundamental forms & $k,\chi,\ta$ &\eqref{defi2ndff}\\
  \hline
  Curvature tensors & $\a,\a'$ &\eqref{curvaturea}\\
  \hline
  Set of commutation vector-fields & $\mathcal{Z}$, $\mathcal{P}$ &Definition\ref{commutationvf}\\
  \hline
  Energies and fluxes  & $E_{\a},F_{\a}$ $(\a=0,1)$ and $E^{\zeta}, F^{\zeta}$ &Definition \ref{energy} \\
  \hline
  High order energies and fluxes & $\mathbb{W},\mathbb{U},\mathbb{Q},\mathbb{V}$ &Definition \ref{highorderenergy}\\
  \hline
  Modified energies and fluxes & $\bar{\mathbb{W}},\bar{\mathbb{U}},\bar{\mathbb{Q}},\bar{\mathbb{V}}$ &\eqref{modifiedenergy}\\
  \hline
  \end{tabular}
\end{table}\\
\end{itemize}


\section{The Geometric setup}\label{section2}
In this section, we first construct the acoustical geometry and derive the basic properties.
\subsection{The acoustical coordinate system and frames}\label{section21}
One can refer to \cite{christodoulou2014compressible} and \cite{CZD1} for details of the construction of the acoustical coordinate system. The acoustical metric $g$ and its inverse are given by
\begin{align}
	\begin{split}
		g&=-\eta^2dt\otimes dt+(dx^i-v^idt)\otimes(dx^i-v^idt),\\
		g^{-1}&=-\frac{1}{\eta^2}B\otimes B+\pa_i\otimes \pa_i,
	\end{split}
\end{align}
where $\eta^2=h$. The acoustical coordinate system consists of three functions: the time function $t$, the ekinoal function $u$, and  the angular function $\vartheta$.\\
 \hspace*{2pt}On $\Sigma_0$, the ekinoal function $u$ is defined as $u=1-x_1$. Then, $u$ is extended to space-time by the following equation:
 \begin{align}
 g^{\a\be}\pa_{\a}u\pa_{\be}u=0,\quad \pa_tu>0,
 \end{align}
 so that the level sets of $u$ are the characteristic null hypersurfaces $C_u$ of $g$. Since the initial data is completely trivial for $x_1\geq1$, then $C_0$ is the standard null cone. The inverse foliation density function $\mu$ is defined as $\frac{1}{\mu}=-g^{\a\be}\pa_{\a}t\pa_{\be}u$. The null vectorfield $L$, whose integral curves are the null geodesics, on characteristic null hypersurfaces $C_u$ is defined as $L=-\mu g^{\a\be}\pa_{\a}u\pa_{\be}$. Let $S_{t,u}=C_u\cap\Sigma_t$ be an $1$-dimensional manifold. On the tangent space of $\Sigma_t$, define the ``radial" vectorfield $T$ which is $g-$orthogonal to $S_{t,u}$ as $Tu=1$. Denote:
\begin{align}
\begin{split}
\Sigma_t^u&:=\{(t,u',\vartheta)\in\Sigma_t|\ 0\leq u'\leq u\},\\
C_u^t&:=\{(t',u,\vartheta)\in C_u|\ 0\leq t'\leq t\},\\
W_t^u&:=\bigcup_{(t',u')\in[0,t)\times[0,u]} S_{t',u'}.
\end{split}
\end{align} 
Note $S_{0,0}$ is the standard torus which admits a local coordinate $\vartheta$ with $\frac{\pa}{\pa\vartheta}=\pa_2$. One extends  $\vartheta$ to $S_{0,u} $ by requiring $T\vartheta=0$ and then to $W_{t}^{u}$ with $L\vartheta=0$. This procedure results in the acoustical coordinate system $\{t,u,\vartheta\}$.\\
\hspace*{2pt}Let $X=\frac{\pa}{\pa \vartheta}$ be an $S_{t,u}$ tangential vectorfield and $\sg=g(X,X)>0$. We normalize $T$ and $X$ as $\hat{T}=\tfrac{1}{\sqrt{g(T,T)}}T $ and $\hat{X}=\sg^{-\frac{1}{2}}X$. We focus on the case\footnote{This condition is only used in Lemma \ref{faiL2} and \eqref{recover}, which is not important for shock formation. If one further assumes that
\begin{align}
	|\pa^{\a}(\fe_0,\fe_1,\fe_2)|\leq C
\end{align}
for some constant $C$ and $|\a|\leq22$, then this condition can be dropped.} $u\in[0,\tilde{\da}]:=[0,\da^2]$.

For each $u\in[0,\tilde{\da}]$ and fixed $t$, define $\mu(t,u)$ to be the minimum of $\mu$ on the set $S_{t,u}$. Denote
\begin{align}\label{defmumu}
\mu_m^u=\min\{\inf_{u'\in[0,u]}\mu(t,u'),1\},
\end{align}
and
\begin{align}
s_{\ast}=\sup\{t\ |\ t\geq0\ \text{and}\ \mu_m^{\tilde{\da}}(t)>0\}.
\end{align}

 For each $u\in[0,\tilde{\da}]$, one can define $t_{\ast}(u)$ to be the lifespan of the solution $u$ and define $t_{\ast}$ to be:
\begin{align}
t_{\ast}=\inf_{u\in[0,\tilde{\da}]}t_{\ast}(u)=\sup\{\tau|\text{smooth solution exists for all}\ (t,u)\in[0,\tau)\times[0,\tilde{\da}]\}.
\end{align}
We finally restrict time on $[0,t^{\ast})$ with
\begin{align}\label{defsasttast}
s^{\ast}=\min\{s_{\ast},1\},\quad t^{\ast}=\min\{t_{\ast},s^{\ast}\}.
\end{align}
 In the following, we work on the space-time $W_{\tilde{\da}}^{\ast}$, where
\begin{align}
W_{\tilde{\da}}^{\ast}=\bigcup_{(t,u)\in[0,t^{\ast})\times[0,\tilde{\da}]}S_{t,u}.
\end{align}

 Since the lapse function $\eta=(-g^{\a\be}\pa_{\a}t\pa_{\be}t)^{-\frac{1}{2}}$ measures the density of the foliation $\Sigma_t$ (the level set of $t$), then the vectorfield of material derivative $B=-\eta^2g^{\a\be}\pa_{\a}t\pa_{\be}$ is timelike and future-directed. The integral curves of $B$ are the orthogonal curves to the $\Sigma_t$-foliation. Moreover, $Bt=1$.            
\begin{lemma} (\cite[Proposition 2.1]{CZD1})\label{frame}
One has the following relations:
\begin{align*}
g(L,T)&=-\mu, \quad g(T,T)=\kappa^2:=(\eta^{-1}\mu)^2,\quad g(L,\dl)=-2\mu,\\
B&=L+\eta\hat{T},\quad T=\frac{\pa}{\pa u}-\Xi,\quad \hat{T}^i=\ep_{ij}\hat{X}^j,\\
g^{\a\be}&=-\frac{1}{2\mu}(L^{\a}\dl^{\be}+L^{\be}\dl^{\a})+\hat{X}^{\a}\hat{X}^{\be},\quad\frac{\pa(t,x_1,x_2)}{\pa(t,u,\vartheta)}=\eta^{-1}\mu\sg^{\frac{1}{2}},\\
g&=-2\mu dt\otimes du+\kappa^2du\otimes du+\sg(d\vartheta+\Xi^X du)\otimes(d\vartheta+\Xi^X du),
\end{align*}
where $\Xi=\Xi^{X}X$ is an $S_{t,u}$ tangential vectorfield, $\dl=\eta^{-2}\mu L+2T$, and $\{L,\dl,X\}$ is the null frame.
\end{lemma}
\subsection{Initial data}
In this subsection, we construct the explicit short pulse initial data to \eqref{shallowwater0}.
\begin{lemma}\label{initialdata}
For any given $(f_1,f_2)(s,\theta)\in C_0^{\infty}([0,1]\times[0,1))$, there exists $\da'>0$ which depends on $(f_1,f_2)$ such that  for all $\da<\da'$, if one constructs the initial data as follows:
\begin{align}
	\p=\da\fe_0\left(\frac{1-x_1}{\da},x_2\right),\quad v^1=\da\fe_1\left(\frac{1-x_1}{\da},x_2\right),\quad  v^2=\da^2\fe_2\left(\frac{1-x_1}{\da},x_2\right),
\end{align}
then for $L_{\mathrm{f}}:=\pa_t+(v^1+\eta)\pa_1+v^2\pa_2,$
\[
|L_{\mathrm{f}}(\p,v^1)|+|L_{\mathrm{f}}(\zeta+e^{-\p}-1)|\les\da.
\] 
	
\end{lemma}

\begin{proof}
	Let $\hat{T}_{\text{f}}=-\pa_1$. Then, \eqref{equationp}-\eqref{equationvorticity} can be written as
	\begin{align*}
		&(L_{\text{f}}+\eta\hat{T}_{\text{f}})\p=-(\pa_1v^1+\pa_2v^2),\\
		&(L_{\text{f}}+\eta\hat{T}_{\text{f}})v^1=v^2-e^{\p}\pa_1\p,\\
		&(L_{\text{f}}+\eta\hat{T}_{\text{f}})v^2=-v^1-e^{\p}\pa_2\p,\\
		&(L_{\text{f}}+\eta\hat{T}_{\text{f}})\xi=0.
	\end{align*}
	Thus one has
	\begin{align*}
		L_{\text{f}}\p&=\pa_s\fe_1-\eta\pa_s\fe_0-\da^2\pa_{\theta}\fe_2,\\
		L_{\text{f}}v^1&=-\eta\pa_s\fe_1+e^{\da\fe_0}\pa_s\fe_0+\da^2\fe_2,\\
		L_{\text{f}}\xi&=\eta e^{-\da\fe_0}(\pa_s\pa_{\theta}\fe_1+\pa_s\fe_0+\pa_s^2\fe_2-\da\pa_s\fe_0(\pa_s\fe_2+\pa_{\theta}\fe_2)).
	\end{align*}
	Since $|\eta-1|\les\da$, it suffices to show
	\begin{align*}
		\pa_s\fe_1-\pa_s\fe_0&=O(\da),\\
		\pa_s\pa_{\theta}\fe_1+\pa_s\fe_0+\pa_s^2\fe_2&=O(\da).
	\end{align*}
	 Let $\Phi=\pa_{\theta}\fe_1+\fe_1$. $\Phi$ is obtained by solving the following ODE:
	\begin{equation}
		\left\{
		\begin{aligned}	&\pa_s\Phi+\pa_s^2\fe_2=\da f_2,\\
		&	\Phi(0,x_2)=0.
		\end{aligned}
		\right.
	\end{equation}
	With the aid of $\Phi$, one can obtain $(\fe_0,\fe_1)$ via solving the following problems
	\begin{equation}
		\left\{
	    \begin{aligned}
&\pa_{\theta}\fe_1+\fe_1=\Phi,\\
&		\fe_1(0,x_2)=0,\\
	    \end{aligned}
		\right.
		\quad and \quad 
		\left\{
		\begin{aligned}
		&	\pa_s\fe_0-\pa_s\fe_1=\da f_1,\\
		&	\fe_0(0,x_2)=0.
		\end{aligned}
		\right.
	\end{equation}
 This finishes the proof of the lemma.
\end{proof}
\subsection{Wave-transport system for the wave variables}\label{scetion23}
In this subsection, we first derive the covariant wave equations for the wave variables $\fai\in\{\p,v^i\}$. Then, we construct some important geometric objects.
\begin{lemma}\label{waveoperator}
	For any function $f$, it holds that
	\begin{align}
		\Box_gf=-\frac{1}{\eta^2}B^2f+\Delta f-\frac{1}{\eta^2}(\text{div}v)Bf+\frac{1}{2}\eta^{-2}g^{\a\be}\pa_{\a}f\pa_{\be}h+\eta^{-4}BfBh.
	\end{align}
\end{lemma}
\begin{proof}
	Note that\footnote{It can be showed also by $\Box_gf=\frac{1}{\sqrt{|\det g|}}\pa_{\a}(g^{\a\be}\sqrt{|\det g|}\pa_{\be}f)$.} \[\Box_gf=g^{\a\be}D^2_{\a\be}f=g^{\a\be}\pa_{\a}\pa_{\be}f-g^{\a\be}\Gamma_{\a\be}^{\ga}\pa_{\ga}f\] where $\Gamma_{\a\be}^{\ga}$ is the Christoffel symbol associated with $g$, i.e., 
 \[\Gamma_{\a\be}^{\ga}=g^{\ga\da}\Gamma_{\a\be\da}=\frac{1}{2}g^{\ga\da}(\pa_{\a}g_{\be\da}+\pa_{\be}g_{\a\da}-\pa_{\da}g_{\a\be}).\] One computes $\Gamma_{\a\be\ga}$ as
      \begin{align}
      \begin{split}\label{christoffel}
		\Gamma_{000}&=-\eta\pa_t\eta+v^i\pa_tv^i,\quad \Gamma_{0i0}=-\eta\pa_i\eta+v^j\pa_iv^j,\\
		\Gamma_{00i}&=\eta\pa_i\eta-\pa_tv^i-v^j\pa_iv^j,\quad \Gamma_{0ij}=\frac{1}{2}(\pa_jv^i-\pa_iv^j),\\
		\Gamma_{ij0}&=-\frac{1}{2}(\pa_iv^j+\pa_jv^i),\quad \Gamma_{ijk}=\frac{1}{2}(\pa_ig_{jk}+\pa_jg_{ik}-\pa_kg_{ij})=0.
  \end{split}
	\end{align}
	Then,
	\begin{align*}
		g^{\a\be}\Gamma^{\gamma}_{\a\be}\pa_{\ga}f&=-\frac{1}{\eta^3}B\eta\cdot Bf+\frac{1}{\eta^2}(\text{div}v)Bf-\frac{1}{\eta}\pa_i\eta\pa_if+\frac{1}{\eta^2}Bv^i\pa_if.
	\end{align*}
	Therefore,
	\begin{align*}
		\Box_gf= &-\frac{1}{\eta^2}B^2f+\frac{1}{\eta^2}Bv^i\pa_if+
		\Delta f-(-\frac{1}{\eta^3}B\eta Bf+\frac{1}{\eta^2}(\text{div}v)Bf-\frac{1}{\eta}\pa_i\eta\pa_if+\frac{1}{\eta^2}Bv^i\pa_if)\\
		=&-\frac{1}{\eta^2}B^2f+\Delta f-\frac{1}{\eta^2}(\text{div}v)Bf+\frac{1}{2\eta^2}g^{\a\be}\pa_{\a}h\pa_{\be}f+\frac{1}{\eta^4}BfBh.
	\end{align*}
 This finishes the proof of the lemma.
\end{proof}
\quad It follows from \eqref{equationh}-\eqref{equationvi1} that
\begin{align*}
	B^2\p&=-\text{div}(Bv)-[B,\text{div}]v=-\ep_{ij}\pa_iv^j+\pa_iv^j\pa_jv^i+e^{\p}\Delta\p+e^p(\pa_i\p)^2,\\
	B^2v^i&=\ep_{ij}Bv^j-\pa_iBh-[B,\pa_i]h=\ep_{ij}Bv^j+(\text{div}v)\pa_ih+h\pa_i(\pa_lv^l)+\pa_iv^l\pa_lh.
\end{align*}
\begin{definition}\label{nullform}
	The null forms relative to $g$ are defined as
	\begin{align}
		\mathcal{D}(f,g)&=g^{\a\be}\pa_{\a}f\pa_{\be}g,\quad
		and\quad \mathcal{D}_{\a\be}(f,g)=\pa_{\a}f\pa_{\be}g-\pa_{\be}f\pa_{\a}g.
	\end{align}
\end{definition}

The null forms are ``good" terms to handle since their decomposition does not include the term $\hat{T}f\cdot\hat{T}g$, which will blow up like $\frac{1}{\mu^2}$. 
\begin{remark}\label{nullforms}
	Indeed, the following decompositions hold due to Lemma \ref{frame}.
	\begin{align*}
		\mathcal{D}(f,g)&=\eta^{-2}LfLg-\mu^{-1}(LfTg+LgTf)+\sg^{-1}XfXg,\\
		\mathcal{D}_{ab}(f,g)
		&=\hat{T}^a\hat{X}^b(\hat{T}f\hat{X}g-\hat{T}g\hat{X}f)+\hat{X}^a\hat{T}^b(\hat{X}f\hat{T}g-\hat{T}f\hat{X}g)\\
		&=(\hat{T}^a\hat{X}^b-\hat{X}^a\hat{T}^b)(\hat{T}f\hat{X}g-\hat{T}g\hat{X}f).
	\end{align*}
\end{remark}
\begin{lemma}\label{covariant}
	$(v^i,\p,\zeta)$ satisfy the following covariant wave-transport equations
	\begin{align}
		\begin{split}\label{wavep}
			\Box_g\p
			&=2e^{-\p}\mathcal{D}_{12}(v^1,v^2)+\frac{1}{2}\mathcal{D}(\p,\p)+\zeta,\\
		\end{split}\\
		\begin{split}\label{wavevi}
			\Box_gv^i
			&=-\ep_{ia}e^{\p}\pa_a\xi+2\zeta(\ep_{ij}Bv^j+v^i)-\frac{1}{2}\mathcal{D}(v^i,\p)+e^{-\p}v^i+e^{-\p}\ep_{ij}v^jB\p,
		\end{split}\\
		B&(\zeta-\varrho):=B(\zeta+e^{-\p}-1)=0,\label{transportzeta}
	\end{align}
	where $\varrho=1-e^{-\p}$.
\end{lemma}
\begin{proof}
	It follows from Lemma \ref{waveoperator} that 
	\begin{align*}
		\Box_g\p&=-\frac{1}{h}B^2\p+\Delta \p-\frac{1}{h}(\text{div}v)B\p+\frac{1}{2h}g^{\a\be}\pa_{\a}h\pa_{\be}\p+\frac{1}{h^2}BhB\p\\
		&=2e^{-\p}\mathcal{D}_{12}(v^1,v^2)+\zeta-\frac{1}{2}\mathcal{D}(\p,\p).
	\end{align*}
	Note that
	\begin{align*}
		h\pa_i(\text{div}v)&=h(\pa_a(\pa_iv^a-\pa_av^i))+h\Delta v^i=h\ep_{ia}\pa_a\omega+h\Delta v^i,\\
		\pa_iv^j\pa_jh+(\text{div}v)\pa_ih&=\pa_iv^j(\ep_{jk}v^k-Bv^j)-\frac{1}{h}Bh(\ep_{ij}v^j-Bv^i)\\
  &=-\ep_{ij}\omega Bv^j+\mathcal{D}(v^i,h)-\omega v^i-\frac{1}{h}\ep_{ij}v^jBh+\frac{2}{h}BhBv^i.
	\end{align*}
	Therefore, one has
	\begin{align*}
		\Box_gv^i&=-\frac{1}{h}\ep_{ij}Bv^j-\frac{1}{h}(\pa_iv^j\pa_jh+(\text{div}v)\pa_ih)-\pa_i(\text{div}v)+\Delta v^i+\frac{2}{h^2}BhBv^i+\frac{1}{2h}\mathcal{D}(h,v^i)\\
		&=-\ep_{ia}\pa_a\omega+\zeta(\ep_{ij}Bv^j+v^i)-\frac{1}{2h}\mathcal{D}(h,v^i)
		-\frac{1}{h}\ep_{ij}Bv^j+\frac{1}{h}\ep_{ij}v^jBh.
	\end{align*}
 This finishes the proof of the lemma.
\end{proof}
In view of Lemma \ref{frame}, one has the following decomposition for the wave operator.
\begin{align}
\Box_gf=-\mu^{-1}L\dl f-2\mu^{-1}\slashed{g}^{-1}\tau\cdot Xf+\slashed{\Delta}f-\frac{1}{2}\mu^{-1}\tr\underline{\chi}Lf-\frac{1}{2}\mu^{-1}\tr\chi\dl f.\label{decomposition}
\end{align}
\begin{definition} The second fundamental forms are defined as
\begin{align}
2\eta k=\bar{\mathcal{L}}_Bg,\quad 2\chi=\slashed{\mathcal{L}}_L\sg,\quad 2\theta=\slashed{\mathcal{L}}_{\hat{T}}\sg, \label{defi2ndff}
\end{align}
where $\bar{\mathcal{L}}$ and $\slashed{\mathcal{L}}$ are the Lie derivative $\mathcal{L}$ restricted on $\Sigma_t$ and $S_{t,u}$, respectively. They are related as $\chi=\eta(\slashed{k}-\theta)$ with $k_{ij}=\frac{1}{2}\eta^{-1}(\pa_iv^j+\pa_jv^i)$ and $\theta_{XX}=X^1\hat{X}(X^2)-X^2\hat{X}(X^1)$. We also define $\underline{\chi}=\frac{1}{2}\slashed{\mathcal{L}}_{\dl}\sg=\eta^{-2}\mu \chi+2\kappa\ta.$ 
Note $\chi$ is a $(0,2)$-$S_{t,u}$-tangential tensor where $S_{t,u}$ are $1$-dimensional manifolds. We denote $\chi=\chi(X,X)$ and $\underline{\chi}=\underline{\chi}(X,X)$ for later use.
\end{definition}
\begin{remark}
	For any symmetric $(0,2)$-$S_{t,u}$-tangential tensor $\gamma$, it follows that $\gamma=(\tr\gamma)\sg$ with $\tr\gamma=\sg^{-1}\cdot \gamma.$
\end{remark}
\quad Define three $1$-forms on $S_{t,u}$ as follows.
\begin{align}
\varep=k(X,\hat{T}),\quad \tau=g(D_XL,T),\quad \sigma=-g(D_XT,L).
\end{align}
Clearly, it holds that
\begin{align}
\tau=\mu\varep-\kappa X\eta,\quad and\quad \sigma=\tau+X\mu.\label{oneforms}
\end{align}
The connection coefficients are given as
\begin{align}
	\begin{split}\label{connection}
D_LL&=\mu^{-1}(L\mu)L,\quad D_TL=-\eta^{-1}L\kappa L+\sg^{-1}\sigma X,\quad D_LT=-\eta^{-1}L\kappa L-\sg^{-1}\tau X,\\
D_XL&=-\mu^{-1}\tau L+\tr\chi X=D_LX,\\ 
D_TT&=\eta^{-2}\kappa(T\eta+L\kappa)L+(\eta^{-1}L\kappa+\mu^{-1}T\mu)T-\sg^{-1}\kappa X\kappa X,\\
D_XT&=\eta^{-1}\kappa\varep L+\mu^{-1}\sigma T+\kappa \tr\theta X,\quad D_XX=\eta^{-1}\slashed{k} L+\mu^{-1}\chi T+\frac{1}{2}\sg^{-1}X\sg X.
\end{split}
\end{align}
\begin{definition}\label{commutationvf}
The commutation vector-fields sets are defined as follows:
\begin{align}
\mathcal{Z}:=\{L,T,Y\}\quad \text{and}\quad
\mathcal{P}:=\{L,Y\},
\end{align}
where $Y$ is the $g$-orthogonal projection of $\frac{\pa}{\pa x_2}$ to the tangent space of $S_{t,u}$, i.e., 
\[
Y=\pa_2-g(\pa_2,\hat{T})\hat{T}=\sg^{-1}\hat{X}^2\hat{X}.
\]
\end{definition}
In the following, we commute the wave equations with any $Z\in\mathcal{Z}$ and the vorticity equation with any $P\in\mathcal{P}$.
\begin{definition}
	We define the string of the commutation vectorfields $\mathcal{Z}^{N;\leq M}$ and $\mathcal{P}^{N;\leq M}$ with the non-negative integer $N,M$ as follows.
	\begin{itemize}
		\item $\mathcal{Z}^{N;\leq M}$ means an arbitrary $N$ commutation vectorfields $Z\in\mathcal{Z}$, where the number of $T$ is at least $M$;
		\item $\mathcal{P}^{N;\leq M}$ means an  arbitrary $N$ commutation vectorfields $P\in\mathcal{P}$, where the number of $L$  is at least $M$.
	\end{itemize}
	For simplicity, we also denote $Z^{\leq N}$ as a string of arbitrary commutation vectorfields where the number of $Z\in\mathcal{Z}$ is least $N$ and similar for $P^{\leq N}$.
\end{definition}
\begin{lemma}\label{partialderivatives}
	One can express the partial derivatives in the Cartesian coordinates in terms of the frame $\{L, T, X\}$ as follows.
	\begin{align*}
		\pa_0&=L-\eta^{-1}g_{\a 0}L^{\a}\hat{T}+g_{0k}\hat{X}^k\hat{X},
		\quad \pa_i=g_{ik}\hat{T}^k\hat{T}+g_{ik}\hat{X}^k\hat{X}.
	\end{align*}
\end{lemma}
\begin{proof}
Suppose $\pa_0=aL+b\hat{T}+c\hat{X}$ where $a,b$, and $c$ are functions to be determined. It follows from Lemma \ref{frame} that
\begin{align}
	\begin{split}
 -\eta b&=g(\pa_0,L)=g_{\a\be}\da^{\a}_0L^{\be}=g_{0\be}L^{\be},\\
 -\eta a+b&=g(\pa_0,\hat{T})=g_{\a\be}\da^{\a}_0\hat{T}^{\be}=g_{0i}\hat{T}^i,\\
 c&=g(\pa_0,\hat{X})=g_{\a\be}\da^{\a}_0\hat{X}^{\be}=g_{0i}\hat{X}^i.
 \end{split}
\end{align}
These equations show \[a=1,\quad b=\eta^{-1}g_{\a 0}L^{\a}, \quad c=g_{0i}\hat{X}^i.\] Similar proof leads to the expression of $\pa_i$.	 
\end{proof}

The Riemann curvature tensor $\mathcal{R}$ in arbitrary coordinates is defined as
\begin{align*}
  \mathcal{R}_{\a\be\ga\da}&=\mathcal{R}_{\a\be\ga\da}^{(2)}
  +\mathcal{R}_{\a\be\ga\da}^{(1)},
\end{align*}
where 
\begin{align}
	\begin{split}
		\mathcal{R}_{\a\be\ga\da}^{(2)}:&=\frac{1}{2}(\pa^2_{\a\da}g_{\be\ga}+\pa^2_{\be\ga}g_{\a\da}-\pa^2_{\a\ga}g_{\be\da}-\pa^2_{\be\da}g_{\a\ga}),\\
		\mathcal{R}_{\a\be\ga\da}^{(1)}:&=g^{\kappa\lambda}(\Gamma_{\a\da\kappa}\Gamma_{\be\ga\lambda}-\Gamma_{\a\ga\kappa}\Gamma_{\be\da\lambda}).
	\end{split}
\end{align}
In particular, one has
\begin{align}
\begin{split}\label{curvature1}
  \mathcal{R}_{0i0j}^{(2)}
  &=
  \frac{1}{2}(\pa^2_{ij}h-\pa_0(\pa_iv^j+\pa_jv^i)-\pa_{ij}^2\sum(v^k)^2),\\
  \mathcal{R}_{0ijk}^{(2)}
  &=\frac{1}{2}(\pa_{ik}v^j-\pa_{ij}v^k),\quad
  \mathcal{R}_{ijkl}^{(2)}=0.
  \end{split}
\end{align}
Due to Lemma \ref{frame}, the possible singular term in $\mathcal{R}_{\a\be\ga\da}^{(1)}$ is contained in 
\[
\mathcal{R}_{\a\be\ga\da}^{(1),[s]}:=-\frac{1}{4}\mu^{-1}
(Lg_{\a\da}Tg_{\be\ga}-Lg_{\a\ga}Tg_{\be\da})-\frac{1}{4}\mu^{-1}(Tg_{\a\da}Lg_{\be\ga}-Tg_{\a\ga}L_{\be\da}).
\]
One has
\begin{equation}\label{curvature2}
\begin{aligned}
  \mathcal{R}_{0i0j}^{(1),[s]}
  &=-\frac{1}{4}\mu^{-1}(Lv^jTv^i+Tv^jLv^i),\\
  \mathcal{R}_{0ijk}^{(1),[s]}&=0,\quad 
  \mathcal{R}_{ijkl}^{(0),[s]}=0.
\end{aligned}
\end{equation}
\subsection{Some structure equations}\label{section24}
In this subsection, the structure equations of acoustical geometry are obtained. These equations will be used in later analysis.
\begin{lemma}(\cite[Proposition 2.2]{CZD1}) 
$\mu$ and $\chi$ satisfy the following transport equations.
\begin{align}
\begin{split}\label{equationmu}
L\mu&=m+\mu e,
\end{split}
\end{align}
\begin{align}
\begin{split}\label{equationchi}
L\chi&=(\mu^{-1}L\mu)\chi-\a+\chi\cdot\chi=e\chi-\a'+\chi\cdot\chi.
\end{split}
\end{align}
where 
\begin{align}
m&=\frac{3}{2}\eta\hat{T}^i Tv^i=-\frac{3}{2}Th,\quad e=\frac{1}{2}\eta^{-1}\hat{T}^iLv^i+\frac{1}{2}\eta^{-2}Lh-\ep_{ij}\eta^{-1}\hat{T}^iv^j,\\
	\a&=\mathcal{R}(L,X,L,X),\quad \a'=\a-\frac{3}{2}\mu^{-1}\chi\eta\hat{T}^iTv^i.
\end{align}
\end{lemma}
Thus,
\begin{align}
    L\tr\chi&=e\tr\chi-\tr\a'+(\tr\chi)^2.
\end{align}
It follows from \eqref{curvature1} and \eqref{curvature2} that 
$\a$ can be computed as follows.
\begin{align}
\begin{split}\label{curvaturea}
\a&=\a^{[P]}+\a^{[N]}, \quad \a^{[N]}=f(X^i,L^{\a},\hat{T}^i,\fai)( L\fai X\fai+(L\fai)^2+(X\fai)^2),\\
\a^{[P]}&=\frac{1}{2}D_{XX}^2h-\frac{1}{2}D_{XX}^2\sum(v^k)^2+
L^jD_{XX}^2v^j-X^jD_{LX}^2v^j-\frac{1}{2}\mu^{-1}X^iX^jLv^iTv^j,
\end{split}
\end{align}
where $f$ is some smooth function. It follows from \eqref{connection} that the only singular term in $\a^{[P]}$ is $\frac{3}{2}\mu^{-1}\chi\eta\hat{T}^iTv^i$. Hence, $\a'$ is regular as $\mu\to 0$.
\begin{lemma}
	The following structure equations involving $\tr\chi$ and $\hat{T}^i$ will be used.
\begin{align}
\begin{split}\label{equationtchi1}
T\tr\chi&=\sg^{-1}X\sigma+\mu^{-1}\tau\cdot\sg^{-1}\cdot\sigma-\eta^{-1}L(\eta^{-1}\mu)\tr\chi-\eta^{-1}\mu \tr\theta\cdot \tr\chi-\mathcal{R}_{L\hat{X}T\hat{X}},
\end{split}
\end{align}
\begin{align}\label{LTi}
L\hat{T}^i=(-\hat{T}^i\hat{X}v^i+\hat{X}\eta)\hat{X}^i,\quad T\hat{T}^i=-\hat{X}\kappa\hat{X}^i,\quad X\hat{T}^i=\tr\theta X^i.
\end{align}
Note that due to \eqref{curvature1} and \eqref{curvature2}, there is no singular term in $\mathcal{R}_{L\hat{X}T\hat{X}}$ and 
\begin{align}
	\mathcal{R}^{(2)}_{L\hat{X}T\hat{X}}=L^0\hat{X}^iT^j\hat{X}^k\mathcal{R}_{0ijk}^{(2)}=-\frac{1}{2}\hat{X}^i\hat{X}Tv^i+\frac{1}{2}T^i\hat{X}^2v^i+\text{l.o.ts}.
\end{align} Hence, in view of  \eqref{oneforms}, one rewrites \eqref{equationtchi1} as
\begin{align}
	T\tr\chi &-\slashed{\Delta}\mu=\tfrac{1}{2}\hat{X}^1\hat{X}Tv^1+\tfrac{1}{2}\hat{X}^2\hat{X}Tv^2+\mu \hat{X}^2\fai+\text{l.o.ts}.\label{equationtchi}
\end{align}
\end{lemma}
\begin{proof}
	It follows from the definition of Levi-Civita connection and \eqref{defi2ndff} that 
	\begin{align}
		\begin{split}
		T\chi
  &=g(D_XD_TL+D_{[T,X]}L-\mathcal{R}(T,X)L,X)+g(D_XL,D_XT+[T,X])\\
		&=g(D_XD_TL,X)+g(D_XL,D_XT)-\mathcal{R}_{LXTX}+\chi([T,X],X)+\chi(X,[T,X]).
		\end{split}
	\end{align}
	Hence, in view of \eqref{connection}, one has
	\begin{align}
		\begin{split}\label{liechi}
			(\slashed{\mathcal{L}}_T\chi)(X,X)&=T\chi-\chi([T,X],X)-\chi(X,[T,X])\\
			&=X\sigma-\eta^{-1}(L\kappa)X+\mu^{-1}\tau\cdot \sigma+\sg\kappa \tr\theta\cdot \tr\chi-\mathcal{R}_{LXTX}.
		\end{split}
	\end{align}
	Then, \eqref{equationtchi1} follows from \eqref{liechi} and $T\tr\chi=\sg^{-1}\slashed{\mathcal{L}}_T\chi(X,X)-2\kappa \tr\chi\cdot \tr\theta$.
 
 Next, we turn to the proof for \eqref{LTi}. We give the details for $LT^i$ and deriving $T\hat{T}$ and $X\hat{T}$ is similar. Since $T^0=0$, we suppose $L\hat{T}=a\hat{T}+b\hat{X}$ where $a$ and $b$ are functions to be determined. It follows from \eqref{connection} and \eqref{christoffel} that
	\begin{align*}
		a&=g(L\hat{T},\hat{T})=g(D_L\hat{T},\hat{T})-L^{\a}\hat{T}^{\be}\hat{T}^{\ga}\Gamma_{\a\be\ga}=\frac{1}{2}Lg(\hat{T},\hat{T})-\hat{T}^i\hat{T}^j\Gamma_{0ij}=0,\\
		b&=g(L\hat{T},\hat{X})=\kappa^{-1}g(D_LT,\hat{X})-\hat{T}^i\hat{X}^j\Gamma_{0ij}
   =-\hat{T}^i\hat{X}v^i+\hat{X}\eta.
	\end{align*}
This finishes the proof of the lemma.	
\end{proof}
\subsection{Deformation tensor and volume forms}\label{section25}
The deformation tensor associated with a vector filed $Z$ with respect to $g$ is defined as: $\zpi_{\a\be}=D_{\a}Z_{\be}+D_{\be}Z_{\be}$. 
\begin{align}
  [L,T]=\tpi_L^{\sharp},\quad [L,Y]=\ypi_{L}^{\sharp},\quad [T,Y]=\ypi_{T}^{\sharp}.\label{commutator1}
\end{align}
Here $\zpi_{L}^{\sharp}$ denotes the $\sg$-dual of the tensor $\zpi_L$. The deformation tensors of $Z\in\mathcal{Z}$ are given as follows.
\begin{table}[htbp]
\setlength{\tabcolsep}{12pt}
\setlength{\belowcaptionskip}{0.2cm}
  \centering
  \caption{The deformation tensors of the commutation vector fields}\label{table1}
  \begin{tabular}{c|ccc}
  \hline
  $\zpi$ & $\lpi$ &$\tpi$ &$\ypi$\\
  \hline
  $\zpi_{LL}$ & 0 & 0 &$0$\\
  $\zpi_{LT}$ &$-L\mu$ &$-T\mu$ &$-Y\mu$\\
  $\zpi_{L\dl}$ &$-2L\mu$ &$-2T\mu$ &$-2Y\mu$\\
  $\zpi_{TT}$ &$2\kappa L\kappa$ &$2\kappa T\kappa$ &$2\kappa Y\kappa$\\
  $\zpi_{\dl\dl}$ & $4\mu L(\eta^{-2}\mu)$ & $4\mu T(\eta^{-2}\mu)$ &$4\mu Y(\eta^{-2}\mu)$\\
  $\zpi_{LX}$ &0 &$-(\tau+\sigma)$ &$\sg^{\frac{1}{2}}L\hat{T}^1-X^2\tr\chi$\\
  $\zpi_{TX}$ &$\sigma+\tau$ &0 &$\sg^{\frac{1}{2}}T\hat{T}^1-\kappa X^2 \tr\theta$\\
  $\zpi_{\dl X}$ & $2(\sigma+\tau)$ &$-\eta^{-2}\mu(\tau+\sigma)$ & $\sg^{\frac{1}{2}}\dl\hat{T}^1-2\kappa X^2\tr\slashed{k}$\\
  $\zpi_{XX}$ &$2\chi$ &$2\kappa\theta$ & $2\sg^{\frac{1}{2}}X\hat{T}^1$\\
  \hline
  \end{tabular}
\end{table}\\
Note that the area form of $S_{t,u}$ is given as $\sqrt{\sg}d\vartheta$. Then, we define the following area and volume forms.
\begin{definition} 
\begin{align*}
	\int_{S_{t,u}}f&=\int_{\vartheta'\in \mathbb{S}^1}\sqrt{\sg}f\ d\vartheta',\quad\quad\quad\quad \int_{\Sigma_t^u}f=\int_0^t
	\int_{\vartheta'\in \mathbb{S}^1}\sqrt{\sg}f\ d\vartheta'dt',\\
	\int_{C_u^t}f&=\int_0^u\int_{\vartheta'\in \mathbb{S}^1}\sqrt{\sg}f\ d\vartheta'du',\quad \int_{W_t^u}f=
	\int_0^t\int_0^u\int_{\vartheta'\in \mathbb{S}^1}\sqrt{\sg}f\ d\vartheta' du' dt'.
\end{align*}
\end{definition}

The following two elementary lemmas of calculus on the manifold hold, which can be verified directly.
\begin{lemma}\label{change}
	Let $f$ and $g$ be arbitrary functions defined on $S_{t,u}$ and $X$ be an $S_{t,u}$ tangential vector field. It holds that
	\begin{align}
		\int_{S_{t,u}}f(Xg)=-\int_{S_{t,u}}\left\{g(Xf)+\dfrac{1}{2}(\tr\prescript{X}{}{\slashed{\pi}})fg\right\}.
	\end{align}
\end{lemma}
\begin{lemma}\label{integralbypartsl}
	For any function $f$ on $S_{t,u}$, it holds that
	\begin{equation}\label{patintegralbyparts}
	\frac{\pa}{\pa t}\int_{S_{t,u}}f=\int_{S_{t,u}}(L+\tr\chi)f \quad \text{and}\quad 
	\frac{\pa}{\pa u}\int_{S_{t,u}}f=\int_{S_{t,u}}(T+\kappa \tr\theta)f.
	\end{equation}
\end{lemma}

\section{Bootstrap assumptions and the main estimates of $\mu$}\label{section3}
In this section, we first state the bootstrap assumptions in this paper. Based on these assumptions, we then derive some preliminary estimates.
\subsection{Preliminary estimates}
\quad We assume that the following bootstrap assumptions hold for all $t\in[0,t^{\ast})$: for $0\leq|\a|\leq N_{\infty}:=[\dfrac{N_{top}}{2}]+3$,
\begin{equation}\label{bs}
\da^m\|Y^nT^mL\fai\|_{\supda}+\da^{(p-1)_+}\|Y^lL^p(\zeta-\varrho)\|_{\supda}\les\da M,
\end{equation}
for all $n+m+1=l+p=|\a|$ with $m\leq 3$ and $p\leq 3$. In this paper, we take $N_{top}=20$ so that $N_{\infty}=13$.
\begin{lemma}\label{preliminary1}
Under the bootstrap assumptions, the following estimates hold for sufficiently small $\delta$:
\begin{align}
\begin{split}
\|\eta-1\|_{\supu}+\da^{-1}\|m\|_{\supu}+\|e\|_{\supu}\les\da M,\\
\|\mu\|_{\supu}+\|L\mu\|_{\supu}+\da^{-1}\|T\mu\|_{\supu}+\|X\mu\|_{\supu}\les M.
\end{split}
\end{align}
\end{lemma}
\begin{proof}
The estimates for $\eta,m$, and $e$ follow directly from the definition of $\eta$, \eqref{equationmu} and the bootstrap assumptions \eqref{bs}. Hence, integrating \eqref{equationmu} along the integral curves of $L$ yields
\begin{align*}
\mu(t,u,\vartheta)&=e^{\int_0^{t}e(\tau)d\tau}(\mu(0,u,\vartheta)+\int_0^te^{\int_0^{\tau}-e(s)ds}m\ d\tau)\les M.
\end{align*}
Commuting $X$ with \eqref{equationmu} and noting that $\da^{-1}|Xm|+|Xe|\les\da M$ lead to
\begin{align*}
X\mu(t,u,\vartheta)&\les X\mu(0,u,\vartheta)+\int_0^tXm+\mu Xe\ d\tau\les M.
\end{align*}
To prove the estimates of $T\mu$, commuting $T$ with \eqref{equationmu} shows 
\begin{align}
LT\mu=Tm+eT\mu+\mu Te+\tpi_L^{\sharp}\mu.\label{equationLTmu}
\end{align}
 Therefore, it follows from $|X\mu|+|\tpi_L^{\sharp}|\les M$ and $\da^{-1}|Tm|+|Te|\les M$ that
 \[
 T\mu(t,u,\vartheta)\les T\mu(0,\mu,\vartheta)+\int_0^t\da^{-1} M\les \da^{-1}M.
 \]
 This finishes the proof of the lemma.
\end{proof}
\begin{lemma}\label{2ndff}
Under the bootstrap assumptions, it holds that for sufficiently small $\da$,
\begin{align} 
\|\slashed{k}\|_{\supu}+\|\chi\|_{\supu}+\|\theta\|_{\supu}\les \da M.
  \end{align}
\end{lemma}
\begin{proof}
It follows from the bootstrap assumptions that
\[
|\slashed{k}|=|\eta^{-1}||\hat{X}^iXv^i|\les\da M.
\]
It remains to show the estimate of $\chi$. Let $P(t)$ be the set of $t'$ such that 
\[
\|\chi\|_{L^{\infty}(\Sigma_{t'}^{\tilde{\da}})}\leq C_0\da M
\]
holds for all $t'\in[0,t]$ with sufficiently small $\da$, where $C_0$ is a constant to be determined later. On $\Sigma_{0}$, $|\chi|=|\pa_2v^2|\les\da,$ which implies $0\in P(t)$. Let $t_0$ be the upper bound of $P(t)$. Hence, it follows from Lemma \ref{preliminary1}, \eqref{curvaturea} and the bootstrap assumptions that 
\[|\chi|+|e|\leq C_1\da M\quad \mathrm{and}\quad |\a'|\leq C_2\da M\]
for some constants $C_1,C_2$. This implies
\begin{align*}
  L\chi\leq C_1\da M\chi+C_2\da M.
\end{align*}
Apply the Gronwall inequality yields 
\[|\chi|\leq C_3\da M+C_2C_3\da M<C_0\da M\]
by choosing $C_0>2C_3(1+C_2)$ and for sufficiently small $\da$. The continuous argument shows $[0,t^{\ast})\subset P(t)$.
\end{proof}

As a consequence, it holds that $\sg=1+O(\da M)$ due to the fact that $L\sg=2\chi$ and $|\sg|=1$ on $\Sigma_0$.
\begin{lemma}\label{LiTiXi}(\textbf{Estimates on the difference between two coordinates})
  Under the bootstrap assumptions, the following estimates hold for sufficiently small $\delta$:
  \begin{align}
  \|(L^1,\hat{T}^1,\hat{X}^2)\|_{\supu}=1+O(\da M),\quad \|(L^2,\hat{T}^2,\hat{X}^1)\|_{\supu}\les\da M.
  \end{align}
\end{lemma}
\begin{proof}
It follows from \eqref{LTi} and the bootstrap assumptions that $|\hat{T}^1|\les 1+\da M$ and $|\hat{T}^2|\les \da M$. The estimates of $L^i$ follow directly from $L^i=v^i-\eta \hat{T}^i$ and the bootstrap assumptions. Hence the proof of the lemma is completed.
\end{proof}
The following table shows the estimates for the deformation tensors which follows from table \ref{table1}, the bootstrap assumptions, Lemma \ref{preliminary1} and \ref{2ndff}.
\begin{table}[htbp]
\setlength{\tabcolsep}{12pt}
\setlength{\belowcaptionskip}{0.2cm}
  \centering
  \caption{The deformation tensors' estimates}\label{table2}
  \begin{tabular}{c|cccc}
  \hline
  $\zpi$ & bounds &$\lpi$ &$\tpi$ &$\ypi$\\
  \hline
  $\zpi_{L\dl}$ &$\leq$&$M$ &$\da^{-1}M$ &$M$\\
  $\mu^{-1}\zpi_{\dl\dl}$ &$\leq$& $M$ & $\da^{-1}M$ &$M$\\
  $\zpi_{LX}$ &$\leq$&/ &$M$ &$\da M$\\
  $\zpi_{\dl X}$ &$\leq$& $M$ &$M$ & $\da M$\\
  $\zpi_{XX}$ &$\leq$&$\da M$ &$\da M$ & $\da M$\\
  \hline
  \end{tabular}
\end{table}

\subsection{Accurate estimates of $\mu$}\label{section32}
In this subsection, we obtain the main estimates of $\mu$. This plays a key role in the energy estimates.
\begin{lemma}\label{accuratemu}
It holds that for sufficiently small $\da$:
\begin{align}
|L\mu(t,u,\vartheta)-L\mu(0,u,\vartheta)|&\les\da M^2,\label{accuratemu1}\\
|\mu(t,u,\vartheta)-1-tL\mu(0,u,\vartheta)|&\les \da M^2.\label{accuratemu2}
\end{align}
\end{lemma}
\begin{proof}
In view of \eqref{decomposition} and \eqref{nullforms}, one can rewrite \eqref{wavep} and \eqref{wavevi} as
\begin{align*}
  L\dl\p+\frac{1}{2}(\tr\chi+L\p)\dl\p&=O(\da M^2),\\
  L\dl v^i+\frac{1}{2}(\tr\chi+L\p)\dl v^i&=O(\da M^2),
\end{align*}
respectively. This implies
\begin{align}
|\dl\fai(t,u,\vartheta)-\dl\fai(0,u,\vartheta)|&\les\da M^2,\label{tfait0}\\
|\fai(t,u,\vartheta)-\fai(0,u,\vartheta)|&\les \da M^2.
\end{align}
Therefore, due to Lemma \ref{preliminary1} and \eqref{equationmu}, one obtains
\begin{align*}
L\mu(t,u,\vartheta)-L\mu(0,u,\vartheta)&=m(t,u,\vartheta)-m(0,u,\vartheta)+O(\da M^2)\\
&=\frac{3}{2}(Tv^i(t,u,\vartheta)-Tv^i(0,u,\vartheta))+O(\da M^2)=O(\da M^2).
\end{align*}
Hence, \eqref{accuratemu2} follows from \eqref{accuratemu1} via the following straightforward computations.
\begin{align*}
  \mu(t,u,\vartheta)&=\mu(0,u,\vartheta)+\int_0^tL\mu(\tau,u,\vartheta)\ d\tau=
  1+tL\mu(0,u,\vartheta)+O(\da M^2).
\end{align*}
This finishes the proof of the lemma.
\end{proof}

\begin{remark}\label{Tv2}
	When proving Lemma \ref{accuratemu}, we can obtain $|T(\p-v^1)(t,u,\vartheta)-T(\p-v^1)(0,u,\vartheta)|+|Tv^2(t,u,\vartheta)-Tv^2(0,u,\vartheta)|\les \da M^2$. Hence, it follows from Lemma \ref{initialdata} that $|T(\p-v^1)(t)|\les \da M^2$. This means although $T\p$ and $Tv^1$ are large, their difference is small even up to the shock.
\end{remark}

Define the shock region as 
\begin{align}
    W_{s}:=\{(t,u,\vartheta)|\ \mu(t,u,\vartheta)\leq \tfrac{1}{10}\}\quad \mathrm{and} \quad t_0:=\inf\{[0,t^{\ast}) |\ \mu_m<\tfrac{1}{10}\}.
\end{align}
The following proposition states the key behaviors of $\mu$ in the shock region.
\begin{prop}\label{keymu}
  For all $(t,u,\vartheta)\in W_{s}$ and sufficient small $\da$, it holds that
  \begin{align}
  L\mu(t,u,\vartheta)&\leq-\frac{1}{4t}\les -1,\label{keymu1}\\
  (\mu^{-1}T\mu)_+&\les\frac{\da^{-1}}{\sqrt{s^{\ast}-t}}.\label{keymu2}
  \end{align}
\end{prop}
\begin{remark}
Due to Lemma \ref{preliminary1}, in the shock region, one obtains $|Tv^i|,\ |Th|\geq c>0$ for some constant $c$. Since $|Tv^2|+|\hat{T}^2|\leq\da M$, this implies
\begin{align}
|\hat{T}^1\pa_1v^1|,\ |\hat{T}^1\pa_1h|\geq\mu^{-1}c.
\end{align}
More precisely, in view of \eqref{keymu1} and \eqref{shocktime}, in the shock region $|Tv^1(t,u,\vartheta)|\geq \frac{1}{6t}-O(\da M^2).$ Thus one has,
\begin{align}
	|\pa_1v^1|,\ |\pa_1h|\geq\frac{\max\pa_s\fe_1}{8}\mu^{-1}.
\end{align}
Moreover, it follows from the proof of Proposition \ref{keymu} that $|\pa_1v^1|,\ |\pa_1h|\geq C\frac{\max\pa_s\fe_1}{8}\frac{1}{s^{\ast}-t}$ as $t\to s^{\ast}$. Then, as a shock forms, $\pa_{\a}v^1,\pa_{\a}h$ blow up for $\a\in\{0,1\}$. Furthermore, it follows from \eqref{transportzeta} that $T\zeta=Th+O(\da M^2)$ and as a result, $\pa_{\a}\zeta$ blows up as a shock forms for $\a\in\{0,1\}$. The blow-up of $\pa_{\a}\zeta$ is caused by the discontinuity of the height $h$ near shock. 
\end{remark}

\begin{remark}\label{shockformation}
It follows from the proof of Lemma \ref{accuratemu} that $$m(t,u,\vartheta)=m(0,u,\vartheta)+O(\da M^2)=-\frac{3}{2}\pa_s\fe_1+O(\da M^2)=-\frac{3}{2}\pa_s\fe_0+O(\da M^2).$$ This implies that $$L\mu(t,u,\vartheta)=m(t,u,\vartheta)+O(\da M^2)=\frac{3}{2}\pa_s\fe_1+O(\da M). $$
Hence, one has
\begin{align}
  \mu(t,u,\vartheta)&=\mu(0,u,\vartheta)+\int_0^t L\mu(\tau,u,\vartheta)\ d\tau=1-\frac{3}{2}t\pa_s\fe_1+O(\da M^2).\label{shocktime}
\end{align}
Therefore, if $\max \pa_s\fe_1\geq 1$ or equivalently $\min\pa_{1}\p(0,x)\leq-1$, then $\mu$ becomes $0$ at $t=T_{\ast}=\frac{2}{3}+O(\da).$ That is, a shock forms at $t=T_{\ast}$. Moreover, if one assumes that $\min\pa_1\p\leq-\da^{\a}$ for $-1<\a<1$ (one may take $\p=\da^{1+\a}\fe_0(\frac{1-x_1}{\da},x_2)$), then similar results in this paper hold with $\da$ replaced by $\da^{1+\a}$, i.e., $|L\fai|+|Y\fai|+\da|T\fai|\les\da^{1+\a}$, and \eqref{shocktime} becomes\footnote{In this case, one assumes $s^{\ast}=\frac{2}{\da^{\a}}$}
\begin{align}
	\mu(t,u,\vartheta)&\geq1-\frac{3}{2}\da^{\a}t+\int_0^tO(\da^{1+\a})=1-\frac{3}{2}\da^{\a}t+O(\da),
\end{align}
which implies the blow-up time $T_{\ast}=\frac{2}{3\da^{\a}}+O(\da)$ for $-1<\a<1$.
\end{remark}
\begin{proof}[Proof of Proposition \ref{keymu}]
It follows from Lemma \ref{accuratemu} that
\begin{align*}
1-\mu(t,u,\vartheta)+tL\mu(t,u,\vartheta)\les\da M^2.
\end{align*}
This implies that for $(t,u,\vartheta)\in W_{s}$,
\begin{align}
  L\mu(t,u,\vartheta)&\leq-\frac{1}{4t}\les -1.
\end{align}
To prove \eqref{keymu2}, we first choose a new coordinate $\tilde{\vartheta}$ on $\Sigma_t$ such that $T=\frac{\pa}{\pa u}|_{t,u,\tilde{\vartheta}}$. Along the integral curves of $T$, let the function $f(u)=T\ln\mu(u)$ for $u\in[0,\tilde{\da}]$. Then $f$ attains its maximum at $u=u^{\ast}$. With loss of generality, one may assume $f(u^{\ast})>0$. Otherwise $(\mu^{-1}T\mu)_+=0$. Since on $C_0^t$, $\mu\equiv 1$ and $f(u)\equiv 0$, one has $\frac{df}{du}\geq0$, i.e., $(\mu^{-1}T\mu)_+\leq\sqrt{\frac{T^2\mu}{\inf\mu}}.$ To estimate $T^2\mu$, one first commutes $X$ with \eqref{equationLTmu} and use the same procedure in proving Lemma \ref{preliminary1} to show $|XT\mu|+|TX\mu|\les\da^{-1}M$. Then one commutes $T^2$ with \eqref{equationmu} to obtain
\begin{align}
  LT^2\mu&=T^2m+eT^2\mu+2T\mu\cdot Te+\mu T^2e+(T\tpi_L^{\sharp}+\tpi_L^{\sharp}T)\mu.\label{equationLT2mu}
\end{align}
Applying the same argument in proving Lemma \ref{accuratemu} yields 
\begin{align}
|T^2\mu(t,u,\vartheta)-T^2\mu(0,u,\vartheta)|\les\da^{-1}M^2.  
\end{align}
Then, it follows from Lemma \ref{preliminary1} and the bootstrap assumptions that $|T^2\mu|\les \da^{-2}$ by integrating \eqref{equationLT2mu}. Using \eqref{keymu1} yields
\[
L\mu(0,u,\vartheta)\les L(t,u,\vartheta)+O(\da M^2)\les -1.
\]
Thus, for $t<s^{\ast}$, 
\begin{align}
\mu(t,u,\vartheta)&\geq -\int_t^{s^{\ast}}L\mu(\tau,u,\vartheta)\ d\tau=-\int_t^{s^{\ast}}L\mu(0,u,\vartheta)+O(\da M^2)\ d\tau\gtrsim
|s^{\ast}-t|.\label{mubehavior}
\end{align}
Therefore, one has \eqref{keymu2} and finished the proof of the Proposition.
\end{proof}
\begin{remark}\label{recovertfai}
	Indeed, in proving Lemma \ref{accuratemu}, we have shown that 
 \[
|T\fai(t,u,\vartheta)|\les|T\fai(0,u,\vartheta)|+O(\da M^2)\les 1.
\]
This recovers the assumption $|T\fai|\les M$ by choosing $M$ suitably large. In the proof of  Proposition \ref{keymu}, we obtain
\begin{align*}
|T^2\fai(t,u,\vartheta)|&\les|T^2\fai(0,u,\vartheta)|+O(M^2)\les \da^{-1},\\
|T^3\fai(t,u,\vartheta)|&\les|T^3\fai(0,u,\vartheta)|+O(\da^{-1}M^2)\les \da^{-2}.
\end{align*}
This recovers the assumption $|T^2\fai|+\da|T^{3}\fai|\les \da^{-1} M.$
\end{remark}
\begin{lemma}($C^{\frac{1}{3}}$ blow-up for wave variables)\label{C13}
	Let $\p(0,x)$ be defined in \eqref{selfsimilar2}. Then 
	\begin{align}
		\p,h,v^1,\zeta\in L^{\infty}([0,T_{\ast}),C^{\frac{1}{3}}).
	\end{align}
\end{lemma}
\begin{proof}
 It follows from \eqref{burgesself}-\eqref{selfsimilar2} that
  \begin{align*}
 		\min\pa_1\p
   \leq -(\da\ga'\bar{W}_0+\da\ga\pa_1\bar{W}_0)|_{x_1=1-\frac{\da}{2},x_2=0}=\ga(1-\frac{\da}{2})\pa_1\bar{U}(0)=-1.
 	\end{align*}
  Hence, a shock forms before $t=T_{\ast}$. In view of Lemma \ref{partialderivatives},
	\begin{align}
		\begin{split}
		\int_{x_1'}^{x_1}\pa_1\p\  dx_1&=\int_{u'}^u\hat{T}^1T^1\hat{T}\p\ du+\int_{\vartheta'}^{\vartheta}\hat{X}^1X^1\hat{X}\p\ d\vartheta\\
		&=\p(t,u,\vartheta)-\p(t,u',\vartheta)+O(\da).
		\end{split}
	\end{align}
	\begin{align}
		\begin{split}
			\pa_1\p&=\hat{T}^1T^1\hat{T}\p+\hat{X}^1X^1\hat{X}\p=(\hat{T}^1)^2T\p+(\hat{X}^1)^2X\p.
		\end{split}
	\end{align}
	Due to Lemma \ref{accuratemu}, one has 
 \[
 \p(t,u,\vartheta)=\p(0,u,\vartheta)+\int_0^tO(\da M^2).
\]
Hence, it holds that
\begin{align}
		T\p(t,u,\vartheta)-T\p(0,u,\vartheta)=O(\da),\quad 
		X\p(t,u,\vartheta)-X\p(0,u,\vartheta)=O(\da).
	\end{align}
	This implies $\pa_1\p=\pa_1\p(0,x_1,x_2)+O(\da)$. Therefore, one has
	\begin{align}
		\sup_{x_1\neq x_1'}\frac{|\p(t,x_1,x_2)-\p(t,x_1',x_2)|}{|x_1-x_1'|^{\frac{1}{3}}}\leq\sup_{x_1\neq x_1'}
			\frac{\int_{x_1'}^{x_1}\pa_1\p\ dx_1}{|x_1-x_1'|^{\frac{1}{3}}}\leq C+\frac{C\da t}{|x_1-x_1'|}.
	\end{align}
	Similarly,
	\begin{align}
		\sup_{x_2\neq x_2'}\frac{|\p(t,x_1,x_2)-\p(t,x_1,x_2')|}{|x_2-x_2'|^{\frac{1}{3}}}\leq \da.
	\end{align}
	Hence, for sufficiently small $\da$, it holds that
	\begin{align}
		\begin{split}
			&\sup_{(x_1,x_2)\neq(x_1',x_2')}\frac{|\p(t,x_1,x_2)-\p(t,x_1',x_2')|}{(|x_1-x_1'|^2+|x_2-x_2'|^2)^{\frac{1}{6}}}\\
   \leq &
			\sup_{x_1\neq x_1'}\frac{|\p(t,x_1,x_2)-\p(t,x_1',x_2)|}{|x_1-x_1'|^{\frac{1}{3}}}
   +\sup_{x_2\neq x_2'}\frac{|\p(t,x_1',x_2)-\p(t,x_1',x_2')|}{|x_2-x_2'|^{\frac{1}{3}}}\\
			\leq & C+\frac{C\da t}{|x_1-x_1'|^{\frac{1}{3}}}\leq C.
		\end{split}
	\end{align}
	 Indeed, one can use similar argument to show that $\p\in C^{\frac{1}{2n+1}}$ for any $n\in\mathbb{N}^{\ast}$ by choosing $\bar{U}$ as 
 \begin{align}
     -\frac{1}{2n}\bar{U}+\left(\frac{2n+1}{2n}+\bar{U}\right)\bar{U}_y=0.
 \end{align}
 This finishes the proof of the lemma.
\end{proof}

\begin{remark}(The non-zero potential vorticity near shock)
	Consider the transport equation \eqref{equationvorticity}. Since the integral curves of $\mu B$ is $g$-orthogonal to $\Sigma_t$, then one integrates \eqref{equationvorticity} along integral curves of $\mu B$ backwards and there are following two possibilities:
	\begin{itemize}
		\item either the integral curves eventually intersect $\Sigma_0$, then $\xi(t)=\xi(0)=1+O(\da)$;
		\item or some integral curves intersect $C_0$ at some $\tilde{t}>0$, then $\xi(t)=\xi(\tilde{t}).$ Since $C_0$ is the standard null cone, one integrates \eqref{equationvorticity} along the integral curves of $L$ backwards to obtain 
		\begin{align*}
			(\zeta+e^{\p})(\tilde{t})=(\zeta+e^{\p})(0)+\int_0^{\tilde{t}}-\mu^{-1}T(\zeta+e^{\p})(s)\ ds\\
             =(\zeta+e^{\p})(0)+\pa_1\int_0^{\tilde{t}}(\zeta+e^{\p})(s)\ ds=(\zeta+e^{\p})(0)\neq 0.
		\end{align*}
	\end{itemize}
	Therefore, for the initial data given in Lemma \ref{initialdata}, then the potential vorticity is non-zero on $W_{t^{\ast}}^u$.
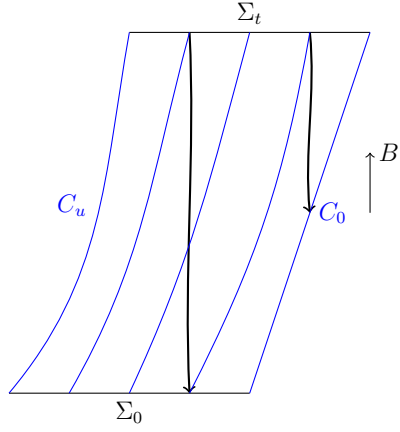
	\begin{center}
		\tikzset{global scale/.style={
				scale=#1,
				every node/.append style={scale=#1}
			}
		}
		\begin{tikzpicture}[global scale = 0.8]
		\draw (-2,0)--node[below] {$\Sigma_{0}$}(2,0);
		\draw (0,6)--node[above] {$\Sigma_{t}$}(4,6);
		\node (1) at (1,6) {};
		\draw[->, thick] (1,6)..controls (1.1,4) and (0.9,2)..(1,0);
		\draw[->] (4,3)--(4,4) node[right] {$B$};
		\draw[->, thick] (3,6)..controls (3.1,5) and (2.9,4)..(3,3);
		\path[blue] (-1,0) edge [out=60, in=255] (1,6);
		\path[blue] (0,0)edge [out=65, in=255] (2,6);
		\path[blue] (1,0)edge [out=63, in=260] (3,6);
		\draw[blue] (-2,0) edge [out=50, in=260]node[above left] {$C_{u}$}(0,6);
		\draw[blue] (2,0)--node[ right] {$C_{0}$}(4,6);
		\end{tikzpicture}
	\end{center}
\end{remark}
\quad The following lemma indicates that $\mu$ behaves like a polynomial in the shock region. This plays a key role in controlling the top order energy estimates.
\begin{lemma}\label{crucial}
	For sufficiently small $\da$, all $t\in[t_0,t^{\ast})$ and $b>1$, it holds that
	\begin{align}
		\int_0^t\mu_m^{-(b+1)}(L\mu)_-\ d\tau\leq C\frac{1}{b}\mu_m^{-b}(t)
	\end{align}
	for some constant $C=1+O(\da)$. Moreover, $\int_0^t\mu_m^{-\frac{3}{4}}(\tau)\ d\tau\leq C$ for some constant $C>0$.
\end{lemma}
\begin{remark}
	The same argument leads to 
 \[\int_0^t\mu^{-(b+1)}_m(\tau)\ d\tau\leq \tfrac{C}{b}\mu^{-b}_m(t).\]
\end{remark}
\begin{remark}\label{crucial2}
	The similar argument in proving Proposition \ref{keymu} indicates that $\mu$ is decreasing along the integral curves of $L$ after some $t_1:=\inf\{t\in[0,t^{\ast})|\ \mu_m(t)\leq\tfrac{1}{b}\}$ for some $b\geq 2$. Hence, one is able to show the following argument: for sufficiently small $\da$ and fixed $b\geq 2$, there is a constant $C_0$ depending on $\da$ and $b$ such that for all $\tau\in[0,t]$,
	\begin{align}
		\mu_m^{b}(t)\leq C_0\mu^b_m(\tau).
	\end{align}
\end{remark}
\begin{proof}
  In view of Proposition \ref{keymu}, define 
  \begin{align}
      \lam_m(t)=\sup_{(u,\vartheta)}(L\mu)_-(t,u,\vartheta)=-L\mu(t,u_m,\vartheta_m)>0.
  \end{align} Let $\mu_m(t)=\mu(t,u_t,\vartheta_t)$ and fix $s\in[t_0,t]$. Then, it follows from Lemma \ref{keymu} that 
	\begin{align}
 \begin{split}
		& L\mu(t,u_s,\vartheta_s)=L\mu(s,u_s,\vartheta_s)+O(\da M^2)\\
           =&-\lam_m(s)+L\mu(s,\mu_s,\vartheta_s)-L\mu(s,u_m,\vartheta_m)+O(\da M^2)=
		-\lam_m(s)+O(\da M^2),
  \end{split}\label{Lmuustas}
	\end{align}
	where the last inequality comes from $\mu(s,u,\vartheta)=1+sL\mu(s,u,\vartheta)+O(\da M^2)$ and $\mu(s,u_m,\vartheta_m)\geq \mu(s,u_s,\vartheta_s)$. It follows from Lemma \ref{accuratemu} that
			\begin{align*}
				\mu(s,u_m,\vartheta_m)&=1+sL\mu(s,u_m,\vartheta_m)+O(\da M^2),\\
				\mu(s,u_s,\vartheta_s)&=1+sL\mu(s,u_s,\vartheta_s)+O(\da M^2).
			\end{align*}
			This implies
   \[
   0\leq L\mu(s,u_s,\vartheta_s)-L\mu(s,u_m,\vartheta_m)\les\da M^2,
   \]
   and 
   \[
   0\leq \mu(s,u_m,\vartheta_m)-\mu_m(s)\les\da M^2.
   \]
   It follows from \eqref{Lmuustas} that
			\begin{align*}
				\mu_m(t)&=\mu(s,u_t,\vartheta_t)+\int_s^tL\mu(\tau,u_t,\vartheta_t)\geq\mu_m(s)-(t-s)\lam_m(s)+O(\da M^2)\\
				&\geq \mu(s,u_m,\vartheta_m)-sL\mu(s,u_m,\vartheta_m)-t\lam_m(s)+O(\da M^2)\geq1-t\lam_m(s)+O(\da M^2).
			\end{align*}
			On the other hand,
			\begin{align*}
				\mu_m(t)&\leq\mu(t,u_s,\vartheta_s)=\mu_m(s)+\int_s^tL\mu(\tau,u_s,\vartheta_s)=\mu_m(s)-(t-s)\lam_m(s)+O(\da M^2)\\
				&\leq \mu(s,u_m,\vartheta_m)-sL\mu(s,u_m,\vartheta_m)-t\lam_m(s)+O(\da M^2)\leq
				1-t\lam_m(s)+O(\da M^2).
			\end{align*}
		Therefore, $\mu_m(t)=1-t\lam_m(s)+O(\da M^2).$ 
Hence, it follows from Proposition \ref{keymu} that 
			\begin{align*}
				\int_0^t\mu_m^{-(b+1)}(L\mu)_-\ d\tau&\leq \int_0^t(1-\tau\lam_m(s)+O(\da M^2))^{-b-1}\ d\tau\\
                &\leq\frac{1}{b\lam_m(s)}(\mu_m(t)+O(\da M^2))^{-b}\leq C\frac{1}{b}\mu_m^{-b}(t),
			\end{align*} where $C=1+O(\da M^2)$. In particular,
			\begin{align*}
				\int_0^t\mu_m^{-\frac{3}{4}}(\tau)&\leq\int_0^t(1-\tau\lam_m(s)-\tfrac{1}{10})^{-\frac{3}{4}}
				=\frac{4}{\lam_m(s)}((\frac{9}{10})^{\frac{1}{4}}-(1-t\lam_m(s)-\frac{1}{10})^{\frac{1}{4}})\leq C.
			\end{align*}
  This finishes the proof of Lemma \ref{crucial}.
\end{proof}
\begin{remark}
For the initial data as given in Remark \ref{main2} and under the same assumption as in Lemma \ref{crucial}, the similar proof yields
\begin{align*}
	\int_0^t\mu_m^{-(b+1)}(L\mu)_-\ dt'\leq \frac{C}{b}\da^{-\a}\mu_m^{-b}(t).
\end{align*}
\end{remark}
\section{Fundamental energy-flux estimates}\label{section4}
In this section, we derive the fundamental energy estimates for $\Psi\in\{\p,v^i,\zeta-\varrho\}$. This relies on the multiplier method. The energy-momentum tensor associated with $\fai$ is defined as $T_{\a\be}=\pa_{\a}\fai\pa_{\be}\fai-\frac{1}{2}g_{\a\be}g^{\mu\nu}\pa_{\mu}\fai\pa_{\nu}\fai.$ In the null frames, $T$ can be computed as follows: 
\begin{align*}
  T_{LL}&=(L\fai)^2,\quad T_{\dl\dl}=(\dl\fai)^2,\quad T_{L\dl}=\mu(\hat{X}\fai)^2,\quad T_{LX}=L\fai\cdot X\fai,\\
  T_{\dl X}&=\dl\fai\cdot X\fai,\quad T_{XX}=\frac{1}{2}(X\fai)^2+\frac{1}{2}\mu^{-1}\sg L\fai\cdot\dl\fai.
\end{align*}

\begin{lemma}\label{divergence}(\textbf{Divergence theorem})
\ For any space-time vector field $J=J^t\frac{\pa}{\pa t}+J^u\frac{\pa }{\pa u}+\slashed{J}\frac{\pa }{\pa \vartheta}$, it holds that
\begin{align*}
  \int_{W_t^u}\mu\text{div}J&=\int_{\Sigma_t^u}\mu J^t+\int_{C_u^t}\mu J^u-\int_{\Sigma_0^u}\mu J^t-\int_{C_0^t}\mu J^u.
\end{align*}
\end{lemma}
In this paper, we choose the multipliers as $K_0=\dl$, $K_1=(1+\eta^{-2}\mu)L$ and define $J_0^{\a}=-T_{\be}^{\a}K_0^{\be},\ J_1^{\a}=-T_{\be}^{\a}K_1^{\be}$, $J^{\zeta}=(\zeta-\varrho)^2B$ where $J^{\zeta}$ is used to produce the energy and flux for the vorticity. Note that 
\begin{align}
	\begin{split}\label{DaJa}
	\mu D_{\a}J_0^{\a}&=\mu Q_0:=-\mu(K_0\fai\Box_g\fai+\frac{1}{2}T^{\a\be}\pi_{0,\a\be}),\\
	\mu D_{\a}J_1^{\a}&=\mu Q_1:=-\mu(K_1\fai\Box_g\fai+\frac{1}{2}T^{\a\be}\pi_{1,\a\be}),\\
	\mu D_{\a}J^{\zeta,\a}&=\mu Q^{\zeta}:=(L\mu+2\eta T\eta)(\zeta-\varrho)^2+2(\zeta-\varrho)\cdot\mu B(\zeta-\varrho)+\eta\mu(\zeta-\varrho)^2 \tr\slashed{k},\\
	\end{split}
\end{align}
where $\pi_0$ and $\pi_1$ are the deformation tensor associated with $K_0$ and $K_1$, respectively. In the null frame, they are given in the following table.
\begin{table}[H]
	\setlength{\tabcolsep}{12pt}
	\setlength{\belowcaptionskip}{0.2cm}
	\centering
	\caption{The deformation tensors of the multipliers}\label{table3}
	\begin{tabular}{c|cc}
		\hline
		$\pi$ & $\pi_0$ &$\pi_1$ \\
		\hline
		$\pi_{LL}$ & 0 & 0 \\
		$\pi_{L\dl}$ &$-2\dl\mu-2\mu L(\eta^{-2}\mu)$ &$-2(1+\eta^{-2}\mu)L\mu-2\mu L(\eta^{-2}\mu)$ \\
		$\pi_{\dl\dl}$ & 0& $-4\mu \dl(\eta^{-2}\mu)+4\mu(1+\eta^{-2}\mu)L(\eta^{-2}\mu)$ \\
		$\pi_{LX}$ &$-2(\sigma+\tau)$ &0 \\
		$\pi_{\dl X}$ & $-2\mu X(\eta^{-2}\mu)$ &$2(1+\eta^{-2}\mu)(\sigma+\tau)$ \\
		$\pi_{XX}$ &$2\underline{\chi}$ & $2(1+\eta^{-2}\mu)\chi$ \\
		\hline
	\end{tabular}
\end{table}
\begin{definition}\label{energy}
	Define the energies and fluxes to the wave variables $\fai\in\{\p,v^1,v^2\}$ and $\zeta$ as follows.
	\begin{align*}
	&E_0(\fai)=\int_{\Sigma_t^u}(\dl\fai)^2+\eta^{-2}\mu^2(\hat{X}\fai)^2,\quad F_0(\fai)=\int_{C_u^t}\mu(\hat{X}\fai)^2,\\
	&E_1(\fai)=\int_{\Sigma_t^u}(1+\eta^{-2}\mu)(\mu(\hat{X}\fai)^2+\eta^{-2}\mu (L\fai)^2),\quad F_1(\fai)=\int_{C_u^t}
	(1+\eta^{-2}\mu)(L\fai)^2,\\
	&E^{\zeta}(\zeta-\varrho)=\int_{\Sigma_t^u}\mu(\zeta-\varrho)^2,\quad F^{\zeta}(\zeta-\varrho)=\int_{C_u^t}\eta^2(\zeta-\varrho)^2.
	\end{align*}
\end{definition}

Applying Lemma \ref{divergence} to \eqref{DaJa} yields
\begin{align}
	E_{\a}(t,u)+F_{\a}(t,u)&=E_{\a}(0,u)+\int_{W_t^u}\mu Q_{\a},\quad \a=0,1,\label{energywave}\\
	E^{\zeta}(t,u)+F^{\zeta}(t,u)&=E^{\zeta}(0,u)+\int_{W_t^u}\mu Q^{\zeta}.\label{energyzeta}
\end{align}
\begin{remark}
	If one takes $J^{\p}=\p^2 B$ and $J^{v^i}=(v^i)^2B$, then 
	\begin{align}
		\mu D_{\a}J^{\p,\a}&=\mu Q^{\p}=(L\mu+2\eta T\eta)\p^2+2\p\mu B\p+\eta\mu\p^2 \tr\slashed{k},\\
		\mu D_{\a}J^{v^i,\a}&=\mu Q^{v^i}=(L\mu+2\eta T\eta)(v^i)^2+2\p\mu Bv^i+\eta\mu (v^i)^2 \tr\slashed{k}.
	\end{align}
	Since $\mu B\p=-\mu \text{div}v$ and $Bv^i=\ep_{ij}-e^{\p}\pa_i\p$, the energy-flux estimates lose one derivative through this approach.
\end{remark}
For $Q_{\a}$ ($\a=0,1$), one has the following decomposition with $Q_{\a}=\sum_{i=0}^6Q_{\a,i}$.
\begin{table}[H]
	\setlength{\tabcolsep}{12pt}
	\setlength{\belowcaptionskip}{0.2cm}
	\centering
 \caption{The decomposition of $Q_{\a}$}\label{table4}
	\begin{tabular}{c|cc}
		\hline
		$Q_{\a,i}$ & $Q_0$ &$Q_1$ \\
		\hline
		$Q_{\a,0}$ & $\mu^{-1}\mathfrak{F}_0^0\dl\p+\mu^{-1}\mathfrak{F}_0^i\dl v^i$ & $\mu^{-1}\mathfrak{F}_0^0(1+\eta^{-2}\mu)L\p+\mu^{-1}\mathfrak{F}_0^i(1+\eta^{-2}\mu)L v^i$ \\
		$Q_{\a,1}$ & 0 & 0 \\
		$Q_{\a,2}$ &$0$ &$\frac{1}{2}[\dl(\eta^{-2}\mu)-(1+\eta^{-2}\mu)L(\eta^{-2}\mu)](L\fai)^2$ \\
		$Q_{\a,3}$ & $\frac{1}{2}[\boxed{\mu^{-1}\dl\mu}+L(\eta^{-2}\mu)](\hat{X}\fai)^2$ & $ \frac{1}{2}[\boxed{\mu^{-1}(1+\eta^{-2}\mu)L\mu}+L(\eta^{-2}\mu)](\hat{X}\fai)^2 $  \\
		$Q_{\a,4}$ &$\mu^{-1}\sg^{-\frac{1}{2}}(\tau+\sigma)\dl\fai\hat{X}\fai$ &0 \\
		$Q_{\a,5}$ & $\hat{X}(\eta^{-2}\mu)L\fai \hat{X}\fai$ &$\mu^{-1}\sg^{-\frac{1}{2}}(1+\eta^{-2}\mu)L\fai\hat{X}\fai$ \\
		$Q_{\a,6}$ &$-[\mu^{-1}L\fai\dl\fai+(\hat{X}\fai)^2]\tr\underline{\chi}$ &$-[\mu^{-1}L\fai\dl\fai+(\hat{X}\fai)^2](1+\eta^{-2}\mu)\tr\chi$ \\
		\hline
	\end{tabular}
 
\end{table}

Note that in Table \ref{table4}, $\mathfrak{F}_0^0=\mu\Box_g\p$ and $\mathfrak{F}_0^i=\mu\Box_g v^i$. The boxed terms in Table \ref{table4} are the most dangerous since they involve $\mu^{-1}(\hat{X}\fai)^2$ which may be singular in the shock region. In view of Remark \ref{nullforms}, one writes $\mathfrak{F}_0^0$ and $\mathfrak{F}_0^i$ as
\begin{align*}
	\mathfrak{F}_0^0=&2\eta e^{-\p}(Tv^1\hat{X}v^2-Tv^2\hat{X}v^1)-\frac{1}{2}L\p\dl\p+\frac{1}{2}\mu(\hat{X}\p)^2+\zeta,\\
	\mathfrak{F}_0^i= &-\ep_{ia}e^{\p}(\hat{T}^a\kappa L(\zeta-\varrho)+\mu\hat{X}^a\hat{X}(\zeta-\varrho))+\mu v^i(2\zeta+e^{-\p})+\ep_{ij}\zeta
	(\mu L+\eta^2 T)v^j\\
	&+\frac{1}{4}(Lv^i\dl\p+\dl v^i L\p)-\frac{1}{2}\mu\hat{X}v^i\hat{X}\p+e^{-\p}\ep_{ij}v^j(\mu L+\eta^2 T)\p.
\end{align*}

\subsection{Energy-flux estimates for the wave variables}\label{fundamentalenergy}
 To derive the energy estimates, one needs the following $L^2$-estimates for $\fai$.
\begin{lemma}(\cite[Lemma 5.5]{CZD1})\label{faiL2}
	For any smooth function $f$ which vanishes on $C_0$, the following estimate holds.
	\begin{align}
		\int_{S_{t,u}}f^2\les \tilde{\da}\int_{\Sigma_t^u}\mu^2 (Lf)^2+(\dl f)^2.
	\end{align}
\end{lemma}

 We first deal with integrals involving the boxed terms. In view of Lemma \ref{preliminary1}, it holds that 
\begin{align}
	\int_{W_t^u}Q_{1,3}&\les\int_{W_t^u}\mu^{-1}L\mu\cdot\mu(\hat{X}\fai)^2+\int_0^uF_0(\fai).\label{Q131}
\end{align}
To deal with any integral involving $\mu^{-1}L\mu$ or $\mu^{-1}\dl\mu$, one can split it into shock part (in $W_{s}$) and non-shock part. The difficult shock part can be bounded by utilizing Proposition \ref{keymu}. Let $W_{n-s}=W_{t,u}\backslash W_{s}$. It follows from Lemma \ref{preliminary1} and Proposition \ref{keymu} that 
\begin{equation}\label{Q132}
\begin{aligned}
	\int_{W_t^u}(\mu^{-1}L\mu)\mu(\hat{X}\fai)^2&=\left(\int_{W_s}+\int_{W_{n-s}}\right)(\mu^{-1}L\mu)\mu(\hat{X}\fai)^2\\
	&\les-\int_{W_s}(\hat{X}\fai)^2+\int_0^uF_0(\fai).
\end{aligned}
\end{equation}
Let $K(t,u)=\int_{W_s}(\hat{X}\fai)^2$. This is a non-negative integral to control $(\hat{X}\fai)^2$ in the shock region. Combing \eqref{Q131} and \eqref{Q132} leads to
\begin{align}
	\int_{W_t^u}Q_{1,3}&\les -K(t,u)+\int_0^uF_0(\fai).
\end{align}
It follows from \eqref{preliminary1} that 
\begin{align}
	\int_{W_t^u}Q_{0,3}&\les\int_{W_t^u}(\mu^{-1}T\mu)\mu(\hat{X}\fai)^2+\int_0^uF_0(\fai).
\end{align}
In view of Proposition \ref{keymu}, one obtains 
\begin{align}
	\int_{W_t^u}(\mu^{-1}T\mu)\mu(\hat{X}\fai)^2&\les \da^{-1}\int_0^t\frac{E_1(\fai)}{\sqrt{s^{\ast}-t'}}+E_1(\fai)\ dt'.
\end{align}
Therefore, it holds that
\begin{align}
	\int_{W_t^u}Q_{0,3}&\les\da^{-1}\int_0^t\frac{E_1(\fai)}{\sqrt{s^{\ast}-t'}}+E_1(\fai)\ dt'+\int_0^uF_0(\fai).
\end{align}
Next we treat $Q_{0,4},Q_{1,5},Q_{0,0}$, and $Q_{1,0}$. It follows from the definition of $K(t,u)$ and Table \ref{table2} that 
\begin{align}
	\int_{W_t^u}Q_{0,4}&\les\int_{W_s}\dl\fai\hat{X}\fai+\int_0^tE_0(\fai)\les K(t,u)+\int_0^tE_0(\fai),\\
		\int_{W_t^u}Q_{1,5}
&\les
			\da K(t,u)+\int_0^tE_1(\fai)+\da^{-1}\int_0^uF_1(\fai).
\end{align}
Lemma \ref{faiL2}, together with the bootstrap assumptions yields
\begin{align}
	\begin{split}
	&\int_{W_t^u}\mathfrak{F}_0^0\cdot \dl\p\les\int_{W_t^u}\da (Tv^i+T\p+\hat{X}\p)\dl\p+\mu(\zeta-\varrho+\varrho)\dl\p\\
   \les
	&(1+\da)\int_0^tE_0(\fai)+\da E_1(\fai)+\int_0^tE^{\zeta}(\zeta-\varrho),\\
	\end{split}\\
	&\int_{W_t^u}\mathfrak{F}_0^i\cdot (1+\eta^{-2}\mu)L\p\les (1+\da)\int_0^tE_1(\fai)+\da\int_0^tE_0(\fai)+\int_0^tE^{\zeta}(\zeta-\varrho).
\end{align}
Similarly, one has
\begin{equation*}
	\begin{aligned}
	\int_{W_t^u}\mathfrak{F}_0^i\cdot(\dl v^i+(1+\eta^{-2}\mu) Lv^i)\les &\int_{W_t^u}\mu(L(\zeta-\varrho)+Y(\zeta-\varrho))\dl v^i+\da\int_0^tE_0(\fai)+E_1(\fai)\\
	\les &\int_0^t(1+\da)E_0(\fai)+\da E_1(\fai)+\int_0^tE^{\zeta}(P(\zeta-\varrho)).
	\end{aligned}
 \end{equation*}
The estimates for the remaining terms are trivial due to Lemma \ref{preliminary1} and the bootstrap assumptions. It holds that 
\begin{align}
	\int_{W_t^u}Q_{0,5}+Q_{0,6}+Q_{1,2}+Q_{1,6}&\les \int_0^t(1+\da)E_0(\fai)+E_1(\fai)+\int_0^uF_1(\fai).
\end{align}
Collecting the above estimates yields 
\begin{align}
\begin{split}
	&(E_0+F_0)(\fai)(t,u)+\da^{-1}(E_1+F_1)(\fai)(t,u)\\
 \les & (E_0+E_1)(0,u)+\int_0^tE^{\zeta}(P^{\leq 1}(\zeta-\p))-\da^{-1}K(t,u)
	\\
	&+\da^{-1}\int_0^t\frac{E_1}{\sqrt{s^{\ast}-t'}}\ dt'+\int_0^tE_0+\da^{-1}E_1\ dt'+\int_0^uF_0+\da^{-1}F_1\ du'.
 \end{split}
\end{align}
Applying the Gronwall inequality and noticing that $\frac{1}{\sqrt{s^{\ast}-t}}$ is integrable in time lead to 
\begin{align}
	(E_0+F_0)(\fai)(t,u)+\da^{-1}(E_1+F_1+K)(\fai)(t,u)\les \mathscr{D}_0+\int_0^tE^{\zeta}(P^{\leq 1}(\zeta-\varrho)),\label{fai0}
\end{align}
where $\mathscr{D}_0:=(E_0+\da^{-1}E_1)(0,u)$.
\subsection{Energy-flux estimates for the vorticity}\label{section42} It follows from \eqref{energyzeta}, Lemma \ref{preliminary1} and \ref{2ndff} that 
\begin{equation}\label{zeta01}
	\begin{aligned}	
		(E^{\zeta}+F^{\zeta})(\zeta-\varrho)(t,u)
  \les  E^{\zeta}(\zeta-\varrho)(0,u)
  +\int_0^uF^{\zeta}(\zeta-\varrho)\ du'+\da\int_0^tE^{\zeta}(\zeta-\varrho) dt'.
	\end{aligned}
\end{equation}
Applying the Gronwall inequality to \eqref{zeta01} yields 
  \begin{align}
  	(E^{\zeta}+F^{\zeta})(\zeta-\varrho)\les  E^{\zeta}(\zeta-\varrho)(0,u).\label{zeta0}
  \end{align}
To obtain the energy-flux estimates for the first derivatives of $\zeta$, one notes that for $P\in\mathcal{P}$, 
\begin{equation}
\begin{aligned}
	\mu BP(\zeta-\varrho)&=-(P\mu)L(\zeta-\varrho)+P(\eta^2)\mu L(\zeta-\varrho)+(\mu\ppi_L^{\sharp}+\eta^2\ppi_T^{\sharp})(\zeta-\varrho)\\
 &:=\mathfrak{F}_1^{\zeta}.
\end{aligned}
\end{equation}
\begin{itemize}
	\item For $P=L$, it follows from Lemma \ref{preliminary1} and table \ref{table2} that 
	\begin{align}
		\int_{W_t^u}2L(\zeta-\varrho)\cdot \mathfrak{F}_1^{\zeta}&\les \da\int_0^t E^{\zeta}(L(\zeta-\varrho))
		+\int_0^u F^{\zeta}(P(\zeta-\varrho)).\label{Lzeta1}
	\end{align}
	\item For $P=Y$, one has
	\begin{align}
		\begin{split}\label{Yzeta1}
		\int_{W_t^u}2Y(\zeta-\varrho)\cdot \mathfrak{F}_1^{\zeta}&\les\int_{W_t^u}(1+\da)\mu Y(\zeta-\varrho)L(\zeta-\varrho)+(1+\da)(Y(\zeta-\varrho))^2\\
		&\les\da\int_0^t E^{\zeta}(P(\zeta-\varrho))+\int_0^uF^{\zeta}(P(\zeta-\varrho)).\\
		\end{split}
	\end{align}
\end{itemize}
The estimates \eqref{energyzeta}, \eqref{Lzeta1}, \eqref{Yzeta1}, together with Lemma \ref{2ndff} yield
\begin{equation}
\begin{aligned}
	&(E^{\zeta}+F^{\zeta})(P(\zeta-\varrho))(t,u)\\
 \leq & E^{\zeta}(P(\zeta-\varrho))(0,u)+C\da\int_0^tE^{\zeta}(P(\zeta-\varrho))+
	C\int_0^uF^{\zeta}(P(\zeta-\varrho)).
\end{aligned}  
\end{equation}
Applying the Gronwall inequality gives
\begin{align}
	(E^{\zeta}+F^{\zeta})(P(\zeta-\varrho))(t,u) &\leq CE^{\zeta}(P^{\leq 1}(\zeta-\varrho))(0,u):=C\mathscr{D}_{1}^{\zeta}.\label{zeta1}
\end{align}
\begin{remark}
	Note that in this subsection, the energy-flux estimates for $\zeta$ can be derived directly without relying on the energy-flux estimates of the wave variables. However, when higher-order derivatives are involved, the commutation between $P^{\alpha}$ and $\mu B$ generates complex terms containing higher-order derivatives of $\varphi \in \{\p, v^i\}$, $\mu$ and $\chi$. This makes estimates for $E^{\zeta}$ and $F^{\zeta}$ couple with those for $E_{\alpha}$ and $F_{\alpha}$ with $\alpha = 0, 1$, so that the analysis is more complicated.
\end{remark}
In conclusion, collecting \eqref{fai0}, \eqref{zeta0} and \eqref{zeta1} yields the following fundamental energy-flux estimates for the wave variables and the vorticity
\begin{align}
	(E_0+\da^{-1}E_1)(t,u)+(F_0+\da^{-1}F_1)(t,u)+\da^{-1}K(t,u)&\leq C(\mathscr{D}_0+\mathscr{D}_1^{\zeta}),\\
	E^{\zeta}(P^{\leq1}(\zeta-\varrho))+F^{\zeta}(P^{\leq1}(\zeta-\varrho))&\leq C\mathscr{D}_1^{\zeta}.
\end{align}
\begin{definition}\label{highorderenergy}
	For $0\leq N\leq N_{top}=20$, the high order energies and fluxes are defined as follows.
	\begin{align}
		\mathbb{W}_N&=\sup_{(t',u')\in[0,t]\times[0,u]}\left\{
		\sum_{Z^{\a}\in\mathcal{Z}^{N;\leq 1}}\da^{2m}(E_0(Z^{\a}\fai)+E_1(Z^{\a}\fai))\right\},\\
		\mathbb{U}_N&=\sup_{(t',u')\in[0,t]\times[0,u]}\left\{
		\sum_{Z^{\a}\in\mathcal{Z}^{N;\leq 1}}\da^{2m}(E_1(Z^{\a}\fai))\right\},\\
		\mathbb{Q}_N&=\sup_{(t',u')\in[0,t]\times[0,u]}\left\{\sum_{Z^{\a}\in\mathcal{Z}^{N;\leq 1}}\da^{2m}(F_0(Z^{\a}\fai)+F_1(Z^{\a}\fai))\right\},\\
		\mathbb{V}_{N+1}&=\sup_{(t',u')\in[0,t]\times[0,u]}\left\{
		\sum_{P^{\a}\in\mathcal{P}^{N+1;\leq 3}}\da^{2(p-1)_+}(E^{\zeta}(P^{\a}(\zeta-\varrho))+F^{\zeta}(P^{\a}(\zeta-\varrho)))\right\},
	\end{align}
	where $m$ is the number of $T$ in $Z^{\a}$ and $p$ is the number of $L$ in $P^{\a}$. Denote $\mathbb{W}_{\leq N}=\sum_{N'\leq N}\mathbb{W}_{N'}$ and similarly for $\mathbb{U}_{\leq N},$ $\mathbb{Q}_{\leq N}$ and $ \mathbb{V}_{\leq N+1}$.
\end{definition}
\subsection{Top order acoustical terms}\label{toporderterms}
In this subsection, we obtain the main difficult terms to be handled in the energy estimates after commutation. 
\begin{lemma}
For any spacetime vector-fields $Z$, it holds that
	\begin{align}
		\begin{split}
			\Box_g(Z\fai)&=Z(\Box_g\fai)+\frac{1}{2}\tr\zpi\cdot\Box_g\fai+\text{div}(\prescript{Z}{}{J}),\quad
			\prescript{Z}{}{J}^{\mu}=(\zpi^{\mu\nu}-\frac{1}{2}\tr\zpi g^{\mu\nu})\pa_{\nu}\fai.
		\end{split}
	\end{align}
\end{lemma}
\begin{proof}
	It follows from the definition of $\Box_g$ that 
	\begin{align*}
		\Box_g(Z\fai)&=g^{\a\be}D_{\a}(D_{\be}Z^{\ga}D_{\ga}\fai+Z^{\ga}D_{\be}D_{\ga}\fai)\\
		&=g^{\a\be}D_{\a}D_{\be}Z^{\ga}D_{\ga}\fai+2D^{\a}Z^{\ga}D_{\a}D_{\ga}\fai+g^{\a\be}Z^{\ga}D_{\a}D_{\be}D_{\ga}\fai.
	\end{align*}
	In view of the fact $D_{\be}Z_{\ga}=\frac{1}{2}(D_{\be}Z_{\ga}-D_{\ga}Z_{\be})+\frac{1}{2}\zpi_{\be\ga}$, one has
 \begin{equation*}
	\begin{aligned}	2D^{\a}Z^{\ga}D_{\a}D_{\ga}\fai=
		(D^{\be}Z^{\ga}-D^{\ga}Z^{\be})D_{\be}D_{\ga}\fai+\zpi^{\a\be}D_{\a}D_{\ga}\fai= \zpi^{\a\be}D_{\a}D_{\ga}\fai,
  \end{aligned}
  \end{equation*}
  and 
  \begin{equation*}   
  \begin{aligned}
		g^{\a\be}D_{\a}D_{\be}Z^{\ga}D_{\ga}\fai= &D_{\a}(\zpi^{\a\ga}-2D^{\ga}Z^{\a})D_{\ga}\fai
  =D_{\a}\zpi^{\a\ga}D_{\ga}\fai-D^{\ga}D_{\a}Z^{\a}D_{\ga}\fai+\mathscr{R}_a^{\ga}Z^{\a}D_{\a}\fai,
	\end{aligned}
  \end{equation*}
	where $\mathscr{R}_{\a\be}=g^{\ga\da}\mathcal{R}_{\a\ga\be\da}$ is the Ricci tensor. Note also that
 \begin{equation*}
	\begin{aligned}	g^{\a\be}Z^{\ga}D_{\a}D_{\be}D_{\ga}\fai= &g^{\a\be}Z^{\ga}D_{\a}D_{\ga}D_{\be}\fai
  =-\mathscr{R}_{\ga}^{\be}Z^{\ga}D_{\be}\fai+Z(\Box_g\fai).
	\end{aligned}
 \end{equation*}
	Collecting the above results yields 
	\begin{align*}
 \Box_g(Z\fai)&=Z(\Box_g\fai)+\zpi^{\a\ga}D_{\a}D_{\ga}\fai+D_{\a}\zpi^{\a\ga}D_{\ga}\fai-\frac{1}{2}D^{\ga}\tr\zpi D_{\ga}\fai\\
		&=Z(\Box_g\fai)+\frac{1}{2}\tr\zpi\Box_g\fai+D_{\a}(\zpi^{\a\ga}D_{\ga}\fai-\frac{1}{2}g^{\a\ga}\tr\zpi D_{\ga}\fai).
	\end{align*}
 This finishes the proof of the lemma.
\end{proof}
 Let $\mu\Box_g\fai=\mathfrak{F}_0$ for $\fai\in\{\p,v^1,v^2\}$ and $\mathfrak{F}_n=Z\mathfrak{F}_{n-1}$. Hence, the above lemma implies
\begin{align}
\mu\Box_g(Z\fai)&=Z(\mathfrak{F}_0)+(\frac{1}{2}\tr\zpi-\mu^{-1}Z\mu)\mathfrak{F}_0+\mu\text{div}(\prescript{Z}{}{J}).
\end{align}
This implies
\begin{align}
\mathfrak{F}_n&=Z\mathfrak{F}_{n-1}+(\frac{1}{2}\tr\zpi-\mu^{-1}Z\mu)\mathfrak{F}_{n-1}+\mu\text{div}(\prescript{Z}{}{J}_{n-1})
:=(Z+\prescript{Z}{}{\da})\mathfrak{F}_{n-1}+\prescript{Z}{}{\sigma}_{n-1}.
\end{align}
Therefore, it holds that
\begin{equation}
\begin{aligned}
\mathfrak{F}_n= &(Z_{n-1}+\prescript{Z_{n-1}}{}{\da})\cdots(Z_1+\prescript{Z_1}{}{\da})\mathfrak{F}_0\\
&+
\sum_{k=0}^{n-2}(Z_{n-1}+\prescript{Z_{n-1}}{}{\da})\cdots (Z_{n-k}+\prescript{Z_{n-k}}{}{\da})\prescript{Z_{n-k-1}}{}{\sigma}_{n-k-1},
\end{aligned}
\end{equation}
where
\begin{equation}
\begin{aligned}
&\qquad\quad\quad \mathfrak{F}_0^0=2e^{-\p}\mu\mathcal{D}_{12}(v^1,v^2)+\frac{1}{2}\mu\mathcal{D}(\p,\p)+\mu\zeta,\\
\mathfrak{F}_0^i=&-\ep_{ia}\mu e^{\p}\pa_a\xi+(\mu v^i+\mu \ep_{ij}Bv^j)\zeta-\frac{1}{2}\mu\mathcal{D}(v^i,\p)+e^{-\p}\mu v^i+e^{-\p}\ep_{ij}v^j \mu B\p.
\end{aligned}
\end{equation}
 Note that from Table \ref{table2}, $\prescript{Z}{}{\delta}=\frac{1}{2}\tr\prescript{Z}{}{\slashed{\pi}}\les\da.$
\begin{lemma}\label{divdecom}
	For any space time vector field $J$, 
	it holds that
	\begin{align}
	\mu \text{div} J 
	&=-\frac{1}{2}L({J_{\dl}})-\frac{1}{2}\dl(J_{L})+\slashed{\text{div}}(\mu\slashed{J})
	-\frac{1}{2}L(\eta^{-2}\mu) J_{L}-\frac{1}{2}\tr\chi J_{\dl}-\frac{1}{2}\tr\underline{\chi} J_L.
	\end{align}
\end{lemma}
\quad The top order acoustical terms in $\mathfrak{F}_n$ is contained in $Z_{n-1}\cdots Z_2\prescript{Z_1}{}{\sigma}_1$. Let $Z_{n-1}\cdots Z_1=Y^{\a'}T^mL^p$ with $a'+m+p=N_{top}=20$ and $p\leq 1, \ m\leq 1$. It follows from Lemma \ref{divdecom} that $\prescript{Z}{}{\sigma}_{n-1}$ can be divided into \[\prescript{Z}{}{\sigma}_{n-1}=\prescript{Z}{}{}{\sigma}_{n-1,1}+\prescript{Z}{}{}{\sigma}_{n-1,2}+\prescript{Z}{}{}{\sigma}_{n-1,3}\] where 
\begin{align*}
	\prescript{Z}{}{\sigma}_{n-1,1}&=\tfrac{1}{2}\tr\prescript{Z}{}{\slashed{\pi}}(L\dl\fai_{n-1}+\tfrac{1}{2}\tr\chi\dl\fai_{n-1})
	+\tfrac{1}{4}\mu^{-1}\zpi_{\dl\dl}L^2\fai_{n-1}-\frac{1}{2}\zpi_{\dl\hat{X}}L\hat{X}\fai_{n-1}\\
 &-\tfrac{1}{2}\zpi_{L\hat{X}}\dl\hat{X}\fai_{n-1}-\tfrac{1}{2}\zpi_{L\hat{X}}\hat{X}\dl\fai_{n-1}-\tfrac{1}{2}\zpi_{\dl\hat{X}}\hat{X}L\fai_{n-1}+\tfrac{1}{2}(\zpi_{L\dl}+\mu \tr\prescript{Z}{}{\slashed{\pi}})\hat{X}^2\fai,\\
	\prescript{Z}{}{\sigma}_{n-1,2}&=\tfrac{1}{4}(\dl \tr\prescript{Z}{}{\slashed{\pi}}+L(\mu^{-1}\zpi_{\dl\dl})-2\hat{X}\zpi_{\dl\hat{X}})L\fai_{n-1}\\
 &-\tfrac{1}{2}(L\zpi_{\dl\hat{X}}+\dl\zpi_{L\hat{X}}-\hat{X}\zpi_{L\dl}-\mu\zpi_{\hat{X}\hat{X}})\hat{X}\fai_{n-1}+\tfrac{1}{4}(L\tr\prescript{Z}{}{\slashed{\pi}}-2\hat{X}\zpi_{L\hat{X}})\dl\fai_{n-1},\\
	\prescript{Z}{}{\sigma}_{n-1,3}&=\tfrac{1}{4}(\tr\underline{\chi}\tr\prescript{Z}{}{\slashed{\pi}}+\mu^{-1}\zpi_{\dl\dl}\tr\chi)L\fai_{n-1}-\tfrac{1}{2}(\tr\prescript{Z}{}{\slashed{\pi}}\tpi_{L\hat{X}}+(L(\eta^{-2}\mu)+\tr\underline{\chi})\zpi_{L\hat{X}}\\
 &+\tr\chi\zpi_{\dl\hat{X}})\hat{X}\fai_{n-1}.
\end{align*}
Denote $[F]_{P.A.}$ to be the principle acoustical part of an object $F$, i.e., the part involves highest order derivatives acting on $\mu$ or $\chi$. Using Lemma \ref{divdecom} yields
\begin{align*}
	[\prescript{Z}{}{\sigma}_1]_{P.A.}= &[\tfrac{1}{2}(T\tr\prescript{Z}{}{\slashed{\pi}}-\hat{X}\zpi_{\dl\hat{X}})L\fai\\
	&+\tfrac{1}{2}(\hat{X}\zpi_{L\dl}-2T\zpi_{L\hat{X}}+\hat{X}(\mu \zpi_{\hat{X}\hat{X}}))\hat{X}\fai-\hat{X}\zpi_{L\hat{X}}T\fai]_{P.A.}.
\end{align*}
For $Z=L,T$, or $Y$, it follows from Table \ref{table1} and \eqref{equationtchi} that 
\begin{align*}
	[\prescript{L}{}{\sigma}_1]_{P.A.}&=[(T\tr\chi-\slashed{\Delta}\mu)L\fai+\mu\hat{X}\tr\chi\hat{X}\fai]_{P.A.}=\mu\hat{X}\tr\chi\hat{X}\fai,\\
	[\prescript{T}{}{\sigma}_1]_{P.A.}&=[-\eta^{-2}\mu (T\tr\chi-\slashed{\Delta}\mu)L\fai+\slashed{\Delta}\mu T\fai-\eta^{-2}\mu^2
	\hat{X}\tr\chi\hat{X}\fai]_{P.A.}=\slashed{\Delta}\mu T\fai-\eta^{-2}\mu^2
	\hat{X}\tr\chi\hat{X}\fai,\\
	[\prescript{Y}{}{\sigma}_1]_{P.A.}&=\eta^{-2}\mu\hat{X}^2\hat{X}\tr\chi L\fai+\hat{X}^2\hat{X}\tr\chi T\fai-\eta^{-1}\mu\hat{X}^1\hat{X}\tr\chi\hat{X}\fai.
\end{align*}
Therefore, the top order acoustical terms in $\mathfrak{F}_{n}$ are 
\begin{align}
   Y^{\a-1}\hat{X}\tr\chi\cdot T\fai\ \text{and}\ Y^{\a-1}\slashed{\Delta}\mu\cdot T\fai\label{topordertermswave}
\end{align} with $|\a|=N_{top}=20.$ 
For the higher-order derivatives of $\zeta$, one commutes $P^{\a+1}$ where $|\a|=N_{top}$ with \eqref{equationvorticity} to obtain
\begin{align}
\begin{split}\label{commutationvorticity}
\mu BP_i^{\a+1}(\zeta-\varrho)= \mathfrak{F}_{|\a|+1}^{\zeta}:= &\sum_{|\be_1|+|\be_2|=|\a|}
P_i^{\be_1}\ppi_{\mu B}^{\sharp}P_i^{\be_2}(\zeta-\varrho)\\
&-([(P\mu)L,P_i^{\a}]+[P(\eta^2)T,P_i^{\a}])(\zeta-\varrho),\\
[(P\mu)L,P_i^{\a}](\zeta-\varrho)= &-(P\mu)\sum_{|\be_1|+|\be_2|=|\a|-1}P_i^{\be_1}\ppi_L^{\sharp}P_i^{\be_2}(\zeta-\varrho)\\
&-\sum_{\substack{|\be_1|+|\be_2|=|\a|, \\|\be_1|>0}}P_i^{\be_1}(P\mu)P_i^{\be_2}L(\zeta-\varrho),\\
[P(\eta^2)T,P_i^{\a}](\zeta-\varrho)=&-(P\eta^2)\sum_{|\be_1|+|\be_2|=|\a|-1}P_i^{\be_1}\ppi_T^{\sharp}P_i^{\be_2}(\zeta-\varrho)\\
&-\sum_{\substack{|\be_1|+|\be_2|=|\a|,\\ |\be_1|>0}}P_i^{\be_1}(P\eta^2)P_i^{\be_2}T(\zeta-\varrho).
\end{split}
\end{align}
The top order acoustical terms in $\mathfrak{F}_{|\a|+1}^{\zeta}$ are contained in $Y^{\a}\ppi_{\mu B}^{\sharp}(\zeta-\varrho)$ and $Y^{\a+1}\mu L(\zeta-\varrho)$. In view of Table \ref{table1}, one has
\begin{align}
      [Y^{\a}\lpi_{\mu B}^{\sharp}]_{P.A.}&=[Y^{\a}(\eta^2\lpi_{T\hat{X}}\hat{X})]_{P.A.}=\frac{\eta^2}{\hat{X}^2}Y^{\a+1}\mu\hat{X},\label{toporderzeta1}\\
      [Y^{\a}\ypi_{\mu B}^{\sharp}]_{P.A.}&=[Y^{\a}(\mu\ypi_{L\hat{X}}\hat{X}+\eta^{2}\ypi_{T\hat{X}}\hat{X})]_{P.A.}
      =-\frac{\eta \hat{X}^1}{\hat{X}^2}Y^{\a+1}\mu\hat{X}\label{toporderzeta2}.
\end{align}
Therefore, it follows from \eqref{equationtchi} that the top order acoustical terms in $\mathfrak{F}_{|\a|+1}^{\zeta}$ is $Y^{\a-1}T\tr\chi P(\zeta-\varrho)$ with $|\a|=N_{top}=20$.
\subsection{Estimates on the lower order terms}\label{section44}
In this subsection, we derive both $L^2$ and $L^{\infty}$ estimates for the lower order terms.
\begin{lemma}\label{lotlinfty}
	For all $\a\leq N_{\infty}-2=11$ and sufficiently small $\da$, it holds that
	\begin{align}
		Y^{\a}\tr\chi,\ \da Y^{\a-1}T\tr\chi,\ \da Y^{\a+1}\mu\in O_{1}^{\leq \a+1}.
	\end{align}
\end{lemma}
\begin{proof}
	We prove this lemma by using an induction argument. Commuting $Y^{\a}$ with \eqref{equationchi} yields
	\begin{align}
		\begin{split}\label{Lyatrchi1}
		LY^{\a}\tr\chi&=Y^{\a}(e\tr\chi-\tr\a'+(\tr\chi)^2)+\sum_{\be_1+\be_2=\a-1}Y^{\be_1}\ypi_{L\hat{X}}\hat{X}Y^{\be_2}\tr\chi\\
		&=(e+2\tr\chi+O(\da))Y^{\a}\tr\chi+O(\fai)_{1}^{\leq \a+2},\\
		\end{split}
	\end{align}
	by an induction process, \eqref{curvaturea}, Lemma \ref{2ndff}, Table \ref{table1} and the bootstrap assumptions. Therefore, integrating \eqref{Lyatrchi1} along the integral curves of $L$ gives
 \[
 |Y^{\a}\tr\chi(t,u,\vartheta)|\les |Y^{\a}\tr\chi(0,u,\vartheta)|+O(\da)\les \da.
 \]
 In view of \eqref{equationtchi}, it suffices to prove $Y^{\a+1}\mu\in O_{0}^{\leq \a+1}$. To this end, commuting $Y^{\a+1}$ with \eqref{equationmu} and noticing $Y^{\a}\tr\chi\in O_{1}^{\leq \a+1}$ yield
	\begin{align}
		\begin{split}
			LY^{\a+1}\mu&=eY^{\a+1}\mu+\frac{3}{2}\hat{T}^iY^{\a+1}Tv^i+O(\fai)_{1}^{\leq \a+1}.\label{lyamu}
		\end{split}
	\end{align}
	This implies $|Y^{\a+1}\mu|\les 1$.
\end{proof}
\begin{remark}
	We write \eqref{Lyatrchi1} and \eqref{lyamu} more accurately for later use as follows. In view of \eqref{curvaturea} and Lemma \ref{LiTiXi}, the highest order terms in $Y^{\a}\tr\a'$ are 
	\begin{align}
 \begin{split}
		&Y^{\a}\sg^{-1}(\frac{1}{2}D_{XX}^2h-\frac{1}{2}D^2_{XX}\sum(v^k)^2+L^jD_{XX}^2v^j-X^jD_{LX}^2v^j)\\=
  &\frac{1}{2}Y^{\a}\hat{X}^2h+L^1Y^{\a}\hat{X}^2v^1+\hat{X}^2Y^{\a}L\hat{X}v^2+\text{l.o.ts}
  \end{split}
	\end{align}
 Hence,
	\begin{align}
		LY^{\a}\tr\chi&=(e+2\tr\chi+O(\da))Y^{\a}\tr\chi+\frac{3}{2}Y^{\a}\hat{X}^2\fai+O(\da)\cdot O(\fai)_{1}^{\leq\a+2}.\label{Lyatrchi}
	\end{align}
	Similarly, one has
	\begin{align}
		LY^{\a+1}\mu&=\frac{3}{2}\hat{T}^1Y^{\a+1}Tv^1+\frac{3}{2}\mu Y^{\a+1}L\fai+(e+O(\da))Y^{\a+1}\mu+O(\da)\cdot O(\fai)_{1}^{\leq\a+2}.\label{Lyamu}
	\end{align}
\end{remark}

As a consequence, it follows from Table \ref{table1} and an induction argument that $Y^{\a}Z\hat{T}^1\in O(\fai)_{1}^{\leq \a+1}$ for $\a\leq 11$. 
\begin{lemma}\label{L2commutator}
	For any $Z\in\mathcal{Z}$ and $\a\leq N_{top}=20$, the following bounds hold
	\begin{align}
 \begin{split}
		|[Z,Y^{\a}]O(\fai)_{1}^{\leq 1}|&=|\sum_{\be_1+\be_2=\a-1}Y^{\be_1}\ypi_{Z\hat{X}}\hat{X}Y^{\be_2}O(\fai)_{1}^{\leq 1}|\\
  &\les \da(\mathbf{1}_{\be_1>[\tfrac{\a}{2}]}Y^{\be_1}\ypi_{Z\hat{X}}+\mathbf{1}_{\be_2>[\tfrac{\a}{2}]}\hat{X}Y^{\be_2}O(\fai)_{1}^{\leq 1}),\label{commutator}
  \end{split}
	\end{align}
	where $\mathbf{1}_A$ denotes the characteristic function of a set $A$.
\end{lemma}
\begin{proof}
	The first equality in \eqref{commutator} comes directly from \eqref{commutator1}. 
	\begin{itemize}
		\item If $\be_1\geq[\tfrac{\a}{2}]+1$, then $\be_2\leq[\tfrac{\a}{2}]$. This implies $|\hat{X}Y^{\be_2}O(\fai)_{1}^{\leq 1}|\les\da$ due to the bootstrap assumptions.
	     \item If $\be_2\geq[\tfrac{\a}{2}]+1$, then $\be_1\leq[\tfrac{\a}{2}]\leq N_{\infty}-2$. It follows from Table \ref{table1} and Lemma \ref{lotlinfty} that $|Y^{\be_1}\ypi_{Z\hat{X}}|\les \da$.
	\end{itemize}
 This finishes the proof of  Lemma \ref{L2commutator}.
\end{proof}
\begin{lemma}\label{lotchiul2}
	For all $\a\leq N_{top}-1=19$ and sufficiently small $\da$, 
	\begin{align}
		\begin{split}
			\|Y^{\a}\tr\chi\|_{\ltwou}\leq C\|Y^{\a}\tr\chi\|_{L^2(\Sigma_0^u)}+\int_0^t(\tfrac{3}{2}+C\da)\mu_m^{-\frac{1}{2}}\sqrt{\mathbb{U}_{\leq\a+1}},
		\end{split}\label{yatrchi}\\
		\begin{split}
			\|Y^{\a+1}\mu\|_{\ltwou}\leq C\|Y^{\a+1}\mu\|_{L^2(\Sigma_0^u)}+\int_0^tC\sqrt{\mathbb{W}_{\leq\a+1}}+(\frac{3}{2}+C\da)\mu_m^{-\frac{1}{2}}\sqrt{\mathbb{U}_{\leq\a+1}}.
		\end{split}
	\end{align}
\end{lemma}
\begin{proof}
	We prove this lemma by an induction argument. It follows from Lemma \ref{L2commutator}, \eqref{Lyatrchi} and an induction process that 
 \begin{equation}\label{l2yachi}
	\begin{aligned}
		&\|Y^{\a}\tr\chi\|_{\ltwou}\\
  \leq & C\|Y^{\a}\tr\chi\|_{L^2(\Sigma_0^u)}+\int_0^tC\da\|Y^{\a}\tr\chi\|_{L^2(\Sigma_{t'}^u)}+ \left(\tfrac{3}{2}+C\da\right)\mu^{-\tfrac{1}{2}}_m\sqrt{\mathbb{U}_{\leq \a+1}(t',u)},
	\end{aligned}
 \end{equation}
	where the highest order energy comes from $Y^{\a}$ acting on $\tr\a'$. The desired estimate \eqref{yatrchi} is a direct consequence of integrating \eqref{l2yachi}. Similarly,
	\begin{align*}
		\|Y^{\a+1}\mu\|_{\ltwou}
		\leq & C\|Y^{\a+1}\mu\|_{L^2(\Sigma_0^u)}+\int_0^tC\da \|Y^{\a+1}\mu\|_{L^2(\Sigma_{t'}^u)}+C\sqrt{\mathbb{W}_{\leq\a+1}(t',u)}\\
  &+(\tfrac{3}{2}+C\da)\mu_m^{-\frac{1}{2}}\sqrt{\mathbb{U}_{\leq\a+1}(t',u)} \ dt'.
	\end{align*}
	This implies 
 \[
 \|Y^{\a+1}\mu\|_{\ltwou}\leq C\|Y^{\a+1}\mu\|_{L^2(\Sigma_0^u)}+\int_0^tC\sqrt{\mathbb{W}_{\leq\a+1}(t',u)}+(\tfrac{3}{2}+C\da)\mu_m^{-\frac{1}{2}}\sqrt{\mathbb{W}_{\leq \a+1}(t',u)}.
 \]
 This finishes the proof of the lemma.
\end{proof}
\section{Estimates on the top order acoustical terms}\label{section5}
In this section, the top order estimates involving $\mu$ and tr$\chi$ are obtained.
 To this end, let $b_1=\cdots=b_{14}=0,$ $b_{15}=\tfrac{3}{4}$, $b_{N+1}=b_N+1$ for $15\leq N\leq 19$ and $b_{21}=b_{20}+\tfrac{1}{2}$. 
 Define
\begin{align}
	\begin{split}\label{modifiedenergy}
	\bar{\mathbb{W}}_{N}&=\sup_{t'\in[0,t]}\{\mu_m^{2b_N}(t')\mathbb{W}_{N}\},\qquad \bar{\mathbb{U}}_{N}=\sup_{t'\in[0,t]}\{\mu_m^{2b_N}(t')\mathbb{U}_{N}\},\\
	\bar{\mathbb{Q}}_{N}&=\sup_{t'\in[0,t]}\{\mu_m^{2b_N}(t')\mathbb{Q}_{N}\},\qquad
	\bar{\mathbb{V}}_{N}=\sup_{t'\in[0,t]}\{\mu_m^{2b_N}(t')\mathbb{V}_{N}\}.
	\end{split}
\end{align}
We also set $\bar{\mathbb{W}}_{\leq N}=\sum_{N'\leq N}\bar{\mathbb{W}}_{N'}$ and similar for $\bar{\mathbb{Q}}_{\leq N},\
\bar{\mathbb{U}}_{\leq N},\ \bar{\mathbb{V}}_{\leq N+1}.$

\subsection{Top order estimates for angular derivatives of $\chi$} Note that \eqref{equationchi} can be written as
\begin{align}
	L(\mu \tr\chi)=2(L\mu)\tr\chi-\mu \tr\a-\mu(\tr\chi)^2.\label{equationmuchi}
\end{align}
In view of \eqref{curvaturea}, the second order derivative terms in $\mu \tr\a$ is 
\begin{align*}
	&\frac{1}{2}\mu\slashed{\Delta}h-\frac{1}{2}\mu\sum_i\slashed{\Delta}(v^i)^2+\mu L^i\slashed{\Delta}v^i-\mu\hat{X}^iL\hat{X}v^i\\
 =&\frac{1}{2}\mu\slashed{\Delta}h-\eta\hat{T}^i\mu\slashed{\Delta}v^i-\mu\hat{X}^iL\hat{X}v^i-\eta\hat{T}^i\mu(\hat{X}v^i)^2.
\end{align*}
It follows from \eqref{decomposition} that 
\begin{align*}
	&\frac{1}{2}\mu\slashed{\Delta}h-\eta\hat{T}^i\mu\slashed{\Delta}v^i-\mu\hat{X}^iLX^iv^i-\eta\hat{T}^i\mu(\hat{X}v^i)^2\\
 =&\frac{1}{2}L\dl h-\eta\hat{T}^iL\dl v^i-\mu\hat{X}^iL\hat{X}v^i+\tilde{\mathfrak{F}}_0^0 +\tilde{\mathfrak{F}}_0^i,
\end{align*}
where precisely, 
\begin{align}
	\begin{split}\label{f00f0i}
	\tilde{\mathfrak{F}}_0^0&=\sg^{-1}\tau Xh+\tfrac{1}{4}(\tr\chi\dl h+\tr\underline{\chi}Lh)+\tfrac{1}{2}\mathfrak{F}_0^0=O(\fai)_1^{\leq 1}+O(\psi)_1^{\leq 1},\\
	\tilde{\mathfrak{F}}_0^i&=2\sg^{-1}\eta\hat{T}^iXv^i+\tfrac{1}{2}(\tr\chi\dl v^i+\tr\underline{\chi}L v^i)-\eta\hat{T}^i\mu(\hat{X}v^i)^2+\mathfrak{F}_0^i=O(\fai)_1^{\leq 1}+O(\psi)_1^{\leq 1}.
	\end{split}
\end{align}
Therefore, \eqref{equationmuchi} can be regularized as
\begin{align}
	&L(\mu \tr\chi+\check{f})=2(L\mu)\tr\chi-\mu(\tr\chi)^2+\check{g},\label{equationmuchif}\\
	\begin{split}
	&\check{f}=\tfrac{1}{2}\dl h-\eta\hat{T}^i\dl v^i-\mu\hat{X}^i\hat{X}v^i,\\
    &\check{g}=-\mu \tr\a^{[N]}-L(\eta\hat{T}^i)\dl v^i-L(\mu\hat{X}^i)\hat{X}v^i-\tilde{\mathfrak{F}}_0^0-\tilde{\mathfrak{F}}_0^i=O(\fai)_1^{\leq 1}+O(\psi)_1^{\leq 1}.\\
    \end{split}
\end{align}
Define $\mathscr{F}_{\a}=\mu Y^{\a}\hat{X}\tr\chi+Y^{\a}\hat{X}\check{f}$ for $0\leq\a\leq  19$. To derive the equation of $\mathscr{F}_0$, one commutes $\hat{X}$ with \eqref{equationmuchif} to obtain
\begin{align*}
	L\mathscr{F}_0= &[L,\hat{X}](\mu \tr\chi+\check{f})-L(\hat{X}\mu\cdot \tr\chi)+\hat{X}(2(L\mu)\tr\chi-\mu(\tr\chi)^2+\check{g})\\
	=&(-3\tr\chi+2\mu^{-1}L\mu)\mathscr{F}_0+(\tfrac{3}{2}\tr\chi-2\mu^{-1}L\mu)\hat{X}\check{f}+\tr\chi\hat{X}(L\mu+\tfrac{1}{2}\check{f})\\
 &-
	\hat{X}\mu(L\tr\chi+(\tr\chi)^2)+\hat{X}\check{g}.
\end{align*}
Hence, one has
\begin{align}
	&L\mathscr{F}_0+(3\tr\chi-2\mu^{-1}L\mu)\mathscr{F}_0=(\frac{3}{2}\tr\chi-2\mu^{-1}L\mu)\hat{X}\check{f}+\check{g}_0,\label{equationf0}\\
	&\check{g}_0=\tr\chi\hat{X}(L\mu+\tfrac{1}{2}\check{f})-	\hat{X}\mu(L\tr\chi+(\tr\chi)^2)+\hat{X}\check{g}.
\end{align}
To compute the equation of $\mathscr{F}_{\a}$, we write the equation of $\mathscr{F}_0$ as
\begin{align}
	L\mathscr{F}_0=2L\mu\hat{X}\tr\chi-3\tr\chi\mathscr{F}_0+\frac{3}{2}\tr\chi\hat{X}\check{f}+\check{g}_0.\label{equationf01}
\end{align}
Commuting $Y^{\a}$ with \eqref{equationf01} and noting that $\mathscr{F}_{\a}=Y^{\a}\mathscr{F}_0-\sum_{\be_1+\be_2=\a,\be_1>0}Y^{\be_1}\mu Y^{\be_2}\hat{X}\tr\chi$ yield
\begin{align}
 L\mathscr{F}_{\a}+(3\tr\chi-2\mu^{-1}L\mu)\mathscr{F}_{\a}&=(\frac{3}{2}\tr\chi-2\mu^{-1}L\mu)Y^{\a}\hat{X}\check{f}+\check{g}_{\a},\label{LFa}
\end{align}
where 
\begin{align*}
	\check{g}_{\a}=&Y^{\a}\check{g}_0+
	\sum_{\be_1+\be_2=\a-1,\be_1>0}Y^{\be_1}\ypi_{L\hat{X}}\hat{X}Y^{\be_2}\mathscr{F}_0-3\tr\chi\sum_{\be_1+\be_2=\a,\be_1>0}Y^{\be_1}\mu Y^{\be_2}\hat{X}\tr\chi\\
&-3\sum_{\be_1+\be_2=\a,\be_1>0}Y^{\be_1}\tr\chi (Y^{\be_2}\mathscr{F}_0-\tfrac{1}{2}Y^{\be_2}\hat{X}\check{f})-L(\sum_{\be_1+\be_2=\a,\be_1>0}Y^{\be_1}\mu Y^{\be_2}\hat{X}\tr\chi).
\end{align*}
The key observation is that due to Proposition \ref{keymu}, in the shock region, $\mu^{-1}L\mu<0$. Therefore, in deriving the estimates of $\mathscr{F}_{\a}$ over the shock region, one can ignore the contribution from the term $(-\mu^{-1}L\mu)\mathscr{F}_{\a}$. Hence, it follows from \eqref{LFa} that
\begin{align*}
	\|\mathscr{F}_{\a}\|_{\ltwou}&\leq \|\mathscr{F}_{\a}\|_{L^2(\Sigma_0^u)}+\int_0^t\|(\tfrac{3}{2}|\tr\chi|+2\mu^{-1}|L\mu|)Y^{\a}\hat{X}\check{f}\|_{L^2(\Sigma_{t'}^u)}+\|\check{g}_{\a}\|_{L^2(\Sigma_{t'}^u)}\ dt'.
\end{align*}
\begin{lemma}\label{ltwofaibound}
Up to the commutation terms, it holds that 
\begin{align*}
	\mu_m^{\frac{1}{2}}\|XY^{\a}\hat{X}\fai\|_{\ltwou}+\|TY^{\a}\hat{X}\fai\|_{\ltwou}\leq \sqrt{\mathbb{W}_{\leq \a+1}(t,u)}+\sqrt{\mathbb{U}_{\leq\a+1}},\\
	\|LY^{\a}\hat{X}\fai\|_{L^2(C_u^t)}\leq \sqrt{\mathbb{Q}_{\leq \a+1}(t,u)}.
\end{align*}
\end{lemma}
 We first estimates $Y^{\a}\check{g}_0$. Note that 
 \[
 L\mu+\tfrac{1}{2}\check{f}=\mu e+O(\da)O(\fai)_{1}^{\leq 1}
 \]
 and 
 \[
 L\tr\chi+(\tr\chi)^2=e\tr\chi+O(\da)O(\fai)_1^{\leq 1}. 
 \]
 Then, in view of \eqref{f00f0i} and Lemma \ref{ltwofaibound}, one has
\begin{align}
	\int_0^t\|Y^{\a}\check{g}_0\|_{L^2(\Sigma_{t'}^u)}\les \da\int_0^t\int_0^{t'}(1+\mu_m^{-\frac{1}{2}})\sqrt{\mathbb{W}_{\leq\a+1}}+\sqrt{\int_0^u\mathbb{Q}_{\leq \a+1}}+\int_0^t\sqrt{\mathbb{V}_{\leq\a+2}},
\end{align}
where the flux term comes from $Y^{\a}\hat{X}\check{g}$ and 
\begin{align*}
&\int_0^t\|LY^{\a}\fai\|_{L^2(\Sigma_{t'}^u)}\ dt'\leq \int_0^t\sqrt{\int_0^u\int_{\vartheta'\in\mathbb{T}}(LY^{\a}\fai)^2\ d\vartheta'du'}dt'\\
\leq&\sqrt{\int_0^u\int_0^t\int_{\vartheta'\in\mathbb{T}}(LY^{\a}\fai)^2\ d\vartheta'dt'du'}
\leq \sqrt{\int_0^u\mathbb{Q}_{\leq \a+1}}.
\end{align*}
For the remaining terms in $\check{g}_{\a}$, it follows from Lemmas \ref{L2commutator}, \ref{lotchiul2}, and \ref{ltwofaibound} that 
\begin{equation}
\begin{aligned}
&\left\|\sum_{\be_1+\be_2=\a,|\be_1|>0}Y^{\be_1}\ypi_{L\hat{X}}\hat{X}Y^{\be_2}\mathscr{F}_0\right\|_{\ltwou}\\
\les & \da(\mathbf{1}_{\be_1\geq[\tfrac{\a}{2}]}\|Y^{\be_1}L\fai\|_{\ltwou}+\mathbf{1}_{\be_2\geq[\tfrac{\a}{2}]}\|Y^{\be_2}(\mu\hat{X}\tr\chi+X\check{f})\|_{\ltwou})\\
	\les & \da\sqrt{\mathbb{W}_{\leq\a+1}}+\da\mu_m^{-\frac{1}{2}}\sqrt{\mathbb{U}_{\leq\a+1}},\\
  \end{aligned}
 \end{equation}
 \begin{equation}
\begin{aligned}	
\left\|\tr\chi\sum_{\be_1+\be_2=\a,|\be_1|>0}Y^{\be_1}\mu Y^{\be_2}\hat{X}\tr\chi\right\|_{\ltwou}&\les \da\sqrt{\mathbb{W}_{\leq\a+1}}+\da\mu_m^{-\frac{1}{2}}\sqrt{\mathbb{U}_{\leq\a+1}},
  \end{aligned}
 \end{equation}
\begin{equation}
\begin{aligned}	 	\left\|\sum_{\be_1+\be_2=\a,|\be_1|>0}Y^{\be_1}\tr\chi Y^{\be_2}\hat{X}\check{f}\right\|_{\ltwou}
	\les & \da\sqrt{\mathbb{W}_{\leq\a+1}}+(1+\da)\mu_m^{-\frac{1}{2}}\sqrt{\mathbb{U}_{\leq\a+1}}.
 \end{aligned}
 \end{equation}
To estimate $||LY^{\be}\hat{X}\tr\chi||_{\ltwou}$ and $||LY^{\be+1}\mu||_{\ltwou}$ with $\be\leq\a-1\leq 18$, one utilizes the equations \eqref{Lyatrchi1} and \eqref{lyamu} to show that 
\begin{align}
	LY^{\be}\hat{X}\tr\chi&=(O(\fai)_1^{\leq 1}+\tr\chi)Y^{\be+1}\tr\chi+O(\fai)_1^{\leq\be+3},\\
	LY^{\be+1}\mu&=O(\fai)_1^{\leq 1}Y^{\be+1}\mu+O(\fai)_0^{\be+2}+O(\fai)_1^{\be+1}.
\end{align}
Therefore, it follows from Lemma \ref{lotchiul2} that 
\begin{align*}
\left\|L\left(\sum_{\be_1+\be_2=\a,\be_1>0}Y^{\be_1}\mu Y^{\be_2}\tr\chi\right)\right\|_{\ltwou}\les\int_0^t\mu_m^{-\frac{1}{2}}\sqrt{\mathbb{U}_{\leq\a+1}}+\da\sqrt{
\mathbb{W}_{\leq \a+1}}.
\end{align*}
Collecting above results yields 
\begin{align}
	\begin{split}\label{I2}
	\int_0^t\|\check{g}_{\a}\|_{L^2(\Sigma_{t'}^u)}\les & \int_0^t\int_0^{t'}\mu_m^{-\frac{1}{2}}\sqrt{\mathbb{U}_{\leq\a+1}}+\da\int_0^t\sqrt{\mathbb{W}_{\leq\a+1}}+\mu_m^{-\frac{1}{2}}\sqrt{\mathbb{U}_{\leq \a+1}}\\
	&+\sqrt{\int_0^u\mathbb{Q}_{\leq \a+1}}+\int_0^t\sqrt{\mathbb{V}_{\leq\a+2}}.
	\end{split}
\end{align}
To estimate the contribution from $\mu^{-1}|L\mu|Y^{\a}\hat{X}\check{f}$, we divide the integral into $W_s$ and $W_{n-s}$ as follows:
\begin{align*}
	&\int_0^t\|2\mu^{-1}|L\mu|)Y^{\a}\hat{X}\check{f}\|_{L^2(\Sigma_{t'}^u)}\\
 \leq &\int_0^t\|2\mu^{-1}|L\mu|Y^{\a}\hat{X}\check{f}\|_{L^2(\Sigma_{t'}^u\cap W_s)}+\int_0^t\|2\mu^{-1}|L\mu|Y^{\a}\hat{X}\check{f}\|_{L^2(\Sigma_{t'}^u\cap W_{n-s})}\\
	\leq &\int_0^t\|2\mu^{-1}|L\mu|Y^{\a}\hat{X}\check{f}\|_{L^2(\Sigma_{t'}^u\cap W_s)}+C\int_0^t\sqrt{\mathbb{U}_{\leq\a+1}}.
\end{align*}
To deal with the difficult ``shock region" part, one obtains from Lemma \ref{crucial} that
\begin{align}
\begin{split}
	\int_0^t\|2\mu^{-1}|L\mu|Y^{\a}\hat{X}\check{f}\|_{L^2(\Sigma_{t'}^u\cap W_s)}&\leq\int_0^t\mu_m^{-b_{\a+1}-1}\left(\sqrt{\bar{\mathbb{W}}_{\leq\a+1}}+\da\sqrt{\bar{\mathbb{U}}_{\leq\a+1}}\right)\ dt'\\
 &\leq \frac{2}{b_{\a+1}}\mu_m^{-b_{\a+1}}\left(\sqrt{\bar{\mathbb{W}}_{\leq\a+1}}+\da\mu_m^{\frac{1}{2}}\sqrt{\bar{\mathbb{U}}_{\leq\a+1}}\right).
 \end{split}\label{I1}
\end{align}
The point is that the constant in front of $\sqrt{\bar{\mathbb{W}}_{\leq\a+1}}$ cannot be large by choosing the top order derivatives suitably. Hence, combing the estimates \eqref{I2} and \eqref{I1} yields the following top order estimates for $\mu Y^{\a}\hat{X}\tr\chi$:
\begin{align}
	\begin{split}\label{toporderchi}
		&\|\mu Y^{\a}\hat{X}\tr\chi\|_{\ltwou}\\
  \leq & \|\mathscr{F}_{\a}\|_{L^2(\Sigma_0^u)}+C(
\|Y^{\a}\tr\chi\|_{L^2(\Sigma_0^u)}+\da\|Y^{\a+1}\mu\|_{L^2(\Sigma_0^u)})\\
 & +\frac{2}{b_{\a+1}}\mu_m^{-b_{\a+1}}(\sqrt{\bar{\mathbb{W}}_{\leq\a+1}}+\da\mu_m^{\frac{1}{2}}\sqrt{\bar{\mathbb{U}}_{\leq\a+1}})+C\int_0^t\int_0^{t'}\mu_m^{-\frac{1}{2}}\sqrt{\mathbb{U}_{\leq\a+1}}\\		
  &+C\da\int_0^t\sqrt{\mathbb{W}_{\leq\a+1}}+\mu_m^{-\frac{1}{2}}\sqrt{\mathbb{U}_{\leq \a+1}}+C\sqrt{\int_0^u\mathbb{Q}_{\leq \a+1}}+C\int_0^t\sqrt{\mathbb{V}_{\leq\a+2}}.
	\end{split}
\end{align}
\subsection{Top order estimates for spatial derivatives of $\chi$}
Let $\mathscr{H}_{\a}=\mu Y^{\a}T\tr\chi+Y^{\a}T\check{f}$ for $0\leq \a\leq 19$. To obtain the equation for $\mathscr{H}_0$, one commutes $T$ with \eqref{equationmuchif}
\begin{align*}
	L\mathscr{H}_0= &[L,T](\mu \tr\chi+\check{f})+TL(\mu \tr\chi+\check{f})-[(LT\mu)\tr\chi+(T\mu)L\tr\chi]\\
	=&(2\mu^{-1}L\mu-2\tr\chi)\mu T\tr\chi+\tpi_{L\hat{X}}(\mu\hat{X}\tr\chi+\hat{X}f)+(TL\mu)\tr\chi-(T\mu)[L\tr\chi+(\tr\chi)^2]+T\check{g}.
\end{align*}
Therefore,
\begin{align}
	L\mathscr{H}_0+(2\tr\chi-2\mu^{-1}L\mu)\mathscr{H}_0=(2\tr\chi-2\mu^{-1}L\mu)T\check{f}+\tpi_{L\hat{X}}\mathscr{F}_0+\check{l}_0,\label{equationh0}
\end{align}
where 
\[\check{l}_0=T\check{g}+(TL\mu)\tr\chi-(T\mu)[L\tr\chi+(\tr\chi)^2].\]
We rewrite \eqref{equationh0} as follows in order to derive the equation of $\mathscr{H}_{\a}$:
\begin{align}
	L\mathscr{H}_0=(2L\mu-2\mu \tr\chi)T\tr\chi+\tpi_{L\hat{X}}\mathscr{F}_0+\check{l}_0.
 \label{equationh01}
\end{align} 
Commuting $Y^{\a}$ with \eqref{equationh01} and noticing that $\mathscr{H}_{\a}=Y^{\a}\mathscr{H}_0-\sum_{\be_1+\be_2=\a,\ \be_1>0}Y^{\be_1}\mu Y^{\be_2}T\tr\chi$ yield
\begin{align}
	L\mathscr{H}_{\a}+(2\tr\chi-2\mu^{-1}L\mu)\mathscr{H}_{\a}&=
	(2\tr\chi-2\mu^{-1}L\mu)Y^{\a}T\check{f}+\tpi_{L\hat{X}}\mathscr{F}_{\a}+\check{l}_{\a},\label{LHa}
\end{align}
where

\begin{align*}
	\begin{split}
		&\check{l}_{\a}= [L,Y^{\a}]\mathscr{H}_0+Y^{\a}\check{l}_0+\sum_{\be_1+\be_2=\a,\be_1>0}Y^{\be_1}(2L\mu-2\mu \tr\chi)Y^{\be_2}T\tr\chi\\
		+&\sum_{\be_1+\be_2=\a,\be_1>0}(\tpi_{L\hat{X}}Y^{\be_1}\mu Y^{\be_2}\hat{X}\tr\chi+Y^{\be_1}\tpi_{L\hat{X}}Y^{\be_2}\mathscr{F}_0)+L(\sum_{\be_1+\be_2=\a,\be_1>0}Y^{\be_1}\mu Y^{\be_2}T\tr\chi).
	\end{split}
\end{align*}
It follows from \eqref{LHa} that 
\begin{equation*}
\begin{aligned}
	&\|\mathscr{H}_{\a}\|_{\ltwou}\\
 \leq & \|\mathscr{H}_{\a}\|_{L^2(\Sigma_0^u)}
 +\int_0^t||(\tfrac{3}{2}|\tr\chi|+2\mu^{-1}|L\mu|)Y^{\a}T\check{f}||_{L^2(\Sigma_{t'}^u)}	+||\mathscr{F}_{\a}||_{L^2(\Sigma_{t'}^u)}+||\check{l}_{\a}||_{L^2(\Sigma_{t'}^u)}\ dt'.
\end{aligned}
\end{equation*}

It follows from \eqref{equationvorticity} that 
\begin{align*}
	TL(\zeta-\p)&=LT(\zeta-\p)+[T,L](\zeta-\p)=-L(\eta^{-2}\mu L(\zeta-\p))-\tpi_{L\hat{X}}\hat{X}(\zeta-\p),\\
	TY(\zeta-\p)&=-Y(\eta^{-2}\mu L(\zeta-\p))+\ypi_{T\hat{X}}\hat{X}(\zeta-\p).
\end{align*}
Thus one has
\begin{align*}	\int_0^t\da\|Y^{\a}T\check{g}\|_{L^2(\Sigma_{t'}^u)}\les\int_0^t\sqrt{\mathbb{V}_{\leq\a+2}}+\sqrt{\mathbb{W}_{\leq\a+1}}\ dt'+\sqrt{\int_0^u\mathbb{Q}_{\leq \a+1}}.
\end{align*}
Since $TL\mu=T^2\fai+TL\fai+\text{l.o.ts}$, it holds that
\begin{equation*}
\begin{aligned}
    \da\|Y^{\a}TL\mu\|_{\ltwou}\leq  &\da\|TY^{\a}T\fai\|_{\ltwou}+\da\|TY^{\a}L\fai\|_{\ltwou}+\text{l.o.ts}\\
    \les &(1+\da) \sqrt{\mathbb{W}_{\leq\a+1}}+\da\mu_m^{-\frac{1}{2}}\sqrt{\mathbb{U}_{\leq\a+1}}.
    \end{aligned}
\end{equation*}
Therefore,
\begin{align}
\int_0^t\da\|Y^{\a}\check{l}_0\|_{L^2(\Sigma_{t'}^u)}\les \int_0^t\sqrt{\mathbb{V}_{\leq\a+2}}+\sqrt{\mathbb{W}_{\leq\a+1}}\ dt'+\sqrt{\int_0^u\mathbb{Q}_{\leq \a+1}}.
\end{align}
Note that in view of \eqref{equationtchi} and Lemma \ref{lotchiul2}, for $\be\leq 18$, one has
\begin{align*}
	\|Y^{\be}T\tr\chi\|_{\ltwou}&\leq\|Y^{\be+2}\mu\|_{\ltwou}+\mu_m^{-\frac{1}{2}}\sqrt{\mathbb{W}_{\leq\be+1}}\les
\mu_m^{-\frac{1}{2}}\sqrt{\mathbb{W}_{\leq\be+1}}+\int_0^t\mu_m^{-\frac{1}{2}}\sqrt{\mathbb{W}_{\leq\be+1}},\\
	LY^{\be}T\tr\chi&=(\tr\chi+O(\fai)_1^{\leq 1})Y^{\be}T\tr\chi+TY^{\be}O(\fai)_1^{\leq 2}+\text{l.o.ts}.
\end{align*}
Therefore,
\begin{align*}
	&\da\left\|\sum_{\be_1+\be_2=\a,\be_1>0}Y^{\be_1}(2L\mu-2\mu \tr\chi)Y^{\be_2}T\tr\chi\right\|_{\ltwou}\les \da\mu_m^{-\frac{1}{2}}\sqrt{\mathbb{W}_{\leq\a+1}}+\int_0^t\mu_m^{-\frac{1}{2}}\sqrt{\mathbb{W}_{\leq\a+1}},\\
	&\da\left\|L\left(\sum_{\be_1+\be_2=\a,\be_1>0}Y^{\be_1}\mu Y^{\be_2}T\tr\chi\right)\right\|_{\ltwou}\leq(C+\da\mu_m^{-\frac{1}{2}})\sqrt{\mathbb{W}_{\leq\a+1}}+\int_0^t\da\mu_m^{-\frac{1}{2}}\sqrt{\mathbb{W}_{\leq \a+1}}.
\end{align*}
Collecting the above results yields
\begin{equation*}
\begin{aligned}
\int_0^t\|\check{l}_{\a}\|_{L^2(\Sigma_{t'}^u)}\les & \da +\int_0^t\sqrt{\mathbb{V}_{\leq\a+2}}+(1+\mu_m^{-\frac{1}{2}})\sqrt{\mathbb{W}_{\leq\a+1}}\\
 &+\int_0^t\int_0^{t'}\mu_m^{-\frac{1}{2}}\sqrt{\mathbb{W}_{\leq \a+1}}+\sqrt{\int_0^u\mathbb{Q}_{\leq \a+1}}.
\end{aligned}
\end{equation*}
To estimated the contribution from $\mu^{-1}|L\mu|Y^{\a}T\check{f}$, we do similar as before and obtain
\begin{align*}
	\int_0^t\da\|2\mu^{-1}|L\mu|Y^{\a}T\check{f}\|_{L^2(\Sigma_{t'}^u)}
	&\leq\int_0^t\da\|2\mu^{-1}|L\mu|Y^{\a}T\check{f}\|_{L^2(\Sigma_{t'}^u\cap W_s)}+C\int_0^t\sqrt{\mathbb{U}_{\leq\a+1}}.
\end{align*}
Thus, it follows from Lemma \ref{crucial} that 
\begin{align}
	\int_0^t\da\|2\mu^{-1}|L\mu|Y^{\a}T\check{f}\|_{L^2(\Sigma_{t'}^u\cap W_s)}\leq\frac{2}{b_{\a+1}} \mu_m^{-b_{\a+1}}\left(\sqrt{\bar{\mathbb{W}}_{\leq\a+1}}+\da\mu_m^{\frac{1}{2}}\sqrt{\mathbb{U}_{\leq\a+1}}\right).
\end{align}
Collecting the above results and noticing \eqref{toporderchi} yield the top order estimate for $\mu Y^{\a}T\tr\chi$: 
\begin{align}
	\begin{split}\label{topordermu}
		\da\|\mu Y^{\a}T\tr\chi\|_{\ltwou}\leq &\da\|\mathscr{H}_{\a}\|_{L^2(\Sigma_0^u)}+C\da+\frac{4}{b_{\a+1}}\mu_m^{-b_{\a+1}}\left(\sqrt{\bar{\mathbb{W}}_{\leq\a+1}}+\da\mu_m^{\frac{1}{2}}\sqrt{\mathbb{U}_{\leq\a+1}}\right)\\
  &+C\int_0^t\int_0^{t'}(1+\da^{\frac{1}{2}}\mu_m^{-\frac{1}{2}})\sqrt{\mathbb{W}_{\leq\a+1}}+C\int_0^t(1+\da^{\frac{1}{2}}\mu_m^{-\frac{1}{2}})\sqrt{\mathbb{W}_{\leq \a+1}}\\
  &+C\sqrt{\int_0^u\mathbb{Q}_{\leq \a+1}}+C\int_0^t\sqrt{\mathbb{V}_{\leq\a+2}}.
	\end{split}
\end{align}
Hence, the estimate for $\mu Y^{\a}\slashed{\Delta}\mu$ follows from \eqref{topordermu} and \eqref{equationtchi}.
\subsection{The inhomogeneous terms' estimates after commutation}\label{section53} Note that in Section \ref{toporderterms}, we have computed the top order acoustical terms in $\mathfrak{F}_n$. In this subsection, we investigate the complicated terms in $\mathfrak{F}_n$ and one can see why we define the high order energies-fluxes in Definition \ref{highorderenergy}. 
\begin{lemma}\label{inhomo1}
	 Let $\a=\a'+2\leq20$ and $Z^{\a}=Z_{n-1}\cdots Z_1=Y^{\a'}TL$. Then we have
	 	\begin{equation}\label{zaf00} 	
	 		\begin{aligned}	&\int_0^t\|Z^{\a}\mathfrak{F}_0^0\|_{L^2(\Sigma_{t'}^u)}\\
    \leq & C\da\int_0^t\|TZ^{\a}\fai\|_{L^2(\Sigma_{t'}^u)}+C\da\int_0^t\|\dl Z^{\a}\fai\|_{L^2(\Sigma_{t'}^u)}+C\da\int_0^t\|\mu\hat{X}Z^{\a}\fai\|_{L^2(\Sigma_{t'}^u)}\\	 		&+C\sqrt{\int_0^u\|LZ^{\a}\fai\|_{L^2(C_{u'}^t)}}+\int_0^t\sqrt{\mathbb{V}_{\leq\a+1}}+C\int_0^u\da^{-1}\sqrt{\mathbb{V}_{\leq\a+1}}+C\int_0^t\mu_m^{-\frac{1}{2}}\sqrt{\mathbb{U}_{\leq\a}},
	 	\end{aligned}
   \end{equation}
   \begin{equation}\label{zaf01}
	 	\begin{aligned} 	\int_0^t\|Z^{\a}\mathfrak{F}_0^1\|_{L^2(\Sigma_{t'}^u)}
   \leq  &\int_0^t\da^{-1}\sqrt{\mathbb{V}_{\leq\a+1}}+C\da\int_0^t\|\dl Z^{\a}\fai\|_{L^2(\Sigma_{t'}^u)}+\|\mu\hat{X}Z^{\a}\fai\|_{L^2(\Sigma_{t'}^u)}\ dt'\\
	 		&+(C+\da)\sqrt{\int_0^u\|LZ^{\a}\fai\|_{L^2(C_{u'}^t)}}+C\int_0^t\mu_m^{-\frac{1}{2}}\sqrt{\mathbb{U}_{\leq\a}},
	 	\end{aligned}
   \end{equation}
   \begin{equation}\label{zaf02}
	 	\begin{aligned}
	 		\int_0^t\|Z^{\a}\mathfrak{F}_0^2
    \|_{L^2(\Sigma_{t'}^u)}\leq & C\int_0^t\sqrt{\mathbb{V}_{\leq\a+1}}+C\da\int_0^t\|\dl Z^{\a}\fai\|_{L^2(\Sigma_{t'}^u)}+\|\mu\hat{X}Z^{\a}\fai\|_{L^2(\Sigma_{t'}^u)}\ dt'\\
&+C\sqrt{\int_0^u\|LZ^{\a}v^2\|_{L^2(C_{u'}^t)}}+\da\sqrt{\int_0^u
\|LZ^{\a}\p\|_{L^2(C_{u'}^t)}}\\
    &+C\int_0^t\mu_m^{-\frac{1}{2}}\sqrt{\mathbb{U}_{\leq\a}}+\int_0^u\da^{-2}\sqrt{\mathbb{V}_{\leq\a+1}}.
	 	\end{aligned}
	 	\end{equation}
\end{lemma}
\begin{remark}
	Similar results hold for $Z^{\a}=Y^{\a'}TL$ replaced by $Z^{\a}=Y^{\a}$ and $Z^{\a}=Y^{\a'+1}T$.
\end{remark}
\begin{proof}
	It suffices to consider contribution from the top order terms. The major point is that when $Z^{\a}$ acts on the product terms in $\mathfrak{F}_0^0$ or $\mathfrak{F}_0^i$, by Leibniz rules and the bootstrap assumptions, we handle the resulting terms as follows: if the lower order terms are bounded by $C$, one uses the fluxes to control the highest order terms; while if the lower order terms are small, one uses the energies to bounded the highest order terms. For $Z^{\a}\mathfrak{F}_0^0$, it holds that
	\begin{align}
 \begin{split}
		Z^{\a}\mathfrak{F}_0^0= &O(\fai)_1^{1}(TY^{\a}TL\fai+\mu\hat{X}Y^{\a'}TL\fai)+O(\fai)_0^{1}(LY^{\a'}T\hat{X}\fai+LY^{\a'}TL\fai)\\
  &+O(1)\mu Y^{\a'}L^2(\zeta-\varrho)+\text{l.o.ts},
  \end{split}\label{zaf001}
	\end{align}
	where the vorticity term comes from \[Y^{\a'}TL\zeta=-\eta^{-2}\mu Y^{\a'}L^2(\zeta-\varrho)+Y^{\a'}TL\p+\text{l.o.ts}.\] The energy $\mu_m^{-\frac{1}{2}}\sqrt{\mathbb{U}_{\leq\a}}$ comes from the commutators
	 \begin{align*}
	 	\begin{split}
	&O(\da)\|[Y^{\a'}TL,T]\fai\|_{\ltwou}+\|[Y^{\a'}TL,L]\fai\|_{\ltwou}\\
	\leq & C\da\|\hat{X}Z^{\a-1}\fai\|_{\ltwou}+C\da\|\hat{X}Y^{\a'+1}T\fai\|_{\ltwou}+C\|\hat{X}Y^{\a'+1}L\fai
 \|_{\ltwou}\leq C\mu_m^{-\frac{1}{2}}\sqrt{\mathbb{U}_{\leq\a}}.
	\end{split}
	\end{align*} Then, \eqref{zaf00} comes directly from by integrating \eqref{zaf001} over $W_u^t$. For the $\pa(\zeta-\varrho)$ in $\mathfrak{F}_0^i$, in view of the homogeneous equation \eqref{equationvorticity}, one has
	\begin{align*}
		Z^{\a}(-\ep_{ia}\pa_a(\zeta-\varrho))&=-\ep_{ia}\hat{T}^a\eta^{-3}\mu^2Y^{\a'}L^3(\zeta-\varrho)-\ep_{ia}\eta^{-2}\mu^2\frac{\hat{X}^a}{\hat{X}^2}Y^{\a'+1}L^2(\zeta-\varrho)+\text{l.o.ts}.
	\end{align*}
	Therefore, it holds that
	\begin{align}
		\begin{split}
		Z^{\a}\mathfrak{F}_0^1&=O(\da)\mu^2 Y^{\a'}L^3(\zeta-\varrho)+O(1)\mu Y^{\a'+1}L^2(\zeta-\varrho)\\
  &+O(\fai)_1^{\leq1}(\dl Z^{\a}\fai+\mu\hat{X}Z^{\a}\fai)+O(\fai)_0^{\leq1}LZ^{\a}\fai+\text{l.o.ts},
		\end{split}\label{zaf0i1}\\
		\begin{split}
			Z^{\a}\mathfrak{F}_0^2&=O(1)\mu^2  Y^{\a'}L^3(\zeta-\varrho)+O(\da)\mu Y^{\a'+1}L^2(\zeta-\varrho)\\
   &+O(\fai)_1^{\leq1}(\dl Z^{\a}\fai+\mu\hat{X}Z^{\a}\fai)+O(1)LZ^{\a}v^2+O(\da)LZ^{\a}\p+\text{l.o.ts}.
		\end{split}\label{zaf0i2}
	\end{align}
	Similarly, the energy $\mu_m^{-\frac{1}{2}}\sqrt{\mathbb{U}_{\leq\a}}$ comes from the commutators. Then, \eqref{zaf01} and \eqref{zaf02} are the conseequence of \eqref{zaf0i1}, \eqref{zaf0i2} and Definition \ref{highorderenergy}.
\end{proof}

Next, we turn to the summation terms 
\[
\mathfrak{G}_n=\sum_{k=0}^{n-2}(Z_{n-1}+\prescript{Z_{n-1}}{}{\da})\cdots (Z_{n-k}+\prescript{Z_{n-k}}{}{\da})\prescript{Z_{n-k-1}}{}{\sigma}_{n-k-1,1},
\]
where the top order terms are contained in $Z_{n-1}\cdots Z_2\prescript{Z_1}{}{\sigma}_{1,2}$. To this end, define
\begin{align}
	\mathbb{K}_N&=\sup_{(t',u')\in[0,t]\times[0,u]}\left\{
	\sum_{Z^{\a}\in\mathcal{Z}^{N;\leq 1}}\da^{2m}K(Z^{\a}\fai)\right\},\quad \bar{\mathbb{K}}_{N}=\sup_{t'\in[0,t]}\{\mu_m^{2b_N}(t')\mathbb{K}_{N}\}
\end{align}
and $\mathbb{K}_{\leq N}=\sum_{N'\leq N}\mathbb{K}_{N'},$ $\bar{\mathbb{K}}_{\leq N}=\sum_{N'\leq N}\bar{\mathbb{K}}_{N'}.$
\begin{lemma}\label{inhomo2}
	Let $|\a|=20$ and $Z^{\a}=Z_{n-1}\cdots Z_1=Y^{\a'}T^mL^p$ with $m\leq1,\ p\leq 1$. Then,
	\begin{align*}
		\begin{split}
				\int_{W_t^u}\da^{2m}\mathfrak{G}_n\cdot K_0(Z^{\a}\fai)\leq & (C+\da^{\frac{1}{2}})\int_0^t\mathbb{W}_{\leq\a+1}(t',u)\ dt'+\frac{1}{5}\mathbb{K}_{\leq\a+1}
    \\
   & +C\da^{\frac{1}{2}}\int_0^t\mu_m^{-1}\mathbb{U}_{\leq\a+1} +C\da^{-1}\int_0^u\mathbb{Q}_{\leq\a+1}(t,u'),
		\end{split}\\
		\begin{split}
		\int_{W_t^u}\da^{2m}\mathfrak{G}_n\cdot K_1(Z^{\a}\fai)\leq & \da^{\frac{1}{2}}\int_0^t\mathbb{W}_{\leq\a+1}(t',u)\ dt'+C\int_0^t\mathbb{U}_{\leq\a+1}(t',u)\ dt'\\
  &+\da^{\frac{1}{2}}\mathbb{K}_{\leq\a+1}+C\da^{-\frac{1}{2}}\int_0^u\mathbb{Q}_{\leq\a+1}(t,u')\ du',
		\end{split}
	\end{align*}
	where $m\leq1$ is the number of $T$ in $Z^{\a}$.
\end{lemma}
\begin{proof}
It suffices to consider the contribution from $Z_{n-1}\cdots Z_2\prescript{L}{}{\sigma}_{1,1}$ and the remaining terms are lower order terms.
\begin{itemize}
	\item For $Z^{\a}=Z_{n-1}\cdots Z_1=Y^{\a'}TL$, it holds that
\begin{align}
	\begin{split}\label{ytlsigma}
	Z_{n-1}\cdots Z_2\prescript{L}{}{\sigma}_{1,1}&=\tr\chi\dl Y^{\a'}L^2\fai+O(\fai)_0^{\leq1}LZ^{\a}\fai+O(\fai)_0^{\leq 1}LY^{\a'+1}T\fai\\
	&+\boxed{(-L\mu+\mu \tr\chi)\hat{X}Y^{\a'+1}T\fai}+\text{l.o.ts}.
	\end{split}
\end{align}
\item For $Z^{\a}=Z_{n-1}\cdots Z_1=Y^{\a'+1}T$, it holds that
\begin{align}
	\begin{split}\label{ytsigma}
	Z_{n-1}\cdots Z_2\prescript{T}{}{\sigma}_{1,1}&=\kappa \tr\theta\dl Y^{\a'+1}L\fai+T\mu LY^{\a'+1}L\fai+\mu(\tau+\sigma)LY^{\a}\fai\\
 &+(\tau+\sigma)\dl Y^{\a}\fai+\boxed{(-T\mu+\mu\kappa \tr\chi)\hat{X}Y^{\a}\fai}+\text{l.o.ts}.
	\end{split}
\end{align}
\item For $Z^{\a}=Z_{n-1}\cdots Z_1=Y^{\a}$, it holds that
\begin{align}
	\begin{split}\label{ysigma}
	Z_{n-1}\cdots Z_2\prescript{Y}{}{\sigma}_{1,1}&=\tr\chi \hat{X}^1\dl Y^{\a'+1}L\fai+Y\mu LY^{\a'+1}L\fai+O(\fai)_1^{\leq1}Y^{\a}L\fai\\
 &+O(\fai)^{\leq1}_1\dl Y^{\a}\fai+\boxed{(-Y\mu+\mu \tr\chi\hat{X}^1)\hat{X}Y^{\a}\fai}+\text{l.o.ts}.
	\end{split}
\end{align}
\end{itemize}
For the boxed terms in \eqref{ytlsigma}-\eqref{ysigma} which might be possible singular, one fully uses the structure of $K(Z^{\a}\fai)$ as follows:
\begin{align*}
	\begin{split}	&\int_{W_t^u}\da^2|T\mu|\hat{X}Y^{\a}\fai\cdot\dl Y^{\a'+1}T\fai\\
 \leq &
 \left(\int_{W_s}(\hat{X}Y^{\a}\fai)^2+\int_{W_{n-s}}(\hat{X}Y^{\a}\fai)^2\right)^{\frac{1}{2}}\left(\int_0^t\da^2\|\dl Y^{\a'+1}T\fai\|^2_{L^2(\Sigma_{t'}^u)}\right)^{\frac{1}{2}}\\
	\leq&\frac{1}{10}\mathbb{K}_{\leq\a+1}+C\int_0^t\|\mu\hat{X}Y^{\a}\fai\|^2_{L^2(\Sigma_{t'}^u)}+\da^2\|\dl Y^{\a'+1}T\fai\|_{L^2(\Sigma_{t'}^u)}^2\leq\frac{1}{10}\mathbb{K}_{\leq\a+1}+C\int_0^t\mathbb{W}_{\a+1},
	\end{split}
 \end{align*}
 \begin{align*}
	\begin{split}
	\int_{W_t^u}|Y\mu|\hat{X}Y^{\a}\fai\cdot\dl Y^{\a}\fai&\leq\left(\int_{W_s}(\hat{X}Y^{\a}\fai)^2+\int_{W_{n-s}}(\hat{X}Y^{\a}\fai)^2\right)^{\frac{1}{2}}\left(\int_0^t\da^2||\dl Y^{\a'+1}T\fai||^2_{L^2(\Sigma_{t'}^u)}\right)^{\frac{1}{2}}\\
	&\leq \frac{1}{10}\mathbb{K}_{\leq\a+1}+C\int_0^t\mathbb{W}_{\a+1}.
	\end{split}
\end{align*}
For estimates associated with $K_1$, one uses flux to bound as follows:
\begin{align*}
	\begin{split}
		&\int_{W_t^u}|Y\mu|\hat{X}Y^{\a}\fai\cdot (1+\eta^{-2}\mu)LY^{\a}\fai\\
  \leq & C\left(\int_{W_s}(\hat{X}Y^{\a}\fai)^2+\int_{W_{n-s}}(\hat{X}Y^{\a}\fai)^2\right)^{\frac{1}{2}}\left(\int_0^u||LY^{\a}\fai||_{L^2(C_{u'}^t)}\right)^{\frac{1}{2}}
		+\int_{W_t^u}\mu \hat{X}Y^{\a}\fai\cdot LY^{\a}\fai\\
  \leq & C\da^{\frac{1}{2}}\mathbb{K}_{\leq\a+1}+C\int_0^t\mathbb{U}_{\leq\a+1}+C\da^{-\frac{1}{2}}\int_0^u\mathbb{Q}_{\leq\a+1}.
	\end{split}
\end{align*}
Therefore, for commutators $Z^{\a}$ listed above, one obtains the following estimates:
\begin{align*}
	&\int_{W_t^u}\da^{2m}|Z_{n-1}\cdots Z_2\prescript{Z}{}{\sigma}_{1,1}\cdot K_0(Z^{\a}\fai)|\\
	\leq &(C+\da)\int_0^t\da^{2m}\|\dl Z^{\a}\fai\|^2_{L^2(\Sigma_{t'}^u)}\\
 &+\da^{-1}\int_0^u\da^{2m}\|LZ^{\a}\fai\|^2_{L^2(C_{u'}^t)}+\frac{1}{5}\mathbb{K}_{\leq\a+1}+\da\int_0^t\da^{2m}\|\mu\hat{X}Z^{\a}\fai\|^2_{L^2(\Sigma_{t'}^u)}\\
 &+\da^{\frac{1}{2}}\int_0^t\mu_m^{-1}\|\mu^{\frac{1}{2}}\hat{X}Y^{\a}\fai\|_{L^2(\Sigma_{t'}^u)}^2+\da^{\frac{1}{2}}\int_0^t\|\dl Y^{\a}\fai\|_{L^2(\Sigma_{t'}^u)}^2+\text{l.o.ts},
\end{align*}
where $m\leq1$ is the number of $T$ in $Z^{\a}$. Similarly,
\begin{align*}
	&\int_{W_t^u}\da^{2m}|Z_{n-1}\cdots Z_{2}\prescript{Z}{}{\sigma}_{1,1}\cdot K_1(Z^{\a}\fai)|\\
	\leq & (C+\da^{\frac{1}{2}})\int_0^t\da^{2m}\mu\|LZ^{\a}\fai\|_{L^2(\Sigma_{t'}^u)}^2\\
& +(\da^{-\frac{1}{2}}+\da^{\frac{1}{2}})\int_0^u\da^{2m}\|LZ^{\a}\fai\|_{L^2(C_{u'}^t)}+\da\int_0^t\da^{2m}\|\dl Z^{\a}\fai\|_{L^2(\Sigma_{t'}^u)}^2\\
 &+\da^{\frac{1}{2}}\mathbb{K}_{\leq \a+1}+(C+\da^{\frac{1}{2}})\int_0^t\|\mu\hat{X}Y^{\a}\fai\|_{L^2(\Sigma_{t'}^u)}^2+\text{l.o.ts},
\end{align*}
where $m\leq1$ is the number of $T$ in $Z^{\a}$. Thus the Lemma follows from the above estimates and Definition \ref{highorderenergy}.
\end{proof}
\begin{lemma}\label{commutatorvorticity}(\textbf{Commutators' estimates for the vorticity equation})
	Let $|\a|=20$ and $P_i^{\a+1}=Y^{\a'}L^p$ with $\a'+p=21$ and $p\leq3$. Denote \[\mathfrak{G}^{\zeta}_{|\a|+1}=[(P\mu)L,P_i^{\a}](\zeta-\varrho)+[(P\eta^2)T,P_i^{\a}](\zeta-\varrho)\] as given in \eqref{commutationvorticity}. Then,
	\begin{align}
		\begin{split}
		&\int_{W_t^u}\da^{2(p-1)_{+}}P_i^{\a+1}(\zeta-\varrho) \cdot\mathfrak{G}_{|\a|+1}^{\zeta}\\
  \leq& \da\int_0^t\mathbb{W}_{\leq\a+1}+\da^2\int_0^t\int_0^s\mu_m^{-1}\mathbb{W}_{\leq\a}+\da\int_0^u\mathbb{Q}_{\leq\a}
		+\da\int_0^t\mu_m^{-1}\mathbb{V}_{\leq\a+1}(t',u)\ dt'\\
  &+\da^{-\frac{1}{2}}\int_0^u\mathbb{V}_{\leq\a+1}(t,u')\ du'+\da^{\frac{1}{2}}\int_0^t\mu_m^{-2}\|\mu Y^{\a-1}\slashed{\Delta}\mu\|^2_{L^2(\Sigma_{t'}^u)}.
		\end{split}\label{Pa1zeta}
	\end{align}
\end{lemma}
\begin{proof}
	We first consider the case with $p=0$. It follows from \eqref{commutationvorticity} that the top order terms are
	\begin{align}
		Y^{\a+1}\mu\cdot L(\zeta-\varrho),\quad
		 Y^{\a+1}(\eta^2)\cdot T(\zeta-\varrho), \label{commutationyatop}
	\end{align}
	and the highest order terms involving vorticity are
	\begin{align}
		[(Y\mu)\prescript{Y}{}{\pi}_L^{\sharp}+(Y\eta^2)\prescript{Y}{}{\pi}_T^{\sharp}]Y^{\a-1}(\zeta-\varrho),\ Y^2\mu Y^{\a-1}L(\zeta-\varrho),\ Y^2(\eta^2) Y^{\a-1}T(\zeta-\varrho). \label{commutationyanext}
	\end{align}
	Hence, it follows from the bootstrap assumption \eqref{bs} that for the contribution from \eqref{commutationyatop} with $P_i^{\a+1}=Y^{\a+1}$, one obtains
	\begin{align*}
		&\int_{W_t^u}Y^{\a+1}(\zeta-\varrho) \cdot [Y^{\a+1}\mu\cdot L(\zeta-\varrho)+Y^{\a+1}(\eta^2)\cdot T(\zeta-\varrho)]\\
  \leq\da&\int_{W_t^u}
		Y^{\a+1}(\zeta-\varrho) \cdot ( Y^{\a-1}\slashed{\Delta}\mu+\mu\hat{X}Y^{\a}\fai)\\
		\leq \da&\int_0^t\mu_m^{-1}\mathbb{V}_{\leq\a+1}(t',u)\ dt'+\da\int_0^t\mu_m^{-1}\mathbb{W}_{\leq\a}+\da^{-\frac{1}{2}}\int_0^u\mathbb{V}_{\leq\a}\cdot \da^{\frac{1}{2}}\int_0^t\mu_m^{-2}\| \mu Y^{\a-1}\slashed{\Delta}\mu\|_{L^2(\Sigma_{t'}^u)}^2.
	\end{align*}
	For terms in \eqref{commutationyanext}, in view of Lemma \ref{preliminary1} and Table \ref{table2}, it suffices to consider the contribution from $Y^2\mu Y^{\a-1}L(\zeta-\varrho)$. Thus, it holds that
	\begin{align}
	&\int_{W_t^u}Y^{\a+1}(\zeta-\varrho) \cdot [Y^2\mu Y^{\a-1}L(\zeta-\varrho)]\leq \da^{-\frac{1}{2}}\int_0^u\mathbb{V}_{\leq\a}+\da^{\frac{1}{2}}\int_0^u\mathbb{V}_{\leq\a+1}.
	\end{align}
	Therefore, for the case $P_i^{\a+1}=Y^{\a+1}$, \eqref{Pa1zeta} holds directly from above analysis. For the case $P_i^{\a+1}=Y^{\a'}L^p$ where $1\leq p\leq3$, one can always act $L$ on $\mu$ first since the commutator $[L,P_i^{\a}]$ is a lower order term compared with $P_i^{\a}L$. Hence, it suffices to consider the contribution from 
	\begin{align}
			P_i^{\a}L\mu\cdot L(\zeta-\varrho),\quad P_i^{\a}L(\eta^2)\cdot T(\zeta-\varrho),\quad P^2\mu P_i^{\a-1}L(\zeta-\varrho).
	\end{align} 
	In view of \eqref{equationmu}, one writes \[P_i^{\a}L\mu=TP_i^{\a}\fai+\mu LP_i^{\a}\fai+O(\fai)_{1}^{\leq1}P_i^{\a}\mu+\text{l.o.ts}.\] Note also that the number of $L$ in $P_i^{\a}$ is at least $p-1$. Then, it follows from Lemma \ref{lotchiul2} that 
	\begin{align*}
		&\int_{W_t^u}\da^{2(p-1)_+}P_i^{\a+1}(\zeta-\varrho)\cdot P_i^{\a}L\mu\cdot L(\zeta-\varrho)\\
  \leq& \da\int_0^t\mathbb{W}_{\leq\a}+\da^2\int_0^t\int_0^s\mu_m^{-1}\mathbb{W}_{\leq\a}+\da\int_0^u\mathbb{Q}_{\leq\a}+\mathbb{V}_{\leq\a+1},
		\end{align*}
  \begin{align*}
		\int_{W_t^u}\da^{2(p-1)_+}P_i^{\a+1}(\zeta-\varrho)\cdot P^2\mu P_i^{\a-1}L(\zeta-\varrho)&\leq C\int_0^u\mathbb{V}_{\leq\a+1},
  \end{align*}
  and
		\begin{align*}
		\int_{W_t^u}\da^{2(p-1)_+}P_i^{\a+1}(\zeta-\varrho)\cdot P_i^{\a}L(\eta^2)\cdot T(\zeta-\varrho)\leq &\da\int_{W_t^u}\mu P_i^{\a}L\fai\cdot P_i^{\a+1}(\zeta-\varrho)\\
  \leq & \da\int_0^t\mathbb{W}_{\leq\a}+\mathbb{V}_{\leq\a+1}\ dt'.
	\end{align*}
	This completes the proof for this Lemma.
\end{proof}
\section{Top order energy estimates}\label{section6}
In this section, we first give the top order energy estimates for the vorticity and then for the wave variables. This is based on the estimates derived in Section \ref{section5}.
\subsection{Top order energy estimates for the vorticity}
 Let $|\a|=20$ and $P_i^{\a+1}\in\mathcal{P}^{21;\leq3}$. It follows from \eqref{energyzeta} that 
\begin{align}
	E^{\zeta}(P_i^{\a+1}(\zeta-\varrho))+F^{\zeta}(P_i^{\a+1}(\zeta-\varrho))&=E^{\zeta}(P_i^{\a+1}(\zeta-\varrho))(0,u)+\int_{W_t^u}\mu Q_{|\a|+1}^{\zeta},\label{toporderenergyzeta}
\end{align}
where 
\begin{align}
	\mu Q_{|\a|+1}^{\zeta}=(L\mu)(P_i^{\a+1}(\zeta-\varrho))^2+2P_i^{\a+1}(\zeta-\varrho)\mathfrak{F}_{|\a|+1}^{\zeta}+\eta\mu \tr\slashed{k}(P_i^{\a+1}(\zeta-\varrho))^2.
\end{align}
Note that Lemma \ref{preliminary1} gives
\begin{align*}
	\int_{W_t^u}\da^{2(p-1)_+}	[(L\mu)(P_i^{\a+1}(\zeta-\varrho))^2+\eta\mu \tr\slashed{k}(P_i^{\a+1}(\zeta-\varrho))^2]\leq C\int_0^u\mathbb{V}_{\leq 21}+\da\int_0^t\mathbb{V}_{\leq 21}.
\end{align*}
In view of Lemma \ref{commutatorvorticity}, it suffices to consider the contribution from the top order terms in $\sum_{|\be_1|+|\be_2|=|\a|}
P_i^{\be_1}\ppi_{\mu B}^{\sharp}P_i^{\be_2}(\zeta-\varrho)$.  Due to \eqref{commutatorvorticity}, the top order terms to be considered are
\begin{align*}
	P_i^{\a}\prescript{P}{}{\pi}_{\mu B}^{\sharp}\quad\text{and}\quad  \prescript{P}{}{\pi}_{\mu B}^{\sharp}P_i^{\a}(\zeta-\varrho).
\end{align*}
Note that from Table \ref{table2}, in the deformation tensor $\prescript{P}{}{\pi}_{\mu B}^{\sharp}$, if $P=Y$, then $|\prescript{Y}{}{\pi}_{\mu B \hat{X}}|\les\da$; while if $P=L$, then $|\prescript{L}{}{\pi}_{\mu B \hat{X}}|\les 1$ and the number of $L$ in $P_i^{\a}$ is at least $p-1$. Then,
\begin{align}
	&\int_{W_t^u}\da^{2(p-1)_+}\prescript{P}{}{\pi}_{\mu B}^{\sharp}P_i^{\a}(\zeta-\varrho)\cdot P_i^{\a+1}(\zeta-\varrho)\leq\da\int_0^t\mathbb{V}_{\leq 21}.
\end{align}
For the remaining top order acoustical terms, due to \eqref{toporderzeta1} and \eqref{toporderzeta2}, if $p\geq2$, then 
\begin{align*}
	P_i^{\a}\prescript{P}{}{\pi}_{\mu B\hat{X}}&=Y^{\a'}L^{p-1}\prescript{L}{}{\pi}_{\mu B}^{\sharp}=Y^{\a'}L^{p-1}(X\mu+X\fai+\text{l.o.ts})\\
	&=TY^{\a'}XL^{p-2}\fai+\mu Y^{\a'}XL^{p-1}\fai+Y^{\a'}L^{p-1}X\fai+\text{l.o.ts}.
\end{align*}
Hence, for $p\geq 2$, one has
\begin{align}
	\begin{split}
		&\int_{W_t^u}\da^{2(p-1)_+}Y_i^{\a'}L^p(\zeta-\varrho)\cdot P_i^{\a}\prescript{P}{}{\pi}_{\mu B}^{\sharp}(\zeta-\varrho)\leq
		\da\int_{W_t^u}\da^{2(p-1)_+}(Y_i^{\a'}L^p(\zeta-\varrho))^2\\
		&+\da\cdot\da^{2(p-1)_+}\int_0^t\|TY^{\a'+1}L^{p-1}\fai\|_{L^2(\Sigma_{t'}^u)}^2
		+\da\cdot\da^{2(p-1)_+}\int_{W_t^u}\mu_m^{-1}\mu(Y^{\a'+1}L^{p-1}\fai)^2\\
		\leq&\da\mu_m^{-2b_{21}}\bar{\mathbb{V}}_{\leq21}+C\da\mu_m^{-2b_{20}}\bar{\mathbb{W}}_{\leq20}.
	\end{split}
\end{align}
 It remains to consider the contribution from $P_i^{\a}\prescript{P}{}{\pi}_{\mu B}^{\sharp}$ with $p\leq1$. In view of \eqref{toporderzeta1} and \eqref{toporderzeta2}, $|Y^{\a}\prescript{Y}{}{\pi}_{\mu B}^{\sharp}(\zeta-\varrho)|\leq\da|Y^{\a}\prescript{L}{}{\pi}_{\mu B}^{\sharp}(\zeta-\varrho)|$ and one needs only to consider the contribution from $Y^{\a-1}\slashed{\Delta}\mu\cdot \hat{X}(\zeta-\varrho)$. It follows from \eqref{topordermu} and \eqref{modifiedenergy} that 
\begin{align*}
	\begin{split}
	&\int_{W_t^u}Y^{\a+1}(\zeta-\varrho)Y^{\a-1}\slashed{\Delta}\mu\cdot \hat{X}(\zeta-\varrho)\leq\left(\int_{W_t^u}Y^{\a+1}(\zeta-\varrho)\right)^{\frac{1}{2}}\left(\int_{W_t^u}\da Y^{\a-1}\slashed{\Delta}\mu\right)^{\frac{1}{2}}\\
	\leq&\da^{-1}\int_0^u\mathbb{V}_{\leq 21}+\da\int_0^t\mu_m^{-2}\da||\mu Y^{\a-1}\slashed{\Delta}\mu||_{L^2(\Sigma_{t'}^u)}^2
\end{split}
\end{align*}
and
\begin{align*}
	\begin{split}
		&\da^2||\mu Y^{\a-1}\slashed{\Delta}\mu||_{L^2(\Sigma_{t'}^u)}^2\leq\da^2||\mathscr{H}_{\a}||^2_{L^2(\Sigma_0^u)}+C\da^2+\frac{16}{(b_{20}+\tfrac{1}{2})^2}\mu_m^{-2b_{20}+1}\bar{\mathbb{W}}_{\leq20}\\
		&+C\left(\int_0^t(1+\mu_m^{-\frac{1}{2}})\sqrt{\mathbb{W}_{\leq 20}}\right)^2+C\int_0^u\mathbb{Q}_{\leq 20}+C\left(\int_0^t\sqrt{\mathbb{V}_{\leq21}}\right)^2.
	\end{split}
\end{align*}
In view of Lemma \ref{crucial}, it holds that
\begin{align*}
	\int_0^t\mu_m^{-2}\cdot\mu_m^{-b_{20}+1}\bar{\mathbb{W}}_{\leq20}&\leq \frac{1}{2b_{20}}\mu_m^{-2b_{20}}\bar{\mathbb{W}}_{\leq20},\quad \left(\int_0^t\sqrt{\mathbb{V}_{\leq21}}\right)^2\leq\mu_m^{-2b_{21}+2}\bar{\mathbb{V}}_{\leq21},
 \end{align*}
 and
 \begin{align*}
	\int_0^t\mu_m^{-2}(s)\left(\int_0^s(1+\mu_m^{-\frac{1}{2}}(\tau))\sqrt{\mathbb{W}_{\leq 20}}\right)^2&\leq\int_0^t\mu_m^{-2}(s)\left(\frac{1}{2b_{20}}\sqrt{\bar{\mathbb{W}}_{\leq 20}}\mu_m^{-b_{20}}(s)\right)^2\\
 &\leq\frac{1}{4b_{20}^2(2b_{20}+1)}\mu_m^{-2b_{21}}\bar{\mathbb{W}}_{\leq 20}.
\end{align*}
Therefore, one has
\begin{align}
	\begin{split}
		&\int_{W_t^u}Y^{\a+1}(\zeta-\varrho)Y^{\a-1}\slashed{\Delta}\mu\cdot \hat{X}(\zeta-\varrho)\\
  \leq &\da^2||\mathscr{H}_{\a}||_{L^2(\Sigma_0^u)}^2+C\da^2+C\da(\mu_m^{-2b_{21}+1}+\mu_m^{-2b_{21}})\bar{\mathbb{W}}_{\leq20}\\
&+C\da\mu_m^{-2b_{20}}\int_0^u\bar{\mathbb{Q}}_{\leq20}+C\da\mu_m^{-2b_{21}+2}\bar{\mathbb{V}}_{\leq21}+\da^{-1}\mu_m^{-2b_{21}}\int_0^u\bar{\mathbb{V}}_{\leq 21}.
	\end{split}
\end{align}
It follows from \eqref{toporderenergyzeta} that 
\begin{align*}
	\begin{split}
		&\sum_{p_i^{\a+1}\in\mathcal{P}^{21;3}}\da^{2(p-1)_+}[E^{\zeta}(P_i^{\a+1}(\zeta-\varrho))+F^{\zeta}(P_i^{\a+1}(\zeta-\p))](t,u)\\
  \leq&\sum_{P_i^{\a+1}\in\mathcal{P}^{21;3}}\da^{2(p-1)_+}(E^{\zeta}(P_i^{\a+1}(\zeta-\varrho))(0,u)+\int_{W_t^u}\mu Q_{|\a|+1}^{\zeta}).
	\end{split}
\end{align*}
Denote \begin{align}
    \mathscr{D}_{21}^{\zeta}=\sum_{P_i^{\a+1}\in\mathcal{P}^{21;3}}\da^{2(p-1)_+}E^{\zeta}(P_i^{\a+1}(\zeta-\varrho))(0,u).\label{D21zeta}
\end{align} Therefore, collecting the above results yields
\begin{align}
	\begin{split}\label{toporderenergyzeta1}
		\mathbb{V}_{\leq 21}(t,u)\leq &\sum_{P_i^{\a+1}\in\mathcal{P}^{21;3}}\da^{2(p-1)_+}(E^{\zeta}(P_i^{\a+1}(\zeta-\varrho))(0,u)+\int_{W_t^u}\mu Q_{|\a|+1}^{\zeta})\\
		\leq& \mathscr{D}_{21}^{\zeta}+\da\cdot\da^2\|\mathscr{H}_{\a}\|^2_{L^2(\Sigma_{t'}^u)}+C\da(\mu_m^{-2b_{20}}+\mu_m^{-2b_{21}})\bar{\mathbb{W}}_{\leq 20}+C\da\mu_m^{-2b_{21}}\bar{\mathbb{V}}_{\leq 21}\\	&+C\da\mu_m^{-2b_{20}}\int_0^u\bar{\mathbb{Q}}_{\leq20}+C\da\int_0^t\mathbb{V}_{\leq21}+C\da^{-\frac{1}{2}}\mu_m^{-2b_{21}}\int_0^u\bar{\mathbb{V}}_{\leq21}.
	\end{split}
\end{align}
Thus multiplying $\mu_m^{2b_{21}}(t)$ on both sides yields that for sufficiently small $\da$
\begin{align}
	\begin{split}\label{toporderenergyzeta2}		\bar{\mathbb{V}}_{\leq21}\leq&\mathscr{D}_{21}^{\zeta}+\da\cdot\da^2\|\mathscr{H}_{\a}\|^2_{L^2(\Sigma_{0}^u)}+C\da\bar{\mathbb{W}}_{\leq20}\\
 &+C\da\int_0^u\bar{\mathbb{Q}}_{\leq20}+C\da\int_0^t\bar{\mathbb{V}}_{\leq21}+C\da^{-\frac{1}{2}}\int_0^u\bar{\mathbb{V}}_{\leq21}.
	\end{split}
\end{align}
Applying Gronwall inequality yields
\begin{align}
	\begin{split}\label{toporderenergyzeta3}
 \bar{\mathbb{V}}_{\leq21}\leq\mathscr{D}_{21}^{\zeta}+C\da^2+\da\cdot\da^2\|\mathscr{H}_{\a}\|^2_{L^2(\Sigma_{0}^u)}+C\da\bar{\mathbb{W}}_{\leq20}+C\da\int_0^u\bar{\mathbb{Q}}_{\leq20}.
	\end{split}
\end{align}
\subsection{Top order energy estimates for wave variables}
We consider the following integrals:
\begin{align*}
	-\int_{W_t^u}\mathfrak{F}_{\a+1}\cdot K_{0}Z^{\a}\fai/K_1Z^{\a}\fai\ d\vartheta'du'dt'
\end{align*}
for $Z^\a\in\mathcal{Z}^{20;1}$. In view of Section 4.3, the contributions in $\mathfrak{F}_{\a+1}$ to be considered are
\begin{align*}
	Y^{\a-1}\hat{X}\tr\chi\cdot T\fai,\quad Y^{\a-1}T\tr\chi\mu \cdot T\fai.
\end{align*}
\subsubsection{Contributions associated with $K_0$}
We first consider the space-time integral:
\begin{align}
	\int_{W_t^u}T\fai\cdot Y^{\a-1}\hat{X}\tr\chi\cdot\dl Y^{\a}\fai.\label{K0top1}
\end{align}
In view of Remark \ref{recovertfai}, $|T\fai|\leq 1+C\da$ and then \eqref{K0top1} can be bounded as
\begin{align*}
	\int_{W_t^u}T\fai\cdot Y^{\a-1}\hat{X}\tr\chi\cdot\dl Y^{\a}\fai\leq \int_0^t\mu_m^{-1}\|\mu Y^{\a-1}\hat{X}\tr\chi\|_{L^2(\Sigma_{t'}^u)}\|\dl Y^{\a}\fai\|_{L^2(\Sigma_{t'}^u)}\ dt'.
\end{align*}
We estimate the above integral term by term in view of \eqref{toporderchi}. It follows from \eqref{modifiedenergy} that 
\[
\|\dl Y^{\a}\fai\|_{\ltwou}\leq\sqrt{\mathbb{W}_{\leq 20}}.
\]
To deal with the integrals involving $\mu^{-c}$ for $c>1$, one splits them into shock part $\int_{t_0}^t\mu^{-c}$ and non-shock part $\int_0^{t_0}\mu^{-c}$. It follows from Lemma \ref{crucial} and Remark \ref{crucial2} that 
\begin{align}
	\int_0^{_0}\mu_m^{-c-1}\leq \int_0^{t_0}\mu_m^{-c},\quad \int_{t_0}^t\mu_m^{-c-1}\leq\frac{1+C\da}{c}\mu_m^{-c}.
\end{align}
Hence one has
\begin{equation*}
\begin{aligned}
	&\frac{2}{b_{20}+\tfrac{1}{2}}\int_0^t\mu_m^{-b_{20}-\tfrac{1}{2}}\sqrt{\bar{\mathbb{W}}_{\leq20}}\cdot\|\dl Y^{\a}\fai\|_{L^2(\Sigma_{t'}^u)}\leq\frac{2}{b_{20}+\tfrac{1}{2}}\int_0^t\mu_m^{-2b_{20}-\tfrac{1}{2}}\bar{\mathbb{W}}_{\leq20}\\
	\leq & \frac{2}{(2b_{20}-\tfrac{1}{2})(b_{20}+\tfrac{1}{2})}\mu_m^{-2b_{20}}\bar{\mathbb{W}}_{\leq20}+\frac{1}{3}\int_0^t\mu_m^{-2b_{20}}\bar{\mathbb{W}}_{\leq20},
\end{aligned}
\end{equation*}
\begin{equation*}
\begin{aligned}
&C\int_0^t\mu_m^{-1}\int_0^s(1+\da^{\frac{1}{2}}\mu_m^{-\frac{1}{2}})\sqrt{\mathbb{W}_{\leq20}} ds\cdot\|\dl Y^{\a}\fai\|_{L^2(\Sigma_{t'}^u)}\\
 \leq &\frac{1}{b_{20}+\tfrac{1}{2}}\int_0^t\mu_m^{-2b_{20}-\frac{1}{2}}\bar{\mathbb{W}}_{\leq20}+C\int_0^t\mu_m^{-2b_{20}}\bar{\mathbb{W}}_{\leq20}\\
 \leq &\frac{1}{(b_{20}+\tfrac{1}{2})(2b_{20}-\tfrac{1}{2})}\mu_m^{-2b_{20}}\bar{\mathbb{W}}_{\leq20}+C\int_0^t\mu_m^{-2b_{20}}\bar{\mathbb{W}}_{\leq20},
\end{aligned}
\end{equation*}
\begin{equation*}
\begin{aligned}
&C\int_0^t\sqrt{\int_0^u\mathbb{Q}_{\leq20}}\cdot\|\dl Y^{\a}\fai\|_{L^2(\Sigma_{t'}^u)}\leq C\int_0^t\mu_m^{-2b_{20}}\sqrt{\int_0^u\bar{\mathbb{Q}}_{\leq20}}\sqrt{\bar{\mathbb{W}}_{\leq20}}\\	\leq &\da\mu_m^{-2b_{20}}\bar{\mathbb{W}}_{\leq20}+C\da^{-1}\mu_m^{-2b_{20}}\int_0^u\bar{\mathbb{Q}}_{\leq20}+\da\int_0^t\mu_m^{-2b_{20}}\bar{\mathbb{W}}_{\leq20}.
\end{aligned}
\end{equation*}
It follows from \eqref{toporderenergyzeta} that 
\begin{align*}
	\begin{split}
		&C\int_0^t\mu_m^{-1}\int_0^s\sqrt{\mathbb{V}_{\leq21}}\ ds\cdot \|\dl Y^{\a}\fai\|_{L^2(\Sigma_{t'}^u)}
		\leq \frac{C}{(b_{21}-1)(2b_{20}-\frac{1}{2})}\mu_m^{-2b_{20}}\sqrt{\bar{\mathbb{W}}_{\leq20}}\sqrt{\bar{\mathbb{V}}_{\leq21}}\\
		\leq & C\mu_m^{-2b_{20}}(\mathscr{D}_{21}^{\zeta}+\da^2+\da^2\|\mathscr{H}_{\a}\|^2_{L^2(\Sigma_{0}^u)})
  +\frac{\mu_m^{-2b_{20}}}{(b_{21}-1)^2(2b_{20}-\frac{1}{2})^2}\bar{\mathbb{W}}_{\leq20}+C\da^{\frac{1}{2}}\mu_m^{-2b_{20}}\int_0^u\bar{\mathbb{Q}}_{\leq20}
	\end{split}
\end{align*}
Therefore, collecting the above results yields
\begin{align}
	\begin{split}
	\eqref{K0top1}\leq & C\mu_m^{-2b_{20}}(\mathscr{D}_{21}^{\zeta}+\da^2+\da^2\|\mathscr{F}_{\a}\|^2_{L^2(\Sigma_{0}^u)}+\da^2\|\mathscr{H}_{\a}\|^2_{L^2(\Sigma_{0}^u)})\\
 &+
	\frac{3+C\da^{\frac{1}{2}}}{(2b_{20}-\tfrac{1}{2})(b_{20}+\tfrac{1}{2})}\mu_m^{-2b_{20}}\bar{\mathbb{W}}_{\leq20}\\
	&+C\int_0^t\mu_m^{-2b_{20}}(\bar{\mathbb{W}}_{\leq20}+\bar{\mathbb{Q}}_{\leq20})+C\da^{-1}\mu_m^{-2b_{20}}\int_0^u\bar{\mathbb{Q}}_{\leq20}.
	\end{split}\label{K0top1final}
\end{align}
Next, we estimate the following space-time integral:
\begin{align}
\begin{split}
	&\int_{W_t^u}\da^2 T\fai\cdot Y^{\a-1}T\tr\chi\cdot \dl Y^{\a-1}T\fai\\
\leq&\int_0^t\mu_m^{-1}\da\|\mu Y^{\a-1}T\tr\chi\|_{L^2(\Sigma_{t'}^u)}\cdot\da\|\dl Y^{\a-1}T\fai\|_{L^2(\Sigma_{t'}^u)}.
\end{split}\label{K0top2}
\end{align}
Similarly,
\begin{align*}
	\da\|\dl Y^{\a-1}T\fai\|_{L^2(\Sigma_{t}^u)}\leq\sqrt{\mathbb{W}_{\leq20}}.
\end{align*}
It follows from \eqref{topordermu} and the above arguments that 
\begin{align}
	\begin{split}
	\eqref{K0top2}\leq &  C\mu_m^{-2b_{20}}(\mathscr{D}_{21}^{\zeta}+\da^2+\da^2\|\mathscr{F}_{\a}\|^2_{L^2(\Sigma_{0}^u)}+\da^2\|\mathscr{H}_{\a}\|^2_{L^2(\Sigma_{0}^u)})\\
 &+\frac{5+C\da^{\frac{1}{2}}}{(2b_{20}-\tfrac{1}{2})(b_{20}+\tfrac{1}{2})}\mu_m^{-2b_{20}}\bar{\mathbb{W}}_{\leq20}\\
	&+C\int_0^t\mu_m^{-2b_{20}}(\bar{\mathbb{W}}_{\leq20}+\bar{\mathbb{Q}}_{\leq20})+C\da^{-1}\mu_m^{-2b_{20}}\int_0^u\bar{\mathbb{Q}}_{\leq20}.
\end{split}\label{K0top2final}
\end{align}
\subsubsection{Contributions associated with $K_1$}
We consider the following space-time integral:
\begin{align}
\begin{split}
	&\int_{W_t^u}T\fai\cdot Y^{\a-1}\hat{X}\tr\chi\cdot (1+\eta^{-2}\mu)LY^{\a}\fai\\
	\leq&\int_{W_t^u}Y^{\a-1}\hat{X}\tr\chi\cdot LY^{\a}\fai+\left(\int_0^t||\mu Y^{\a-1}\hat{X}\tr\chi||^2_{L^2(\Sigma_{t'}^u)}\right)^{\frac{1}{2}}\cdot\left(\int_0^u|| LY^{\a}\fai||^2_{L^2(C_{u'}^t)}\right)^{\frac{1}{2}}\\
 := & S_1+S_2.
\end{split} \label{K1top1}
\end{align}
 Note that 
\begin{align*}
	S_2\leq\da^{\frac{1}{2}}\int_0^t\|\mu Y^{\a-1}\hat{X}\tr\chi\|_{L^2(\Sigma_{t'}^u)}^2+\da^{-\frac{1}{2}}\int_0^u\mathbb{Q}_{\leq\a+1}.
\end{align*}
Hence $S_2$ can be estimated similarly as in Subsection 6.2.1. For $S_1$, it follows from Lemma \ref{integralbypartsl} that
\begin{align*}
	\begin{split}
		S_1= &\int_{W_t^u}(L+\tr\chi)\left(Y^{\a-1}\hat{X}\tr\chi\cdot Y^{\a}\fai\right)-\int_{W_t^u}(L+\tr\chi)(Y^{\a-1}\hat{X}\tr\chi)\cdot Y^{\a}\fai\\
		= &\int_{\Sigma_{t}^u}Y^{\a-1}\hat{X}\tr\chi\cdot Y^{\a}\fai-\int_{\Sigma_0^u}Y^{\a-1}\hat{X}\tr\chi\cdot Y^{\a}\fai -\int_{W_t^u}(L+\tr\chi)(Y^{\a-1}\hat{X}\tr\chi)\cdot Y^{\a}\fai\\
		:= &I_0-I_1-I_2,
	\end{split}
\end{align*}
where
\begin{align*}
	I_0=\int_{\Sigma_{t}^u}Y^{\a-1}\hat{X}\tr\chi\cdot Y^{\a}\fai,\quad  I_1=\int_{\Sigma_0^u}Y^{\a-1}\hat{X}\tr\chi\cdot Y^{\a}\fai,\
 \end{align*}
 and
 \begin{align*}
	I_2=\int_{W_t^u}(L+\tr\chi)(Y^{\a-1}\hat{X}\tr\chi)\cdot Y^{\a}\fai.
\end{align*}
Clearly, $I_0\leq C\|Y^{\a}\fai\|^2_{L^2(\Sigma_0^u)}$. For $I_0$, it follows from Lemma \ref{change} that
\begin{align*}
	\begin{split}
		I_0=-\int_{\Sigma_t^u}Y^{\a-2}\hat{X}\tr\chi\cdot Y^{\a+1}\fai-\int_{\Sigma_{t}^u}\frac{1}{2}\tr\prescript{Y}{}{\slashed{\pi}} Y^{\a}\fai\cdot Y^{\a-2}\hat{X}\tr\chi:=-A_1-A_2.
	\end{split}
\end{align*}
In view of Table \ref{table2}, $A_2$ is a lower order term compared with $A_1$. Note that $Y^{\a-2}\hat{X}\tr\chi$ is not a top order term. Then, it follows from Lemma \ref{lotchiul2} that
\begin{equation*}
\begin{aligned}
	A_1&\leq \|Y^{\a-2}\hat{X}\tr\chi\|_{L^2(\Sigma_{t}^u)}\|Y^{\a+1}\fai\|_{\ltwou}\leq\mu_m^{-b_{20}-\frac{1}{2}}\\
	&\leq\frac{\tfrac{3}{2}+C\da}{b_{20}-\frac{1}{2}}\bar{\mathbb{U}}_{\leq20}+C\int_0^t\mu_m^{-2b_{20}}\bar{\mathbb{U}}_{\leq20}.
\end{aligned}
\end{equation*}
To estimate $I_2$, one first writes it as
\begin{align*}
	I_2=\int_{W_t^u}LY^{\a-1}\hat{X}\tr\chi\cdot Y^{\a}\fai+\tr\chi Y^{\a-1}\hat{X}\tr\chi\cdot Y^{\a}\fai:=B_1+B_2.
\end{align*}
In view of Lemma \ref{2ndff}, $\tr\chi\leq\da$ and  then $B_2$ is a lower order term compared with $B_1$. It suffices to estimate $B_1$ and it follows from \eqref{commutator} that
\begin{align*}
  B_1&=\int_{W_t^u}YLY^{\a-2}\hat{X}\tr\chi+\prescript{Y}{}{\pi}_{L\hat{X}}\hat{X}Y^{\a}\hat{X}\tr\chi\cdot Y^{\a}\fai:=B_{11}+B_{12}.
\end{align*}
In view of Table \ref{table2} and Lemma \ref{faiL2}, $B_{12}$ is a lower order term which can be bounded directly as in Section 4.4. For $B_{11}$, due to Lemma \ref{change}, one writes it as
\begin{align*}
	B_{11}&=-\int_{W_t^u}LY^{\a-2}\hat{X}\tr\chi\cdot Y^{\a+1}\fai-\int_{W_t^u}\frac{1}{2}\tr\prescript{Y}{}{\slashed{\pi}}LY^{\a-2}\cdot Y^{\a}\fai:=-B_{111}-B_{112},
\end{align*} 
where $B_{112}$ is a lower order term compared with $B_{111}$. Thanks to \eqref{Lyatrchi}, 
\begin{align*}
	LY^{\a-2}\hat{X}\tr\chi&=(e+2\tr\chi+O(\da))Y^{\a-2}\hat{X}\tr\chi+\frac{3}{2}Y^{\a-2}\hat{X}^3\fai+O(\da)\cdot O(\fai)_1^{\leq\a+1}.
\end{align*}
Hence it follows from Lemma \ref{lotchiul2} that
\begin{align*}
	B_{111}\leq\left(\frac{3}{2}+C\da\right)\int_0^t\mu_m^{-2b_{20}-1}\bar{\mathbb{U}}_{\leq 20}\leq\frac{\tfrac{3}{2}+C\da}{2b_{20}}\bar{\mathbb{U}}_{\leq20}.
\end{align*}
Collecting the above results yields
\begin{align}
	\begin{split}
    \eqref{K1top1}\leq &\left(\frac{\tfrac{3}{2}+C\da}{b_{20}-\tfrac{1}{2}}+\frac{\tfrac{3}{2}+C\da}{2b_{20}}\right)\mu_m^{-2b_{20}}\bar{\mathbb{U}}_{\leq20}+C\da\mu_m^{-2b_{20}}\bar{\mathbb{W}}_{\leq20}\\
    &+C\int_0^t\mu_m^{-2b_{20}}(\bar{\mathbb{U}}_{\leq20}+\bar{\mathbb{W}}_{\leq20})+C\da^{-\frac{1}{2}}\mu_m^{-2b_{20}}\int_0^u\bar{\mathbb{Q}}_{\leq20}.
\end{split}\label{K1top1final}
\end{align}
Finally, one has
\begin{align}
\begin{split}
	&\int_{W_t^u}\da^2T\fai\cdot Y^{\a-1}T\tr\chi\cdot (1+\eta^{-2}\mu) LY^{\a-1}T\fai\\
	\leq&\int_{W_t^u}\da^2 Y^{\a-1}T\tr\chi\cdot LY^{\a-1}T\fai+\int_{W_t^u}\da^2\mu Y^{\a-1}T\tr\chi\cdot LY^{\a-1}T\fai:=R_1+R_2.
 \end{split}\label{K1top2}
\end{align}
Similarly as before,
\begin{align*}
	R_2&\leq\left(\int_0^t\da^2\|\mu Y^{\a-1}T\tr\chi\|^2_{L^2(\Sigma_{t'}^u)}\right)^{\frac{1}{2}}\cdot\left(\int_0^u\da^2\|LY^{\a-1}T\fai\|_{L^2(\Sigma_{t'}^u)}\right)^{\frac{1}{2}}\\
	&\leq\da^{\frac{1}{2}}\int_0^t\da^2\|\mu Y^{\a-1}T\tr\chi\|^2_{L^2(\Sigma_{t'}^u)}+\da^{-\frac{1}{2}}\int_0^u\mathbb{Q}_{\leq20},
\end{align*}
so that $R_2$ can be bounded directly due to \eqref{topordermu}. Again, in view of Lemma \ref{integralbypartsl}, one writes $R_1$ as
\begin{align*}
	\begin{split}
		R_1=&\da^2\int_{W_t^u}(L+\tr\chi)[Y^{\a-1}T\tr\chi\cdot Y^{\a-1}T\fai]-\da^2\int_{W_t^u}(L+\tr\chi)(Y^{\a-1}T\tr\chi)\cdot Y^{\a-1}T\fai\\
		=&\da^2\int_{\Sigma_t^u}Y^{\a-1}T\tr\chi\cdot Y^{\a-1}T\fai-\da^2\int_{\Sigma_0^u}Y^{\a-1}T\tr\chi\cdot Y^{\a-1}T\fai\\
		&-\da^2\int_{W_t^u}(L+\tr\chi)(Y^{\a-1}T\tr\chi)\cdot Y^{\a-1}T\fai:=H_0-H_1-H_2,
	\end{split}
\end{align*}
where
\begin{align}
\begin{split}
	H_0&=\da^2\int_{\Sigma_t^u}Y^{\a-1}T\tr\chi\cdot Y^{\a-1}T\fai,\ H_1=\da^2\int_{\Sigma_0^u}Y^{\a-1}T\tr\chi\cdot Y^{\a-1}T\fai,\\
 H_2&=\da^2\int_{W_t^u}(L+\tr\chi)(Y^{\a-1}T\tr\chi)\cdot Y^{\a-1}T\fai.
 \end{split}
\end{align}
Clearly, $H_1\leq\da^2||Y^{\a-1}T\fai||^2_{L^2(\Sigma_0^u)}$. For $H_0$, it follows from Lemma \ref{change} that
\begin{align*}
	H_0=-\da^2\int_{\Sigma_{t}^u}Y^{\a-2}T\tr\chi\cdot Y^{\a}T\fai-\da^2\int_{\Sigma_{t}^u}\tfrac{1}{2}\tr\prescript{Y}{}{\slashed{\pi}} Y^{\a-2}T\tr\chi\cdot Y^{\a-1}T\fai:=-C_1-C_2.
\end{align*}
Since $|\tr\prescript{Y}{}{\slashed{\pi}}| \les\da$, $C_2$ is a lower order term compared with $C_1$. Due to \eqref{equationtchi}, $Y^{\a-2}T\tr\chi$ is not a top order term and it holds that 
\begin{equation*}
\begin{aligned}
\|Y^{\a-2}T\tr\chi\|_{\ltwou}\leq &\|Y^{\a}\mu\|_{\ltwou}+\da\|Y^{\a-1}Tv^1\|_{\ltwou}+\|Y^{\a-1}Tv^2\|_{\ltwou}\\
+&\mu\|Y^{\a}\fai\|_{\ltwou}+\text{l.o.ts}.
\end{aligned}
\end{equation*}
Hence it follows from Lemmas \ref{lotchiul2} and \ref{crucial} that
\begin{align*}
	\begin{split}
	 	C_1&\leq\da\|Y^{\a-2}T\fai\|_{\ltwou}\cdot\da||Y^{\a}T\fai||_{\ltwou}
   \\
   &\leq C\da\int_0^t(\sqrt{\mathbb{W}_{\leq20}}+\mu_m^{-\frac{1}{2}}\sqrt{\mathbb{U}_{\leq20}})\cdot\mu_m^{-\frac{1}{2}}\sqrt{\mathbb{U}_{\leq20}}\\
	 	&\leq C\da\mu_m^{-2b_{20}}(\bar{\mathbb{W}}_{\leq20}+\bar{\mathbb{U}}_{\leq20}+\bar{\mathbb{W}}^{\text{par}}_{\leq20})+C\da\int_0^t\mu_m^{-2b_{20}}(\bar{\mathbb{W}}_{\leq20}+\bar{\mathbb{U}}_{\leq20}+\bar{\mathbb{W}}^{\text{par}}_{\leq20}).
	\end{split}
\end{align*}
For $H_2$, it can be divided as
\begin{align*}
	-\da^2\int_{W_t^u}LY^{\a-1}T\tr\chi\cdot Y^{\a-1}T\fai-\da^2\int_{W_t^u}\tr\chi Y^{\a-1}T\tr\chi\cdot Y^{\a-1}T\fai:=-D_1-D_2.
\end{align*}
To estimate $D_2$, one uses Lemma \ref{change} as
\begin{equation*}
\begin{aligned}
	D_2&=\int_{W_t^u}Y^{\a-2}T\tr\chi\cdot Y(\tr\chi\cdot Y^{\a-1}T\fai)+\tfrac{1}{2}\tr\prescript{Y}{}{\slashed{\pi}}\tr\chi Y^{\a-2}T\tr\chi\cdot Y^{\a-1}T\fai:=D_{21}+D_{22},
\end{aligned}
\end{equation*}
where $D_{22}$ is a lower order term compared with $D_{21}$. It suffices to consider the contribution from $\tr\chi Y^{\a}T\fai$. Note that $Y^{\a-2}T\tr\chi$ is not a top order term. Thus one has
\[
D_{21}\leq\da\int_0^t\da||Y^{\a-2}T\tr\chi||_{L^2(\Sigma_{t'}^u)}\cdot \da||Y^{\a}T\fai||_{L^2(\Sigma_{t'}^u)},
\]
so that $D_{21}$ can be bounded in the same manner as $C_1$. It remains to estimate $D_1$, which can be written as
\begin{equation*}
\begin{aligned}
	D_1= &-\da^2\int_{W_t^u}YLY^{\a-2}T\tr\chi\cdot Y^{\a-1}T\fai+\ypi_{L\hat{X}}\hat{X}Y^{\a-2}T\tr\chi\cdot Y^{\a-1}T\fai:= -D_{11}-D_{12}.
\end{aligned}  
\end{equation*}
Since $|\ypi_{L\hat{X}}|\leq\da$, then $D_{12}$ can be bounded similarly as $D_{2}$. For $D_{11}$, one again uses Lemma \ref{change} to rewrite it as
\begin{align*}
	D_{11}&=\da^2\int_{W_t^u}LY^{\a-2}T\tr\chi\cdot Y^{\a}T\fai+\tfrac{1}{2}\tr\prescript{Y}{}{\slashed{\pi}} Y^{\a-2}T\tr\chi\cdot Y^{\a-1}T\fai:=D_{111}+D_{112},
\end{align*}
where it suffices to estimate $D_{111}$. Thanks to \eqref{equationtchi} and \eqref{Lyamu}, one obtains
\begin{align*}
	LY^{\a-2}T\tr\chi&=LY^{\a+1}\mu+\hat{X}^iLY^{\a-2}\hat{X}Tv^i+\mu LY^{\a+1}\fai+\text{l.o.ts},\\
	LY^{\a+1}\mu&=\frac{3}{2}\hat{T}^1Y^{\a+1}Tv^1+\frac{3}{2}\mu Y^{\a+1}L\fai+(e+O(\da))Y^{\a+1}\mu+O(\da)\cdot O(\fai)_{1}^{\leq\a+2}.
\end{align*}
Therefore, it follows from Lemma \ref{crucial} that
\begin{align*}
	\begin{split}
	D_{111}&\leq\int_0^t\da(\sqrt{\mathbb{W}_{\leq20}}+\mu_m^{-\frac{1}{2}}\sqrt{\mathbb{U}_{\leq20}}+\mu_m^{-\frac{1}{2}}\sqrt{\mathbb{W}^{\text{par}}_{\leq20}})\cdot\mu_m^{-\frac{1}{2}}\sqrt{\mathbb{U}_{\leq20}}\ dt'\\
	&\leq C\da\mu_m^{-2b_{20}}(\bar{\mathbb{W}}_{\leq20}+\bar{\mathbb{U}}_{\leq20}+\bar{\mathbb{W}}^{\text{par}}_{\leq20})+C\da\int_0^t\mu_m^{-2b_{20}}(\bar{\mathbb{W}}_{\leq20}+\bar{\mathbb{U}}_{\leq20}+\bar{\mathbb{W}}^{\text{par}}_{\leq20}).
\end{split}
\end{align*}
Hence collecting the above results yields
\begin{align}
	\begin{split}
	\eqref{K1top2}\leq &C\da\mu_m^{-2b_{20}}(\bar{\mathbb{W}}_{\leq20}+\bar{\mathbb{U}}_{\leq20}+\bar{\mathbb{W}}^{\text{par}}_{\leq20})\\
	&+C\da\int_0^t\mu_m^{-2b_{20}}(\bar{\mathbb{W}}_{\leq20}+\bar{\mathbb{U}}_{\leq20}+\bar{\mathbb{W}}^{\text{par}}_{\leq20})
	+C\da^{-\frac{1}{2}}\mu_m^{-2b_{20}}\int_0^u\bar{\mathbb{Q}}_{\leq20}.
	\end{split}\label{K1top2final}
\end{align}
\begin{remark}\label{Wpar}
	 It follows from Remark \ref{Tv2} that $|Tv^2|\les\da$. This implies that for $\fai=v^2$, the estimates \eqref{K0top1final}, \eqref{K0top2final}, \eqref{K1top1final} and \eqref{K1top2final} hold with an additional $\da$ in the front of right-hand side of these inequalities. More precisely, for \eqref{K0top1final} with $\fai=v^2$, one has 
	 \begin{align*}
	 	\begin{split}
	 	\int_{W_t^u}Tv^2\cdot Y^{\a-1}\hat{X}\tr\chi\cdot \underline{L}Y^{\a}v^2\leq &  C\da\mu_m^{-2b_{20}}(\bar{\mathbb{W}}_{\leq20}+\bar{\mathbb{U}}_{\leq20}+\bar{\mathbb{W}}^{\text{par}}_{\leq20})\\
	 & +C\da\int_0^t\mu_m^{-2b_{20}}(\bar{\mathbb{W}}_{\leq20}+\bar{\mathbb{U}}_{\leq20})+C\mu_m^{-2b_{20}}\int_0^u\bar{\mathbb{Q}}_{\leq20}.
	 \end{split}
	 \end{align*} 
\end{remark}
\subsection{Top order energy estimates}
For $Z^{\a}\in \mathcal{Z}^{20,1}$, it follows from \eqref{energywave} that 
\begin{align}\label{toporderK11}
\begin{split}
&\sum_{|\a'|\leq|\a|}\da^{2m'}\left(E_1(Z^{\a'}\fai)(t,u)
+F_1(Z^{\a'}\fai)(t,u)+K(Z^{\a'}\fai)(t,u)\right)\\
\leq & C\sum_{|\a'|\leq|\a|}\da^{2m'}E_1(Z^{\a'}_i\fai)(0,u)+
\sum_{|\a'|\leq|\a|}\da^{2m'}\int_{W_t^{u}}\mu Q_{1}^{|\a'|},
\end{split}
\end{align}
where $m'\leq1$ is the number of $T$ in $Z^{\a'}$ and $\mu Q_1^{|\a|}:=(-\mathfrak{F}_{|\a|+1}K_1(Z^{\a}\fai)+\frac{1}{2}\mu T^{\a\be}(Z^{\a}\fai)\pi_{1,\a\be})$. The integrals $\int_{W_t^{u}}\mu Q_{1}^{|\a|}$ contain the following two parts
\begin{itemize}
\item[(1)] Contributions from $T^{\a\be}(Z^{\a}\fai) $ which can be bounded by high order energies and fluxes directly and have been treated in previous sections.
\item[(2)] Major terms of the form
\begin{align}\label{integralQ1a+2}
\int_{W^{u}_t}-\mathfrak{F}_{|\a|+1}\cdot K_1 Z^{\a}\fai\ d\vartheta' du'dt',
\end{align}
where the contributions from the top order acoustical terms have been estimated in the last subsection, see\eqref{K1top1final} and \eqref{K1top2final}. The contributions from other terms in $\mathfrak{F}_{|\a|+1}$ have been given in Lemmas \ref{inhomo1} and \ref{inhomo2}.
\end{itemize}
Collecting the results stated above yields
\begin{align}
	\begin{split}
	&\mu_m^{2b_{20}}\sum_{|\a'|\leq|\a|}\da^{2m'}\left(E_1(Z^{\a'}\fai)(t,u)
	+F_1(Z^{\a'}\fai)(t,u)+K(Z^{\a'}\fai)(t,u)\right)\\
	\leq&\left(\frac{\tfrac{3}{2}}{b_{20}-\tfrac{1}{2}}+\frac{\tfrac{3}{2}}{2b_{20}}+C\da\right)\bar{\mathbb{U}}_{\leq20}+C\da(\bar{\mathbb{W}}_{\leq20}+\bar{\mathbb{W}}_{\leq20}^{\text{par}})+C\da^{\frac{1}{2}}\bar{\mathbb{K}}_{\leq20}\\
 &+C\int_0^t\bar{\mathbb{W}}_{\leq20}+\bar{\mathbb{U}}_{\leq20}\ dt'+\int_0^t\bar{\mathbb{V}}_{\leq21}\\
&+\da^{-1}\int_0^u\bar{\mathbb{V}}_{\leq21}+C\da^{\frac{1}{2}}\int_0^t\mu_m^{-\frac{1}{2}}\mathbb{U}_{\leq20}+C\da^{-\frac{1}{2}}\int_0^u\bar{\mathbb{Q}}_{\leq20}.
	\end{split}\label{energytopK11}
\end{align}
Since $b_{20}=5+\frac{3}{4}$, $\frac{\tfrac{3}{2}}{b_{20}-\tfrac{1}{2}}+\frac{\tfrac{3}{2}}{2b_{20}}<1/2$. It follows from Lemma \ref{crucial} that \[\int_0^t\mu_m^{-\frac{1}{2}}\mathbb{U}_{\leq20}\leq\frac{1}{2b_{20}-\tfrac{1}{2}}\bar{\mathbb{U}}_{\leq20}.\] Therefore, in view of $\mu_m\leq 1$, top order energy estimates for the vorticity \eqref{toporderenergyzeta} and the monotonicity of right hand side of \eqref{energytopK11}, \eqref{energytopK11} implies that for sufficiently small $\da$,
\begin{align}
	\begin{split}
		&\bar{\mathbb{U}}_{\leq20}(t,u)+\bar{\mathbb{Q}}_{\leq20}(t,u)+\bar{\mathbb{K}}_{\leq20}(t,u)\\
		\leq&C\da(\bar{\mathbb{W}}_{\leq20}+\bar{\mathbb{W}}_{\leq20}^{\text{par}})+C\int_0^t\bar{\mathbb{W}}_{\leq20}+\bar{\mathbb{U}}_{\leq20}\ dt'+C\da^{-\frac{1}{2}}\int_0^u\bar{\mathbb{Q}}_{\leq20} +\mathbb{U}_{\leq20}(0,u)+\mathscr{D}_{21}^{\zeta}\\
 &+\da^2\|\mathscr{H}_{\a}\|^2_{L^2(\Sigma_{0}^u)}+\|\mathscr{F}_{\a}\|^2_{L^2(\Sigma_0^u)}+\|Y^{\a}\tr\chi\|^2_{L^2(\Sigma_0^u)}+\|Y^{\a+1}\mu\|_{L^2(\Sigma_0^u)}^2
	\end{split}\label{topenergyK12}
\end{align}
for $|\a|=19$. Note that on $\Sigma_0$, $\mu=\eta$, $\tr\chi=\chi=\pa_2v^2$ and $Y=\pa_2$. Then, 
\begin{equation*}
\da^2\|\mathscr{H}_{\a}\|^2_{L^2(\Sigma_{0}^u)}+\|\mathscr{F}_{\a}\|^2_{L^2(\Sigma_0^u)}+\|Y^{\a}\tr\chi\|^2_{L^2(\Sigma_0^u)}+\|Y^{\a+1}\mu\|_{L^2(\Sigma_0^u)}^2\leq C(\mathbb{W}_{\leq 20}+\mathbb{U}_{\leq 20})(0,u).
\end{equation*}
We denote $$\mathscr{D}_{20}^{\text{wave}}:=(\mathbb{W}_{\leq20}+\mathbb{U}_{\leq20})(0,u)\quad \text{and}\quad  \mathscr{D}_{21}=\mathscr{D}_{21}^{\zeta}+\mathscr{D}_{20}^{\text{wave}}.$$ 
Since $\bar{\mathbb{U}},\ \bar{\mathbb{W}},$ and $\bar{\mathbb{K}}$ are all non-negative, then keeping only $\bar{\mathbb{U}}_{\leq20}$ on the left hand side of \eqref{topenergyK12} yields
\begin{equation}\label{topenergyK13}
\begin{aligned}
\bar{\mathbb{U}}_{\leq20}(t,u)\leq & C\mathscr{D}_{21}+C\da(\bar{\mathbb{W}}_{\leq20}+\bar{\mathbb{W}}_{\leq20}^{\text{par}})+C\int_0^t\bar{\mathbb{W}}_{\leq20}+\bar{\mathbb{U}}_{\leq20}\ dt'+C\da^{-\frac{1}{2}}\int_0^u\bar{\mathbb{Q}}_{\leq20}.
\end{aligned}
\end{equation}
Applying Gronwall inequality to \eqref{topenergyK13} yields
\begin{equation}
\begin{aligned}
	&\bar{\mathbb{U}}_{\leq20}(t,u)+\bar{\mathbb{Q}}_{\leq20}(t,u)+\bar{\mathbb{K}}_{\leq20}(t,u)
 \\
 \leq & C\mathscr{D}_{21}+C\da(\bar{\mathbb{W}}_{\leq20}+\bar{\mathbb{W}}_{\leq20}^{\text{par}})+C\int_0^t\bar{\mathbb{W}}_{\leq20}+C\da^{-\frac{1}{2}}\int_0^u\bar{\mathbb{Q}}_{\leq20}.\label{topenergyK14}
\end{aligned} 
\end{equation}
Applying same argument to $\bar{\mathbb{Q}}_{\leq20}$ yields
\begin{align}
	\bar{\mathbb{U}}_{\leq20}(t,u)+\bar{\mathbb{Q}}_{\leq20}(t,u)+\bar{\mathbb{K}}_{\leq20}(t,u)\leq C\mathscr{D}_{21}+C\da(\bar{\mathbb{W}}_{\leq20}+\bar{\mathbb{W}}_{\leq20}^{\text{par}})+C\int_0^t\bar{\mathbb{W}}_{\leq20}.\label{topenergyK1final}
\end{align}
For any $\a$ with $Z^{\a}\in\mathcal{Z}^{ 20;1}$, it follows from \eqref{energywave} that 
\begin{align}\label{toporderK01}
	\begin{split}		
        &\sum_{|\a'|\leq|\a|}\da^{2m'}\left(E_0(Z^{\a'}\fai)(t,u)
		+F_0(Z^{\a'}\fai)(t,u)\right)\\
		\leq & C\sum_{|\a'|\leq|\a|}\da^{2m'}E_0(Z^{\a'}_i\fai)(0,u)+
		\sum_{|\a'|\leq|\a|}\da^{2m'}\int_{W_t^{u}}\mu Q_{0}^{|\a'|},
	\end{split}
\end{align}
where $m'\leq1$ is the number of $T$ in $Z^{\a'}$ and $\mu Q_1^{|\a|}:=(-\mathfrak{F}_{|\a|+1}K_0(Z^{\a}\fai)+\frac{1}{2}\mu T^{\a\be}(Z^{\a}\fai)\pi_{1,\a\be})$. The error integrals $\int_{W_t^{u}}\mu Q_{1}^{|\a|}$ contain the following two parts
\begin{itemize}
	\item[(1)] Contributions from $T^{\a\be}(Z^{\a}\fai) $ which can be bounded by high order energies and fluxes directly and have been treated in previous sections.
	\item[(2)] Major terms of the form
	\begin{align}\label{integralQ0a+2}
		\int_{W^{u}_t}-\mathfrak{F}_{|\a|+1}\cdot K_0 Z^{\a}\fai\ d\vartheta' du'dt',
	\end{align}
	where the contributions from the top order acoustical terms have been estimated in the last subsection, see\eqref{K0top1final} and \eqref{K0top2final}. The contributions from other terms in $\mathfrak{F}_{|\a|+1}$ have been given in Lemma \ref{inhomo1} and \ref{inhomo2}.
\end{itemize}
Collecting \eqref{K0top1final}, \eqref{K0top2final}, \eqref{topenergyK1final} and Lemmas \ref{inhomo1}, and \ref{inhomo2} yields
\begin{align}
	\begin{split}	&\mu_m^{2b_{20}}\sum_{|\a'|\leq|\a|}\da^{2m'}\left(E_0(Z^{\a'}\fai)(t,u)
	+F_0(Z^{\a'}\fai)(t,u)\right)\\
 \leq & C\mathscr{D}_{21}+\frac{8+C\da^{\frac{1}{2}}}{(2b_{20}-\tfrac{1}{2})(b_{20}+\tfrac{1}{2})}\bar{\mathbb{W}}_{\leq20}+\frac{1}{2}\bar{\mathbb{K}}_{\leq20}\\	&+C\int_0^t(\bar{\mathbb{W}}_{\leq20}+\bar{\mathbb{U}}_{\leq20})+C\da^{\frac{1}{2}}\int_0^t\mu_m^{-\frac{1}{2}}\mathbb{U}_{\leq20}+C\int_0^t\bar{\mathbb{V}}_{\leq21}+C\da^{-1}\int_0^u\bar{\mathbb{Q}}_{\leq20}.
	\end{split}\label{toporderK02}
\end{align}
The choice of $b_{20}$ gives $\frac{8+C\da^{\frac{1}{2}}}{(2b_{20}-\tfrac{1}{2})(b_{20}+\tfrac{1}{2})}<\frac{1}{2}$. Note that $\da^{\frac{1}{2}}\int_0^t\mu_m^{-\frac{1}{2}}\mathbb{U}_{\leq20}\leq C\da\bar{\mathbb{U}}_{\leq20}$ due to Lemma \ref{crucial}. Hence it follows from the top order energy estimates associated with $K_1$ \eqref{topenergyK1final} and for the vorticity \eqref{toporderK02} that
\begin{align}
	\bar{\mathbb{W}}_{\leq20}(t,u)+\bar{\mathbb{Q}}_{\leq20}(t,u)\leq C\mathscr{D}_{21}+C\int_0^t\bar{\mathbb{W}}_{\leq20}+C\da^{-1}\int_0^u\bar{\mathbb{Q}}_{\leq20}. \label{toporderK03}
\end{align}
Therefore, keeping only $\bar{\mathbb{W}}_{\leq20}$ on the left hand side of \eqref{toporderK03} and applying Gronwall inequality to it yield
\begin{align}
    \bar{\mathbb{W}}_{\leq20}(t,u)+\bar{\mathbb{Q}}_{\leq20}(t,u)\leq C\mathscr{D}_{21}+C\da^{-1}\int_0^u\bar{\mathbb{Q}}_{\leq20}. \label{toporderK04}
\end{align}
Similarly, applying Gronwall inequality to $\bar{\mathbb{Q}}_{\leq20}$ yields
\begin{align}
	\bar{\mathbb{W}}_{\leq20}(t,u)+\bar{\mathbb{Q}}_{\leq20}(t,u)\leq C\mathscr{D}_{21}. \label{toporderK0final}
\end{align}
Therefore, one obtains the following estimates for top order energies and fluxes:
\begin{align}
	\begin{split}
	\bar{\mathbb{W}}_{\leq20}(t,u)+\bar{\mathbb{U}}_{\leq20}(t,u)+\bar{\mathbb{Q}}_{\leq20}(t,u)+\bar{\mathbb{K}}_{\leq20}(t,u)&\leq C\mathscr{D}_{21},\\
	\bar{\mathbb{V}}_{\leq21}(t,u)&\leq C\da\mathscr{D}_{21}.
	\end{split}\label{topenergy}
\end{align}
\begin{remark}
	Applying the similar procedure to $\bar{\mathbb{W}}_{\leq20}^{\text{par}}$ and noticing Remark \ref{Wpar} lead to 
	\begin{align}
		\bar{\mathbb{W}}_{\leq20}^{\text{par}}\leq C\da\mathscr{D}_{21}.
	\end{align}
\end{remark}
\section{The decent scheme and lower-order energy hierarchy}\label{section7}
\hspace*{2pt} The top order energy estimates \eqref{topenergy} are equivalent to 
\begin{align}
	\begin{split}
		\mathbb{W}_{\leq20}(t,u)+\mathbb{U}_{\leq20}(t,u)+\mathbb{Q}_{\leq20}(t,u)+\mathbb{K}_{\leq20}(t,u)&\leq C\mu_m^{-10-\tfrac{3}{2}}\mathscr{D}_{21},\\
		\mathbb{V}_{\leq21}&\leq C\da\mu_m^{-11-\frac{3}{2}}\mathscr{D}_{21}.
	\end{split}
\end{align}
The above energy estimates are singular in $\mu$ so that it is not sufficient to close the bootstrap assumptions. The key point is that by lowering one order of derivatives in energy estimates, the power of $\mu$ decreases by one. After several steps in such a process, the power of $\mu$ is eliminated while the remaining order of derivatives in the energy estimates remains large enough to recover the bootstrap assumptions. 

We turn to the energy estimates for next-to-top order. Let $|\a|=19$. First, one has
\begin{align}
\begin{split}
	&\int_{W_t^{u}}T\fai\cdot Y^{\a-1}\hat{X}\tr\chi\cdot\dl Y^{\a}\fai\leq\int_0^t\|Y^{\a-1}\hat{X}\tr\chi\|_{L^2(\Sigma_{t'}^u)}\cdot\|\dl Y^{\a}\fai\|_{L^2(\Sigma_{t'}^u)}.
 \end{split}\label{ntopK01}
\end{align}
Since $|\a|=19$, one has \[\|\dl Y^{\a}\fai\|_{L^2(\Sigma_{t}^u)}\leq\mu_m^{-b_{19}}\sqrt{\bar{\mathbb{W}}_{\leq19}}.\] Note also that $Y^{\a-1}\hat{X}\tr\chi$ is not a top order acoustical term. Then, it follows from Lemma \ref{lotchiul2} that
\begin{equation*}
\begin{aligned}
	\|Y^{\a-1}\hat{X}\tr\chi\|_{\ltwou}\leq 
 &
	C\mathscr{D}_{21}+\frac{\tfrac{3}{2}+C\da}{b_{20}-\tfrac{1}{2}}\mu_m^{-b_{19}-\frac{1}{2}}\sqrt{\bar{\mathbb{U}}_{\leq20}}.
\end{aligned}    
\end{equation*}
Hence, by the top order energy estimates \eqref{topenergy} and Lemma \ref{crucial}, one obtains
\begin{align}
	\begin{split}
		&\int_0^t\|Y^{\a-1}\hat{X}\tr\chi\|_{L^2(\Sigma_{t'}^u)}\cdot\|\dl Y^{\a}\fai\|_{L^2(\Sigma_{t'}^u)}\\
  \leq & C\mathscr{D}_{21}+
		\frac{(\tfrac{3}{2}+C\da)\mu_m^{-2b_{20}+\tfrac{1}{2}}}{(b_{20}-\tfrac{1}{2})(2b_{19}-\tfrac{1}{2})}\sqrt{\bar{\mathbb{U}}_{\leq20}}\sqrt{\bar{\mathbb{W}}_{\leq19}}
  \leq 
		C\mu_m^{-2b_{19}}\mathscr{D}_{21}+\frac{1}{10}\mu_m^{-2b_{19}}\bar{\mathbb{W}}_{\leq19}.
	\end{split}\label{ntopK01final}
\end{align}
Next, one has
\begin{align*}
	&\da^2\int_{W_t^u}T\fai\cdot Y^{\a-1}T\tr\chi\cdot \dl Y^{\a-1}T\fai\leq\int_0^t\da\|Y^{\a-1}T\tr\chi\|_{L^2(\Sigma_{t'}^u)}\cdot\da
 \|\dl Y^{\a-1}T\fai\|_{L^2(\Sigma_{t'}^u)}.
\end{align*}
Similarly, $\da\|\dl Y^{\a-1}T\fai\|_{\ltwou}\leq \sqrt{\mathbb{W}_{\leq19}}$. Since $Y^{\a-1}T\tr\chi$ is a lower order acoustical term, it follows from Lemma \ref{lotchiul2} and \eqref{equationtchi} that
\begin{align*}
	\da\|Y^{\a-1}T\tr\chi\|_{\ltwou}\leq C\da\mathscr{D}_{21}+C\da\int_0^t\sqrt{\mathbb{W}_{\leq_{20}}}+\mu_m^{-\frac{1}{2}}\sqrt{\mathbb{U}_{\leq20}}.
\end{align*}
Hence, it holds that
\begin{align}
	\begin{split}
		&\int_0^t\da\|Y^{\a-1}T\tr\chi\|_{L^2(\Sigma_{t'}^u)}\cdot\da\|\dl Y^{\a-1}T\fai\|_{L^2(\Sigma_{t'}^u)}\\
  \leq & C\da\mathscr{D}_{21}+C\da\mu_m^{-2b_{19}}(\bar{\mathbb{W}}_{\leq20}+\bar{\mathbb{W}}_{\leq20})+C\da\mu_m^{-2b_{19}}\bar{\mathbb{W}}_{\leq19}\\
		\leq &C\da\mu_m^{-2b_{19}}\mathscr{D}_{21}+C\da\mu_m^{-2b_{19}}\bar{\mathbb{W}}_{\leq19}.
	\end{split}\label{ntopK02final}
\end{align}
Next, one has
\begin{align}
	\begin{split}
\int_{W_t^{u}}T\fai\cdot Y^{\a-1}\hat{X}\tr\chi\cdot LY^{\a}\fai&\leq\left(\int_0^t\|Y^{\a-1}\hat{X}\tr\chi\|^2_{L^2(\Sigma_{t'}^u)}\right)^{\frac{1}{2}}\left(\int_0^u\|L Y^{\a}\fai\|^2_{L^2(C_{u'}^t)}\right)^{\frac{1}{2}}\\
&\leq C\da^{\frac{1}{2}}\mathscr{D}_{21}+C\da^{\frac{1}{2}}\int_0^t\mu_m^{-2b_{20}+1}\bar{\mathbb{U}}_{\leq20}+C\da^{-\frac{1}{2}}\int_0^u\mathbb{Q}_{\leq19}\\
&\leq C\da^{\frac{1}{2}}\mu_m^{-2b_{19}}\mathscr{D}_{21}+C\da^{-\frac{1}{2}}\mu_m^{-2b_{19}}\int_0^u\bar{\mathbb{Q}}_{\leq19}.
\end{split}\label{ntopK11final}
\end{align}
where Lemma \ref{lotchiul2} and \eqref{topenergy} have been used. Finally, it holds that
\begin{align}
	\begin{split}
	\da^2\int_{W_t^u}T\fai\cdot Y^{\a-1}T\tr\chi\cdot LY^{\a}\fai &\leq \da\left(\int_0^t\|Y^{\a-1}T\tr\chi\|^2_{L^2(\Sigma_{t'}^u)}\right)^{\frac{1}{2}}\left(\int_0^u\da^2\|L Y^{\a-1}T\fai\|^2_{L^2(C_{u'}^t)}\right)^{\frac{1}{2}}\\
 &\leq C\da\mathscr{D}_{21}+C\da(\bar{\mathbb{W}}_{\leq20}+\bar{\mathbb{U}}_{\leq20})\int_0^t\mu_m^{-2b_{20}+1}\ dt'+C\int_0^u\mathbb{Q}_{\leq19}\\
 &\leq C\da\mu_m^{-2b_{19}}\mathscr{D}_{21}+C\mu_m^{-2b_{19}}\int_0^u\bar{\mathbb{Q}}_{\leq19}.
	\end{split}\label{ntopK12final}
\end{align}
where Lemma \ref{lotchiul2} and \eqref{topenergy} have been used. Then, we turn to next-to-top order estimates for the vorticity. To this end, note that
\begin{align}
\begin{split}
	&\int_{W_t^u}\hat{X}(\zeta-\p)\cdot Y^{\a+1}(\zeta-\p)\cdot Y^{\a-1}\slashed{\Delta}\mu\\
	\leq&\da\int_{W_t^u}Y^{\a+1}(\zeta-\p)\cdot Y^{\a+1}\mu\leq\da\int_0^t||Y^{\a+1}(\zeta-\p)||_{L^2(\Sigma_{t'}^u)}||Y^{\a+1}\mu||_{L^2(\Sigma_{t'}^u)}\\
	\leq &C\da\mathscr{D}_{21}+C\da\int_0^t\mu_m^{-2b_{20}+\frac{1}{2}}\sqrt{\bar{\mathbb{V}}_{\leq20}}\left(\sqrt{\bar{\mathbb{W}}_{\leq20}}+\sqrt{\bar{\mathbb{U}}_{\leq20}}\right)\ dt'\\
 \leq &C\da\mu_m^{-b_{20}}\mathscr{D}_{21}+C\da\mu_m^{-2b_{20}}\bar{\mathbb{V}}_{\leq20}(t,u).
	\end{split}\label{ntopzetafinal}
\end{align}
where the Lemmas \ref{lotchiul2} and \ref{crucial} has been used.
\subsection{Next-to-top order energy estimates}
\hspace*{2pt} We first derive the next-to-top order energy estimates for the vorticity. Let $|\a|=19$ and $P_i^{\a+1}\in\mathcal{P}^{20;3}$. It follows from \eqref{toporderenergyzeta} that 
\begin{align}
	\begin{split}
	&\sum_{|\a'|\leq|\a|}\da^{2(p'-1)_+}\left(E^{\zeta}(P_i^{\a'+1}(\zeta-\varrho))+F^{\zeta}(P_i^{\a'+1}(\zeta-\varrho))\right)(t,u)\\
	\leq&\sum_{|\a'|\leq|\a|}\da^{2(p-1)_+}E^{\zeta}(P_i^{\a'+1}(\zeta-\varrho))(0,u)+\sum_{|\a'|\leq|\a|}\int_{W_t^u}\da^{2(p-1)_+}\mu Q_{|\a'|+1}^{\zeta}\\
	\leq&C\da\mathscr{D}_{21}+C\da\mu_m^{-2b_{20}}\bar{\mathbb{V}}_{\leq20}(t,u)+C(1+\da^{-\frac{1}{2}})\int_0^u\mathbb{V}_{\leq20}\ du'+C\da\mu_m^{-2b_{20}}\bar{\mathbb{W}}_{\leq20}(t,u)\\
	&+C\da\int_0^t\mu_m^{-\frac{1}{2}}\mathbb{U}_{\leq19}\ dt'+C\da\mu_m^{-2b_{19}}\bar{\mathbb{U}}_{\leq19}(t,u)+C\da\mu_m^{-2b_{19}}\int_0^u\bar{\mathbb{Q}}_{\leq19}\ du'
 \end{split}\label{ntopenergyzeta1}
\end{align}
where $p'$ is the number of $L$'s in $P_i^{\a'+1}$, and \eqref{commutationvorticity} and \eqref{ntopzetafinal} have been used. Since the left hand side of \eqref{ntopenergyzeta1} is increasing with respect to $t$, with the aid of \eqref{topenergy}, one obtains that for sufficiently small $\da$,
\begin{align}
	\begin{split}
	\bar{\mathbb{V}}_{\leq20}(t,u)\leq &C\da\mathscr{D}_{21}+C\da^{-\frac{1}{2}}\int_0^u\bar{\mathbb{V}}_{\leq20}\ du'+C\da(\bar{\mathbb{W}}_{\leq19}+\bar{\mathbb{U}}_{\leq19})(t,u)+C\da\int_0^u\bar{\mathbb{Q}}_{\leq19} \ du'.
	\end{split}\label{ntopzeta2}
\end{align}
Applying Gronwall inequality to \eqref{ntopzeta2} yields
\begin{align}
	\bar{\mathbb{V}}_{\leq20}(t,u)\leq C\da\mathscr{D}_{21}+C\da(\bar{\mathbb{U}}_{\leq19}+\bar{\mathbb{W}}_{\leq19})(t,u)+C\da\int_0^u\bar{\mathbb{Q}}_{\leq19}\ du'.\label{ntopenergyzeta}
\end{align}
We then turn to the next-to-top order energy estimates associated with $K_0$ and $K_1$. Let $|\a|=19$ and $Z_i^{\a}\in\mathcal{Z}^{19;1}$. Similar to \eqref{toporderK11}, it follows from \eqref{energytopK11}, \eqref{ntopK11final} and \eqref{ntopK12final} that
\begin{align}
	\begin{split}
		&\sum_{|\a'|\leq|\a|}\da^{2m'}(E_1(Z_i^{\a}\fai)+F_1(Z_i^{\a}\fai)+K(Z_i^{\a}\fai))(t,u)\\		\leq&\sum_{|\a'|\leq|\a|}\da^{2m'}E_1(Z_i^{\a}\fai)(0,u)+C\da\mu_m^{-2b_{19}}\mathscr{D}_{21}\\
   &+C\da\mu_m^{-2b_{19}}\bar{\mathbb{U}}_{\leq19}(t,u)+C\da\mu_m^{-2b_{19}}\bar{\mathbb{W}}_{\leq19}(t,u)
  +C\da^{\frac{1}{2}}\mathbb{K}_{\leq19}(t,u)\\
  &+C\da^{-\frac{1}{2}}\int_0^u\mathbb{Q}_{\leq 19}\ du'+C\int_0^t\mathbb{V}_{\leq20}\ dt'+C\da^{-\frac{1}{2}}\int_0^u\mathbb{V}_{\leq20}.
	\end{split}\label{ntopenergyK1}
\end{align}
Since the left hand side of \eqref{ntopenergyK1} is increasing with respect to $t$, combining \eqref{ntopenergyzeta} and \eqref{ntopenergyK1} yields that for sufficiently $\da$,
\begin{align}
	\begin{split}
	&\bar{\mathbb{U}}_{\leq19}(t,u)+\bar{\mathbb{Q}}_{\leq19}(t,u)+\bar{\mathbb{K}}_{\leq19}(t,u)\\
 \leq & C\mathscr{D}_{21}	+C\da\bar{\mathbb{W}}_{\leq19}+C\da\int_0^t\bar{\mathbb{U}}_{\leq19}+\bar{\mathbb{W}}_{\leq19}\ dt'+C\da^{-\frac{1}{2}}\int_0^u\bar{\mathbb{Q}}_{\leq 19}\ du'.
	\end{split}\label{ntopenergyK11}
\end{align}
Applying Gronwall inequality to \eqref{ntopenergyK11} yields
\begin{align}
	\begin{split}
	&\bar{\mathbb{U}}_{\leq19}(t,u)+\bar{\mathbb{Q}}_{\leq19}(t,u)+\bar{\mathbb{K}}_{\leq19}(t,u)
 \leq  C\mathscr{D}_{21}
+C\da\bar{\mathbb{W}}_{\leq19}+C\da\int_0^t\bar{\mathbb{W}}_{\leq19}\ dt'.
	\end{split}\label{ntopenergyK1final}
\end{align}
Similar to \eqref{toporderK01}, it follows from \eqref{ntopK01final} and \eqref{ntopK02final} that
\begin{align}
	\begin{split}
    &\sum_{|\a'|\leq|\a|}\da^{2m'}(E_0(Z_i^{\a}\fai)+F_0(Z_i^{\a}\fai))(t,u)\\
    \leq&\sum_{|\a'|\leq|\a|}\da^{2m'}E_0(Z_i^{\a}\fai)(0,u)+C\da\mu_m^{-2b_{19}}\mathscr{D}_{21}\\
    &+(\tfrac{1}{10}+C\da)\mu_m^{-2b_{19}}\bar{\mathbb{W}}_{\leq19}(t,u)+C\mathbb{K}_{\leq19}(t,u)
	+C\int_0^t\mathbb{W}_{\leq19}+\mathbb{U}_{\leq19}\ dt'\\
 &+C\da^{\frac{1}{2}}\int_0^t\mu_m^{-\frac{1}{2}}\mathbb{U}_{\leq19}\ dt'+C\da^{-\frac{1}{2}}\int_0^u\mathbb{Q}_{\leq 19}\ du'+C\int_0^t\mathbb{V}_{\leq20}\ dt'.
	\end{split}\label{ntopenergyK01}
\end{align}
In view of \eqref{ntopenergyzeta} and \eqref{ntopenergyK1final}, \eqref{ntopenergyK01} implies that for sufficiently small $\da$,
\begin{align*}
	\begin{split}
		\bar{\mathbb{W}}_{\leq19}(t,u)+\bar{\mathbb{Q}}_{\leq19}(t,u)
		\leq &C\mathscr{D}_{21}+ C\int_0^t\bar{\mathbb{W}}_{\leq19}\ dt'+C\da^{-\frac{1}{2}}\int_0^u\bar{\mathbb{Q}}_{\leq19}.
	\end{split}
\end{align*}
Applying Gronwall inequality yields that
\begin{align*}
\bar{\mathbb{W}}_{\leq19}(t,u)+\bar{\mathbb{Q}}_{\leq19}(t,u)&\leq C\mathscr{D}_{21}.
\end{align*}
In conclusion, one obtains the following next-to-top order energy estimates: 
\begin{align}
	\begin{split}
		\mathbb{W}_{\leq19}(t,u)+\mathbb{U}_{\leq19}(t,u)+\mathbb{Q}_{\leq19}(t,u)+\mathbb{K}_{\leq19}(t,u)&\leq C\mu_m^{-8-\tfrac{3}{2}}\mathscr{D}_{21},\\
		\mathbb{V}_{\leq20}(t,u)&\leq C\mu_m^{-10-\frac{3}{2}}\mathscr{D}_{21}.
	\end{split}\label{ntopenergy}
\end{align}
Up to now, the power of $\mu_m$ in energy estimates has been reduced. In proving \eqref{ntopenergy}, it is noted that to obtain the energy estimates for $N\leq18$, it suffices to bound the following integrals:
\begin{align*}
	\int_0^t\mu_m^{-b_{N}-\frac{1}{2}}\ dt',\quad \int_0^t\mu_m^{-b_N}\ dt',\quad \int_0^t\mu_m^{-b_{N+1}}\sqrt{\bar{\mathbb{V}}_{\leq N+1}}\ dt'.
\end{align*}
By the choice of $b_N$, for $15\leq N\leq19$, it follows from Lemma \ref{crucial} that 
\[
\int\mu_m^{-b_N-\frac{1}{2}}\ dt'\leq \mu_m^{-b_{N}}\quad \text{and}\quad  \int_0^t\mu_m^{-b_{N+1}}\sqrt{\bar{\mathbb{V}}_{\leq N+1}}\ dt'\leq \mu_m^{-b_N}\sqrt{\bar{\mathbb{V}}_{\leq N+1}}.
\]
Therefore, one has
\begin{align*}
	\mathbb{W}_{\leq 15}(t,u)+\mathbb{Q}_{\leq15}(t,u)+\mathbb{U}_{\leq15}(t,u)+\mathbb{K}_{\leq15}(t,u)+&\mathbb{V}_{\leq 15}(t,u)\leq C\mu_m^{-2b_{15}}\mathscr{D}_{21}.
\end{align*}
 Since $b_{15}=\frac{3}{4}>0$ which means the power of $\mu_m$ is energies has not been fully eliminated, we have to consider the next step. For $N=14$, it follows from Lemma \ref{crucial} that 
\begin{align*}
	\int_0^t\mu_m^{-b_N-\frac{1}{2}}\ dt'\leq C\quad\mathrm{and}\quad \int_0^t\mu_m^{-b_{N+1}}\sqrt{\bar{\mathbb{V}}_{\leq N+1}}\ dt'\leq \sqrt{\bar{\mathbb{V}}_{\leq N+1}}.
\end{align*}
Hence, for $N=14$, one obtains
\begin{align}
	\mathbb{W}_{\leq 14}(t,u)+\mathbb{Q}_{\leq14}(t,u)+\mathbb{U}_{\leq14}(t,u)+\mathbb{K}_{\leq14}(t,u)+&\mathbb{V}_{\leq 14}(t,u)\leq C\mathscr{D}_{21}.\label{14energy}
\end{align}
This is the desired estimate since the power of $\mu_m$ has been eliminated. In conclusion, one has  for $15\leq N\leq20$, 
\begin{align}
\mathbb{W}_{\leq N}(t,u)+\mathbb{Q}_{\leq N}(t,u)+\mathbb{U}_{\leq N}(t,u)+\mathbb{K}_{\leq N}(t,u)+&\mathbb{V}_{\leq N}(t,u)\leq C\mu_m^{-2(N-15)-\frac{3}{2}}\mathscr{D}_{21},\\
&\mathbb{V}_{\leq21}(t,u)\leq C\mu_m^{-11-\frac{3}{2}}\mathscr{D}_{21}, 
\end{align}
\section{Recover the bootstrap assumptions and formation of shock}\label{section8}
In this section, we recover the bootstrap assumptions via Sobolev inequality and finish the proof of the main theorem \ref{main}.
\begin{lemma}\label{sobolev}
	There is a universal constant $C$ such that for any $\fai$ defined on $S_{t,u}$, 
	\begin{align*}
		\|\fai\|_{L^{\infty}(S_{t,u})}\leq C\cdot S_{[1]}^{\frac{1}{2}}:=C\cdot \left(\int_{S_{t,u}}|\fai|^2+|Y\fai|^2\right)^{\frac{1}{2}}.
	\end{align*}
\end{lemma}
\begin{proof}
	The standard Sobolev inequality shows for any $\vartheta\in \mathbb{S}^1$
	\begin{align*}
		|\fai(\vartheta)|\leq C\left(\int_{\vartheta'\in\mathbb{S}^1}\fai^2+|X\fai|^2\ d\vartheta'\right)^{\frac{1}{2}}.
	\end{align*}
	It follows from Lemma \ref{LiTiXi} and the fact that $\sg=1+O(\da M)$ that 
	\begin{align*}
		(1-O(\da))\int_{S_{t,u}}f^2\leq \int_{\vartheta'\in\mathbb{S}^1}f^2\ d\vartheta'\leq (1+O(\da))\int_{S_{t,u}}f^2.
	\end{align*}
	Therefore, in view of $X=\frac{\sg^{\frac{1}{2}}}{\hat{X}^2}Y$,
	\begin{align*}
		\|\fai\|_{L^{\infty}(S_{t,u})}\leq C\|\fai\|_{L^{\infty}(\mathbb{S}^1)}\leq C S_{[1]}^{\frac{1}{2}}.
	\end{align*}
 This finishes the proof of the lemma.
\end{proof}
\begin{proof}[Proof of Theorem \ref{main}]
 Define 
\begin{align*}
	S_{n}(t,u)=\int_{S_{t,u}}\sum_{|\a|\leq n}|\da^m Z_i^{\a}\fai|^2+\sum_{|\a|\leq n}|\da^{(p-1)_+}P_i^{\a}(\zeta-\p)|^2,
\end{align*}
where $m$ is the number of $T$ in $Z_i^{\a}$ and $p$ is the number of $L$ in $P_i^{\a}$, respectively. It follows from Lemma \ref{faiL2} and \eqref{14energy} that
\begin{align*}
	S_{14}\leq C\tilde{\da}(\mathbb{W}_{\leq14}+\mathbb{U}_{\leq14}+\mathbb{V}_{\leq14})\leq C\tilde{\da}\mathscr{D}_{21}.
\end{align*}
Here, the constant $C$ is a universal constant. Then, it follows from Lemma \ref{sobolev} that for any $Z_i^{\a}\in \mathcal{Z}^{13;\leq 1}$, $P_i^{\a}\in \mathcal{P}^{13;\leq 3}$ and all $(t,u)\in [0,t^{\ast})\times[0,\tilde{\da}]$
\begin{align}
	\da^m|Z_i^{\a}\fai|+\da^{(p-1)_+}|P_i^{\a}(\zeta-\p)|\leq C\da\sqrt{\mathscr{D}_{21}}.\label{recover}
\end{align}
Hence, by choosing $M\geq C\sqrt{\mathscr{D}_{21}}$, one recover the bootstrap assumptions \eqref{bs}.

Recall the definitions,
\begin{equation*}
	\begin{split}
	&s_{\ast}=\sup\{t\ |\ t\geq0\ \text{and}\ \mu_m^{\tilde{\da}}(t)>0\},\\
	&t_{\ast}=\sup\{\tau\ |\ \tau\geq0\ \text{such that smooth solution exists for all}\ (t,u)\in[0,\tau)\times[0,\tilde{\da}]\ \text{and}\ \vartheta\in \mathbb{S}^1\},\\
	&s^{\ast}=\min\{s_{\ast},1\}, \quad t^{\ast}=\min\{t_{\ast},s^{\ast}\}.
\end{split}
\end{equation*}

 To finish the proof, it suffices to show $t^{\ast}=s^{\ast}$, i.e. either $t^{\ast}=1$, then the smooth solution exists on $[0,1]\times[0,\tilde{\da}]\times \mathbb{S}^1$, or $t^{\ast}<1$, then at least at one point on $\Sigma_{t^{\ast}}^{\tilde{\da}}$ such that $\mu_m(t^{\ast})=0$, and a shock forms in finite time.
 
 If $t^{\ast}< s^{\ast}$, then \textbf{$\mu_m$ is positive on $\Sigma_{t^{\ast}}^{\tilde{\da}}$}. Hence, the \textbf{Jacobian $\de$ of the transformation between the acoustical coordinates and the Cartesian coordinates never vanishes on $\Sigma_{t^{\ast}}^{\tilde{\da}}$}, i.e. the transformation between two coordinates is regular on $\Sigma_{t^{\ast}}^{\tilde{\da}}$. Moreover, in the acoustical coordinates, $\fai$ and its derivatives are regular on $\Sigma_{t^{\ast}}^{\tilde{\da}}$ due to the bootstrap assumptions. 
Therefore, in the Cartesian coordinates, $\fai$ and its derivatives are regular on $\Sigma_{t^{\ast}}^{\tilde{\da}}$ which belongs to the Sobolev space $H^3$. By the standard local well-posedness theory, one can obtain an extension of the solution to some $t_1>t^{\ast}$. This is a contradiction. Hence, a shock must form.
\end{proof}

 \medskip
 
{\bf Acknowledgements:}
This work is partially supported by National Key R\&D Program of China 2024YFA1013302.
Chen was supported by Shanghai Frontiers Science Center of Modern Analysis.
The research of Xie was partially supported by NSFC grants 12250710674 and 12426203, 
Program of Shanghai Academic Research Leader 22XD1421400.

	\bibliographystyle{plain}

\end{document}